\documentclass[11pt]{amsart}
\usepackage[margin=1.2 in]{geometry}

\usepackage{mathrsfs}
\usepackage{amsfonts}
\usepackage{latexsym,epsfig}
\usepackage{mathabx}
\usepackage{amsmath, amscd, amsthm,amssymb}
\usepackage[pdftex]{color}
\usepackage{comment}
\usepackage{marginnote}
\usepackage[colorlinks,citecolor=blue,linkcolor=red]{hyperref}
\usepackage[nameinlink]{cleveref}
\usepackage{enumitem}

\newcommand{\RR}[0]{\mathbb{R}}

\newcommand{\Z}{\mathbb{Z}}

\newcommand{\CC}[0]{\mathcal{C}}

\newcommand{\ZZ}[0]{\mathbb{Z}}

\newcommand{\id}[0]{\text{id}}
\newcommand{\im}{\mathrm{Im}}

\newcommand{\R}{\mathbb{R}}

\newcommand{\wt}{\widetilde}
\newcommand{\mc}{\mathcal}

\newcommand{\cone}{\mathrm{cone}}
\newcommand{\ab}{\mathcal{AB}}

\newcommand{\uM}{\wt M}
\newcommand{\aM}{\wt M^{\text{ab}}} 
\newcommand{\mr}{\mathring}

\newcommand{\hbs}{\tau^{(2)}}
\newcommand{\bs}{\boldsymbol}

\newcommand{\T}{\mathcal{T}}

\newcommand{\C}{\mathcal{C}}

\newcommand{\define}[1]{\textbf{#1}}

\newcommand\tsim{\kern-.4em\sim}

\newcommand\ssm{\smallsetminus}

\renewcommand{\int}{\mathrm{int}}
\renewcommand{\hom}{\mathrm{Hom}}
\newcommand{\e}{\mathcal{E}}

\newcommand{\ind}{\mathrm{index}}
\newcommand{\algcusps}{\mathrm{algcusps}}

\newcommand{\inv}{\mathrm{inv}}

\renewcommand{\phi}{\varphi}

\DeclareMathOperator{\supp}{supp}

\DeclareMathOperator{\closure}{cl}
\DeclareMathOperator{\intr}{int}

\usepackage{hyperref}

\newtheorem{thm}{Theorem}

\newtheorem{theorem}{Theorem}[section]
\newtheorem{lemma}[theorem]{Lemma}
\newtheorem{proposition}[theorem]{Proposition}
\newtheorem{corollary}[theorem]{Corollary}

\newtheorem{fact}{Fact}
\newtheorem{claim}{Claim}

\theoremstyle{definition}

\newtheorem{remark}[theorem]{Remark}

\begin{document}
\title{A polynomial invariant for veering triangulations}
\author[M. Landry]{Michael Landry}
\address{Department of Mathematics\\
Washington University in Saint Louis }
\email{\href{mailto:mlandry@wustl.edu}{mlandry@wustl.edu}}
\author[Y.N. Minsky]{Yair N. Minsky}
\address{Department of Mathematics\\ 
Yale University}
\email{\href{mailto:yair.minsky@yale.edu}{yair.minsky@yale.edu}}
\author[S.J. Taylor]{Samuel J. Taylor}
\address{Department of Mathematics\\ 
Temple University}
\email{\href{mailto:samuel.taylor@temple.edu}{samuel.taylor@temple.edu}}
\date{\today}
\thanks{This work was partially supported by NSF grants DMS-1610827, DMS-1744551, and the Sloan Foundation.}

\begin{abstract}
We introduce a polynomial invariant $V_\tau \in \Z[H_1(M)/\text{torsion}]$ associated to a 
veering triangulation $\tau$ of a $3$-manifold $M$. In the special case where the triangulation is layered, i.e. comes from a fibration, $V_\tau$ recovers the Teichm\"uller polynomial of the fibered faces canonically associated to $\tau$. Via Dehn filling, this gives a combinatorial description
of the Teichm\"uller polynomial for any hyperbolic fibered $3$-manifold. 

For a general veering triangulation
$\tau$, we show that the surfaces carried by $\tau$ determine a cone in homology 
that is dual to its cone of positive closed transversals. 
Moreover, we prove that this is \emph{equal} to the cone over a (generally non-fibered)
face of the Thurston norm ball, and that $\tau$ computes the norm on this cone in a precise sense. 
We also give a combinatorial description of $V_\tau$ in terms of the \emph{flow graph} for $\tau$ and its Perron polynomial. This perspective allows us to characterize when a veering triangulation 
 comes from a fibration, and more generally to compute the face of the Thurston norm determined by $\tau$. 
\end{abstract}

\maketitle

\setcounter{tocdepth}{1}
\tableofcontents

\section{Introduction}
Veering triangulations of cusped hyperbolic 3-manifolds were introduced by Agol as a means
to canonically triangulate certain pseudo-Anosov mapping tori \cite{agol2011ideal}. In
particular, Agol showed that for any pseudo-Anosov homeomorphism $f$ on a surface $S$, the
surface $S$ may be punctured along the (invariant) singularities of $f$ so that the
resulting mapping torus $M$ of the punctured surface admits a \emph{layered ideal
  triangulation} $\tau$, called the \define{veering triangulation}. This triangulation is
layered in the sense that it is built by layering tetrahedra on an ideal triangulation of
the punctured surface. The triangulation $\tau$ of $M$ is in fact not only an invariant of
the monodromy $f$, but it is also an invariant of the fibered face of the Thurston norm ball that it determines. 

Among all layered triangulations of hyperbolic 3-manifolds, Agol characterized the veering triangulation $\tau$ in terms of a combinatorial condition which can be interpreted as a bicoloring of the edges of $\tau$. Answering a question of Agol, Hodgson-Rubinstein-Segerman-Tillmann \cite{hodgson2011veering} showed that this combinatorial condition can be satisfied by \emph{nonlayered} triangulations and the resulting class of ideal triangulations is now referred to as \define{veering}. Veering triangulations have since emerged as an important object in several areas of low-dimension geometry and have attracted much attention. In particular, they have connections to hyperbolic geometry \cite{hodgson2011veering, futer2013explicit, gueritaud} (although they are usually not geometric \cite{futer2018random}), are used to study algorithmic problems in the mapping class group \cite{bell2014pseudo}, encode the hierarchy of subsurface projections associated to their monodromies \cite{minsky2017fibered}, and, most important for this paper, directly correspond to certain flows on the manifold $M$ \cite{landry2018taut, landry2019stable, schleimer2019veering}.

In this paper, we introduce a polynomial invariant of veering triangulations that generalizes McMullen's Teichm\"uller polynomial from the fibered setting \cite{mcmullen2000polynomial}. 
We remark that although veering triangulations occur only on manifolds with cusps, our construction recovers the Teichm\"uller polynomial in general via Dehn filling.
Before giving the details, let us informally summarize what we see as the central points of our construction. 

$\bullet$ First, in that it recovers the Teichm\"uller polynomial when the veering triangulation is layered, our construction gives a general, straightforward procedure 
to compute the Teichm\"uller polynomial directly from the data of the veering triangulation.
Moreover, this data is readily available through either Mark Bell's program \texttt{flipper} \cite{flipper}, which computes the veering triangulation in the fibered setting, or the veering census of Giannopolous, Schleimer, and Segerman \cite{VeeringCensus}. In fact such an algorithm has already been devised by Parlak \cite{Parlak1}, and implemented by Parlak, Schleimer, and Segerman (see \cite{VeeringCensus}). 

$\bullet$ Second, we show that any veering triangulation $\tau$ is naturally associated to a face $\bf F_\tau$ of the Thurston norm ball of $M$. Indeed, an integral homology class is contained in the cone over $\bf F_\tau$ if and only if it is represented by a surface carried by the underlying branched surface of $\tau$. 
Moreover, the face $\bf F_\tau$ is
fibered \emph{exactly} when $\tau$ is layered; so in particular the surfaces carried by a nonlayered veering triangulation determine a non-fibered face of the Thurston norm ball.
In this sense, our polynomial invariant
provides a generalization of the Teichmuller polynomial as requested by both McMullen \cite{mcmullen2000polynomial} and Calegari \cite[Question 3.2]{calegari2002problems}. 
In general, the polynomial is a quotient of the Perron polynomial 
of
a certain directed graph (the \define{flow graph} $\Phi_\tau$) associated to $\tau$. As in Fried's theory of homology directions \cite{fried1982geometry}, the cone of directed cycles of this graph is dual to $\bf F_\tau$, and 
therefore
$\bf F_\tau$ can also be directly computed from the Perron polynomial of $\Phi_\tau$.

$\bullet$ Third is the connection to certain flows, which we investigate in a sequel paper \cite{veeringpoly2}. 
 Following unpublished work of Agol-Gu\'eritaud and work-in-progress of Schleimer-Segerman (\cite{schleimer2019veering} and a forthcoming sequel), veering triangulations of hyperbolic manifolds precisely correspond to 
 \emph{pseudo-Anosov flows without perfect fits} as introduced and studied by Fenley \cite[Definition 4.2]{fenley1999foliations}.
These form an important class of flows on $3$-manifolds that generalize the suspension flow of a pseudo-Anosov homeomorphism on its mapping torus. When our veering triangulation comes from such a flow $\phi$, we will show that
 \begin{itemize}
 \item the (combinatorially defined) flow graph $\Phi_\tau$ codes $\phi$'s orbits, 
 in a manner similar to a Markov partition for $\varphi$, and 
 \item the Perron polynomial of $\Phi_\tau$ packages growth rates of closed orbits of $\phi$, after cutting $M$ along certain surfaces transverse to the flow. 
 \end{itemize}

We next turn to giving a more formal explanation of our results. 

\subsection{The veering and taut polynomials}
Let $M$ be a $3$-manifold with veering triangulation $\tau$, and let $G = H_1(M)/\text{torsion}$. In \Cref{sec:veering}, we define the \define{veering polynomial} $V_\tau$ which is an invariant of $\tau$ contained in the group ring $\Z[G]$.
Essentially by its construction, $V_\tau$ comes with a canonical factor $\Theta_\tau$,
defined up to multiplication by a unit $\pm g\in \Z[G]$, which we call the \define{taut
  polynomial}. These polynomials are invariants of modules 
defined by relations among the edges of
the veering triangulation on the universal free abelian cover of $M$, which are determined
by the tetrahedra and faces of $\tau$, respectively. 
Informally, the tetrahedron relations for the veering polynomial 
impose conditions modeled on a train track fold,
while the face relations for the taut polynomial impose conditions modeled on the switch
conditions of a train track.

One main result, which in particular is needed for the explicit connection to the Teichm\"uller polynomial in the fibered setting (see \Cref{th:teich_intro}), is the precise relation between these polynomials. For its statement, we note that there is a canonical collection of directed cycles $c_1, \ldots, c_n$ in $M$, which we call \emph{$AB$-cycles}, whose homology classes are denoted by $g_i = [c_i] \in H_1(M)$. See \Cref{sec:intro_faces} for more on their significance.

\begin{thm}[Factorization]
\label{th:factor_intro}
Suppose $\mathrm{rank}(H_1(M)) > 1$. Then up to multiplication by a unit $\pm g\in \Z[G]$,
\[
V_\tau = \Theta_\tau \cdot \prod_{i=1}^n(1 \pm g_i).
\]
\end{thm}
We note the possibility that some $AB$-cycles may be trivial in $H_1(M)$ and refer the
reader to \Cref{th:factorization!} for a more precise formulation, including an explicit
recipe for the signs in the factorization formula. 

Since $V_\tau$ is always computed as a determinant of a square matrix, we will see that \Cref{th:factor_intro} can be interpreted as an extension of McMullen's determinant formula for the Teichm\"uller polynomial (see \Cref{sec:layered}).

\subsection{Connection to the Teichm\"uller polynomial}
Next suppose that $\tau$ is layered. In this case, it corresponds to a fibered face $\bf F = \bf F_\tau$ of the Thurston norm ball and has an associate Teichm\"uller polynomial $\Theta_{\bf F} \in \Z[G]$, defined up to a unit. In \Cref{sec:layered}, we prove

\begin{thm}[Teichm\"uller $=$ taut] \label{th:teich_intro}
The Teichm\"uller polynomial $\Theta_{\bf F}$ agrees with the taut polynomial $\Theta_\tau$: 
\[
\Theta_{\bf F} = \Theta_\tau,
\]
up to a unit $\pm g\in \Z[G]$.

Moreover, if $N$ is \emph{any} hyperbolic $3$-manifold with fibered face ${\bf F}_N$ and $M$ is obtained by puncturing $N$ along the singular orbits of its suspension flow, then
\[
\Theta_{{\bf F}_N} = i_*(\Theta_\tau).
\]
where $i_*$ is induced by the inclusion $i \colon M \to N$ and $\tau$ is the veering triangulation associated to $M$.
\end{thm}

The above statement combines \Cref{th:modules_agree} and \Cref{prop:Teich_punctured}. Note that \Cref{th:teich_intro}, together with \Cref{th:factor_intro}, gives a way to compute the Teichm\"uller polynomial for any fibered hyperbolic $3$-manifold by first puncturing along singular orbits. In the layered (i.e. fibered) setting, no $AB$-cycle can be trivial by \Cref{th:faces_intro} below.

In fact, using our work Parlak \cite{Parlak1} describes an algorithm to compute the veering and taut polynomials given a (possibly nonlayered) 
veering triangulation. In separate work \cite{Parlak2}, she also relates the taut polynomial defined here to the Alexander polynomial, 
thereby generalizing a result of McMullen for the Teichm\"uller polynomial \cite{mcmullen2000polynomial} via \Cref{th:teich_intro}.

We conclude by noting that 
there has been recent interest in developing
algorithms to compute the Teichm\"uller polynomial 
and several special cases were previously considered
in \cite{lanneau2017computing,baik2020algorithm, billet2019teichmuller}.

\subsection{Faces of the Thurston norm ball} \label{sec:intro_faces}
Moving beyond the fibered setting, we obtain results for
general veering triangulations, which we show determine (generally non-fibered) faces of the Thurston norm ball. 

For this, we first describe an alternative construction of $V_\tau$. In
\Cref{sec:flowgraph}, we define a graph $\Phi = \Phi_\tau$ in $M$ associated to $\tau$, which we call the \define{flow graph}. Let $P_\Phi \in \Z[H_1(\Phi)]$ be the Perron polynomial of $\Phi_\tau$ which is defined as $P_\Phi = \det(I-A_\Phi)$, where $A_\Phi$ is an `adjacency matrix'  for $\Phi$ (see \Cref{sec:flowgraph}). In \Cref{th:veering_digraph}, we show

\begin{thm}[From flow graph to veering polynomial]\label{thm:perron_intro}
Let $i_* \colon \Z[H_1(\Phi)] \to \Z[G]$ be induced by the inclusion $i \colon \Phi_\tau \to M$. Then
\[
V_\tau =  i_*(P_{\Phi}).
\]
\end{thm}

Next, we recall that the $2$-skeleton $\tau^{(2)}$ of $\tau$ is a transversely oriented
branched surface which can carry surfaces similar to the way a train track on a surface can carry curves (see \Cref{sec:veering_basics}). We let $\cone_2(\tau)$ be the closed cone in $H_2(M,\partial M)$ positively generated by classes that are represented by the surfaces 
that $\tau$ carries.
We call $\cone_2(\tau)$ the \define{cone of carried classes} and note that it can be explicitly computed as the nonnegative solutions to the \emph{switch conditions} for $\tau^{(2)}$.

We show that $\cone_2(\tau)$ is dual to the cone in $H_1(M)$ generated by closed positive transversals to $\hbs$ in $M$, which we call the \define{cone of homology directions} (see \Cref{sec:conedefs}). The cone of homology directions is in turn generated by the support of $P_\Phi$ (see \Cref{th:cones_equal} and \Cref{lem:support_gen}):

\begin{thm}[Duality of cones]\label{th:cones_intro}
For $\alpha \in H_2(M,\partial M)$, the following are equivalent:
\begin{enumerate}
\item $\alpha \in \cone_2(\tau)$,
\item $\langle \gamma, \alpha \rangle \ge 0$ all closed positive transversals $\gamma$ to $\tau$, and 
\item $\langle i(c), \alpha \rangle \ge 0$, for each $c \in H_1(\Phi_\tau)$ in the support of $P_\Phi$. 
\end{enumerate} 
\end{thm}

For the connection to faces of the Thurston norm ball,
let $x$ denote the Thurston norm on $H_2(M,\partial M)$ and let $B_x(M)$ be its unit ball. 
In \Cref{sec:faces_thurston_norm}, we associate to $\tau$ a combinatorial 
\define{Euler class} $e_\tau \in H^2(M,\partial M)$, one definition of which is
\[
e_\tau=-\frac{1}{2}\langle \prod_{i=1}^ng_i ,\cdot\rangle \in H^2(M,\partial M),
\]
where $g_1,\ldots, g_n$ are classes represented by the $AB$-cycles as in \Cref{th:factor_intro}. In \Cref{th:omni}, we prove

\begin{thm}[$\tau$ determines a face] \label{th:faces_intro}
The cone of carried classes $\cone_2(\tau)$ is \emph{equal} to the cone over a 
(possibly empty)
face $\bf F_\tau$ of the Thurston norm ball $B_x(M)$. 
This cone is characterized by the property that it is the subset of 
$H_2(M,\partial M)$ on which $-e_\tau = x$.

Furthermore, the following are equivalent:
\begin{enumerate}[label=(\roman*)]
\item $i_*(\supp(P_{\Phi_\tau}))$ lies in an open half-space of $H_1(M;\R)$,
\item there exists $\eta\in H^1(M)$ with $\eta([\gamma])>0$ for each closed $\tau$-transversal $\gamma$,
\item $\tau$ is layered, and
\item $\bf F_\tau$ is a fibered face.
\end{enumerate}
\end{thm}

We emphasize that \Cref{th:faces_intro} also gives a characterization of layeredness of $\tau$ (and fiberedness of $\bf F_\tau$) in terms of 
a cohomological positivity condition. This makes it 
a combinatorial analog of Fried's criterion \cite[Theorem D]{fried1982geometry} for 
a flow to be circular, i.e. admit a cross section. 

As a last remark, we note that by \cite{Landry_norm} if $\overline M$ is a closed manifold obtained from $M$ by Dehn filling along slopes intersecting the ladderpoles (see \Cref{sec:veeringonboundary}) of $\partial M$ enough times, the image of $\hbs$ under the inclusion $j\colon M\hookrightarrow \overline M$ determines a face $\overline {\bf F}$ of the Thurston norm ball. Thus $j_*(V_\tau)$ is an object associated to $\overline {\bf F}$ generalizing the Teichm\"uller polynomial. Moreover, $\overline {\bf F}$ is fibered exactly when the image of $\supp(P_{\Phi_\tau})$ lies in an open half-space of $H_1(\overline M;\R)$.

\subsection{Sequel paper: Flows, growth rates, and the veering polynomial}
In a followup \cite{veeringpoly2} to the current paper, we develop the connection between the combinatorial approach to the veering polynomial developed here and the pseudo-Anosov flow associated to the veering triangulation.
Since this is relevant for motivating our constructions, we briefly summarize the main points of \cite{veeringpoly2}.

Suppose that $\phi$ is a \emph{pseudo-Anosov flow without perfect fits} on a closed manifold $\overline M$. 
Then unpublished work of Agol-Gu\'eritaud produces a veering triangulation $\tau$ on the manifold $M$ obtained by puncturing $\overline M$ along the singular orbits of $\phi$.

First, the flow graph $\Phi_\tau$ codes the orbits of $\phi$ in the following precise sense. There is a map from directed cycles of $\Phi_\tau$ to closed orbits of $\phi$ so that each directed cycle is homotopic to its image. This map is uniformly bounded-to-one and for any closed orbit $\gamma$, either $\gamma$ or $\gamma^2$ is in its image. These properties are similar to those of a Markov partition for $\phi$, but we emphasize that $\Phi_\tau$ is combinatorially defined (\Cref{sec:flowgraph}) and canonically associated to $\tau$.

Second, from the connection between orbits of the flow and directed cycles of $\Phi_\tau$, we use $P_\Phi$ to compute the growth rates of orbits of $\phi$.
If $S$ is a fiber surface carried by a layered triangulation $\tau$, this recovers 
well-known properties that McMullen established for the Teichm\"uller polynomial (using \Cref{th:factor_intro} and \Cref{th:teich_intro}).
If $S$ is a transverse surface carried by a (possibly nonlayered) veering triangulation $\tau$ which is \emph{not} a fiber in $\R_+\bf F_\tau$, then $S$ necessarily misses closed orbits of the flow by Fried's criterion, and so the usual counting results are not possible.

However, we can consider the manifold $M|S$ obtained by cutting $M$ along $S$ and ask whether there is a class $\xi \in H^1(M|S)$ that is positive on the surviving closed orbits. In this case, we show how the growth rate of closed orbits \emph{in} $M|S$ with respect to $\xi$ is recorded by the flow graph $\Phi_\tau$ and its Perron polynomial $P_\Phi$. 
This analysis includes the special case where
 $S$ is in the boundary of a fibered cone $\R_+\bf F_\tau$, and our results are new
even in this setting.

\subsection*{Acknowledgements} We thank Anna Parlak for interesting conversations on the subject of this paper and for pointing out an algebraic mistake in an earlier draft. We also thank Spencer Dowdall, Curtis McMullen, and Henry Segerman for helpful feedback.

\section{Background} \label{sec:background}
Here we record some background that we will need throughout the paper. In what follows, all $3$-manifolds are assumed to be connected and oriented.

\subsection{Veering triangulations}
\label{sec:veering_basics}
We begin by defining a taut ideal triangulation following Lackenby \cite{lackenby2000taut} (see also \cite{hodgson2011veering}). Such triangulations are also called \emph{transverse taut} by e.g. \cite{futer2013explicit}.

A \define{taut} ideal tetrahedron is an ideal tetrahedron (i.e. a tetrahedron minus its vertices) along with a coorientation on each face such that
two of its faces point into the tetrahedron and two of its faces point out of the tetrahedron. The inward pointing faces are called its \emph{bottom faces} and the outward faces are called its \emph{top faces}.
Each of its edges is then assigned angle $\pi$ or $0$ depending on whether
 the coorientations on the adjacent faces agree or disagree, respectively. See \Cref{fig:taut}.
 
\begin{figure}
\centering
\includegraphics[height=1.5in]{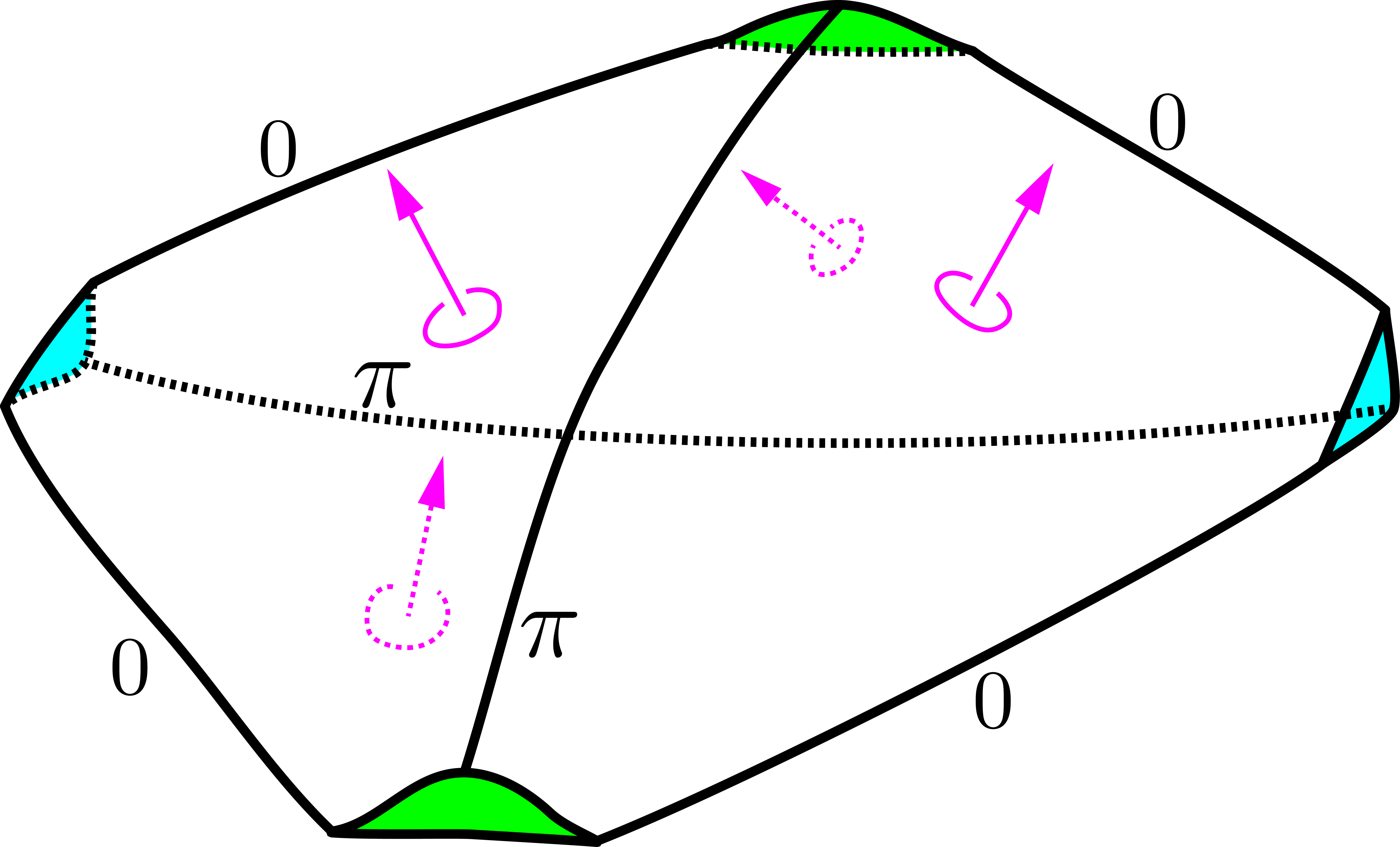}
\caption{A (truncated) taut tetrahedron with its face coorientations and angles.}
\label{fig:taut}
\end{figure}
 
An ideal triangulation of $M$ is \define{taut} if each of its faces has been cooriented so that each ideal tetrahedron is taut and the angle sum around each edge is $2\pi$. 
The local structure around each edge $\bf{e}$ is as follows: $\bf{e}$ includes as a $\pi$-edge in two tetrahedra. For the other tetrahedra meeting $\bf{e}$, $\bf{e}$ includes as a $0$-edge and these tetrahedra form the \define{fan} of $\bf{e}$. 
We observe that the fan of $e$ has two \define{sides} each of which is linearly ordered by the coorientation on faces. See, for example, \Cref{fig:branch_edge}, where the coorientation points upwards. 

A \define{veering triangulation} $\tau$ of $M$ is a taut ideal triangulation of $M$ in
which each edge has a consistent \define{veer}; that is, each edge is labeled to be either
\emph{right} or \emph{left} veering such that each tetrahedron of $\tau$ admits an
orientation preserving isomorphism to the model veering tetrahedron pictured in
\Cref{fig:veer}, in which the veers of the 0-edges are specified: 
right veering edges have positive slope and left veering edges have negative slope. The
$\pi$-edges can veer either way, as long as adjacent tetrahedra satisfy the same rule. 
In other words (c.f. \cite{gueritaud}), each oriented taut tetrahedron $t$ of $\tau$ can be realized as a thickened rhombus in $\mathbb{R}^2 \times \mathbb{R}$ with angle $\pi$ at its diagonal edges and angle $0$ at its side edges such that its vertical diagonal lies above its horizontal diagonal and its right/left veering side edges have positive/negative slope, respectively. 

\begin{figure}[h]
\centering
\includegraphics[height=1.5in]{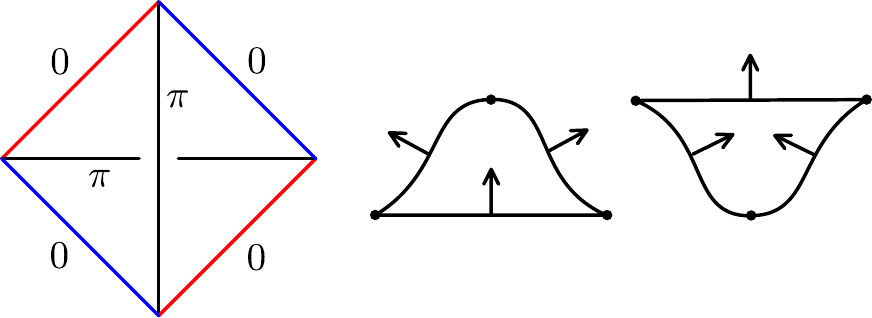}
\caption{On the left is a model veering tetrahedron: right veering edges are red, left veering edges are blue. If this tetrahedron appears in a veering triangulation, then the veers of the top and bottom ($\pi$-) edges are determined by adjacent tetrahedra. In the center is an upward flat triangle, and on the right is a downward flat triangle.} 
\label{fig:veer}
\end{figure}

Note that these conditions imply that the $0$-edges around a tetrahedron $t$ have alternating veers, and that each face has edges that veer in both directions.
From our definition it also follows that for each edge $e$ of $\tau$, each side of $e$ is nonempty (i.e. contains a tetrahedron).
This was observed in \cite[Lemma 2.3]{hodgson2011veering} and is also part of Agol's original definition \cite[Definition 4.1]{agol2011ideal}.
To prove it, note that along a bottom face of a tetrahedron we encounter the $\pi$-edge followed by the left veering followed by the right veering edge in the counterclockwise order, using the coorientation on faces. However, along a top face we encounter the $\pi$-edge followed by the right veering followed by the left veering edge, and so two tetrahedra cannot be glued along faces so that their $\pi$-edges are identified. 

A veering (or taut) triangulation is said to be \define{layered} if it can be built by stacking tetrahedra onto a triangulated surface and quotienting by a homeomorphism of the surface. (For another, more formal definition, we refer the reader to \cite[Definition 2.15]{schleimer2019veering}).  
Finally, we recall that the constructions of layered veering triangulations of Agol \cite{agol2011ideal} and Gu\'eritaud \cite{gueritaud} start with a pseudo-Anosov homeomorphism $f$ of a surface $S$ and produces a veering triangulation $\tau$ on the mapping torus $M$ of $f \colon S \ssm \{\text{singularities}\} \to  S \ssm \{\text{singularities}\}$. 

\begin{remark}[Veering definitions]
For us, a veering triangulation of $M$ is a taut ideal triangulation such that each edge has a consistent veer. Elsewhere in the literature, this is known as a \emph{transverse} veering triangulation \cite{hodgson2011veering, futer2013explicit, schleimer2019veering}. A slightly more general definition can be given where the taut structure is replaced by a \emph{taut angle structure}, which does not impose a coorientation on faces.
The two conditions, however, are equivalent up to a double cover \cite[Lemma 5.4]{futer2013explicit} and agree for layered veering triangulations, which was the setting of Agol's original definition \cite{agol2011ideal}.
\end{remark}

\subsubsection{The $2$-skeleton $\hbs$ as a branched surface}
As observed by Lackenby \cite{lackenby2000taut}, the taut structure of $\tau$ naturally gives its $2$-skeleton $\tau^{(2)}$ the structure of a transversely oriented branched surface in $M$. 
(See \cite{floyd1984incompressible, oertel1984incompressible} for general facts about branched surfaces.)
The smooth structure on $\tau^{(2)}$ can be obtained by, within each tetrahedron, smoothing along the $\pi$-edges and pinching along the $0$-edges, thus giving $\tau^{(2)}$ a well-defined tangent plane field at each of its points. See \Cref{fig:branch_edge}. 

\begin{figure}[h]
\centering
\includegraphics[height=1.5in]{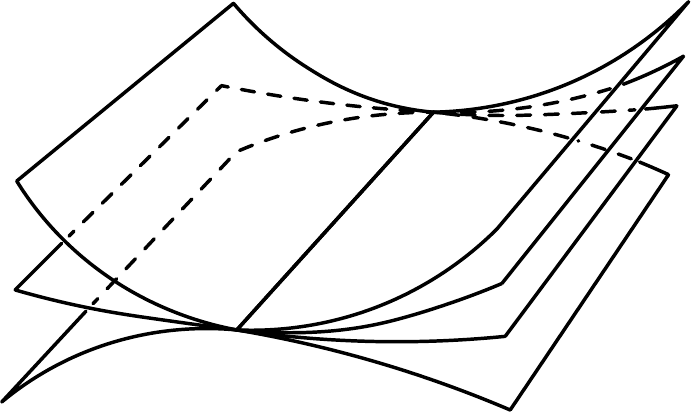}
\caption{The branched surface $\tau^{(2)}$ near an edge of $\tau$.} 
\label{fig:branch_edge}
\end{figure}

With this structure, the branching locus of $\tau^{(2)}$ is the disjoint union of edges of $\tau$ and the sectors (i.e. the complementary components of the branching locus) are the faces of $\tau$. 

The branched surface $\tau^{(2)}$ has a branched surface fibered neighborhood $N= N(\tau^{(2)})$ foliated by intervals such that collapsing $N$ along its $I$-fibers recovers $\tau^{(2)}$. 
The transverse orientation on the faces of $\tau$ consistently orients the fibers of $N$, and a properly embedded oriented surface $S$ in $M$ is \define{carried} by $\tau^{(2)}$ if it 
is contained in $N$ where it is positively transverse to its $I$-fibers. 
We also say that $S$ is carried by $\tau$. Note that up to isotopy a surface may be carried by $\tau$ in different ways.

A carried surface $S$
embedded in $N$ transverse to the fibers 
defines a nonnegative integral weight on each face of $\tau$ given by the number of times the $I$-fibers over that face intersect $S$. These weights satisfy the \define{matching (or switch) conditions} stating that the sum of weights on one side of a edge match the sum of weights on the other side. Conversely, a collection of nonnegative integral weights satisfying the matching conditions gives rise to a surface embedded in $N$ transverse to the fibers in the usual way. 

More generally, any collection of nonnegative weights on faces of $\tau$ satisfying the matching conditions defines a nonnegative relative cycle giving an element of $H_2(M, \partial M ; \RR)$ and we say that a class is \define{carried} by $\tau^{(2)}$ if it can be realized by such a nonnegative cycle. Just as with surfaces, a carried class can be represented by more than one nonnegative cycle on faces.

We conclude by observing the following:

\begin{lemma} \label{lem:integral_classes}
If $\alpha \in H_2(M, \partial M ; \ZZ)$ is carried by $\tau^{(2)}$, then it is realized by a nonnegative integral cycle on $\tau^{(2)}$ and hence an embedded surface carried by $\hbs$.
\end{lemma}

In fact, the lemma holds for any transversely oriented branched surface.

\begin{proof}
Let $W=\R^F$ be the weight space of $\hbs$ and let $Z\subset W$ be the subspace of
relative 2-cycles, i.e. weights satisfying the branching conditions. Let $A\colon Z\to
H_2(M,\partial M;\R)$ be the map to relative homology. We can choose rational bases for
$Z$ and $H_2$ so that $A$ is represented by an integer matrix.

Let $P\subset Z$ be the set of solutions to $Ax=\alpha$. Because $\alpha$ is carried by
$\hbs$, $P$ contains a point $m$ lying in the nonnegative orthant $\R_{\ge0}^F$. Let $Q$
denote the face of the nonnegative orthant containing $m$ in its relative interior. Since
$P\cap Q$ is nonempty and cut out by integer equations, there is a rational point $w\in
P\cap Q$.
(Here we are using the fact that for any rational linear map $L \colon \R^n\to\R^m$ and rational
$w\in \R^m$, if the equation $Lx=w$ has solutions then rational solutions exist and are
dense among all solutions). 

Clearing denominators, there is an integer $n \ge 0$ such that $n\alpha$ is represented by the integral cycle $nw$, giving rise to an embedded surface $S$. By \cite[Lemma 1]{thurston1986norm}, $S$ is a union of $n$ surfaces, each representing $\alpha$ and carried by $\hbs$. 
\end{proof}

\subsubsection{The veering triangulation as seen from $\partial M$}\label{sec:veeringonboundary}
If we truncate the tips of the tetrahedra of $\tau$, as in \Cref{fig:taut}, we obtain a compact manifold whose boundary components are tori. For details, see \cite{lackenby2000taut, hodgson2011veering}.
We denote this manifold $\mr M$ and continue to use $\tau$ to refer to the modified veering structure. 

In what follows, we will often write $\partial M$ to mean $\partial \mr M$ and $(M, \partial M)$ to mean $(\mr M, \partial \mr M)$. This simplifies our notation and should cause no confusion.

The intersection $\partial \tau = \tau^{(2)} \cap \partial M$ is a cooriented train track with each complementary component a \define{flat triangle} (i.e. a bigon with three branches of $\partial \tau$ in its boundary) that corresponds to the tip of truncated taut tetrahedron. See the righthand side of \Cref{fig:veer}. The veering structure of $\tau$ determines (and is determined by) the structure of these induced train tracks. We will recall some facts here that are needed for \Cref{sec:relating}, but we refer the reader to \cite{futer2013explicit} and \cite{Landry_norm}
 for a more detailed analysis.

Each flat triangle $T$ on $\partial M$ complementary to the track $\partial \tau$ has two vertices at cusps, corresponding to $0$-edges of the associated tetrahedron, and one smooth vertex, corresponding to the $\pi$-edge. If the coorientation at the smooth vertex points out of $T$, when $T$ is called an \define{upward} triangle, and otherwise its called \define{downward}. Note that the veer of the $\tau$-edges corresponding to the $0$-vertices (i.e. cusps) of $T$ are determined by whether $T$ is upward or downward, as in \Cref{fig:veer}.

\begin{figure}[h]
\centering
\includegraphics[height=2in]{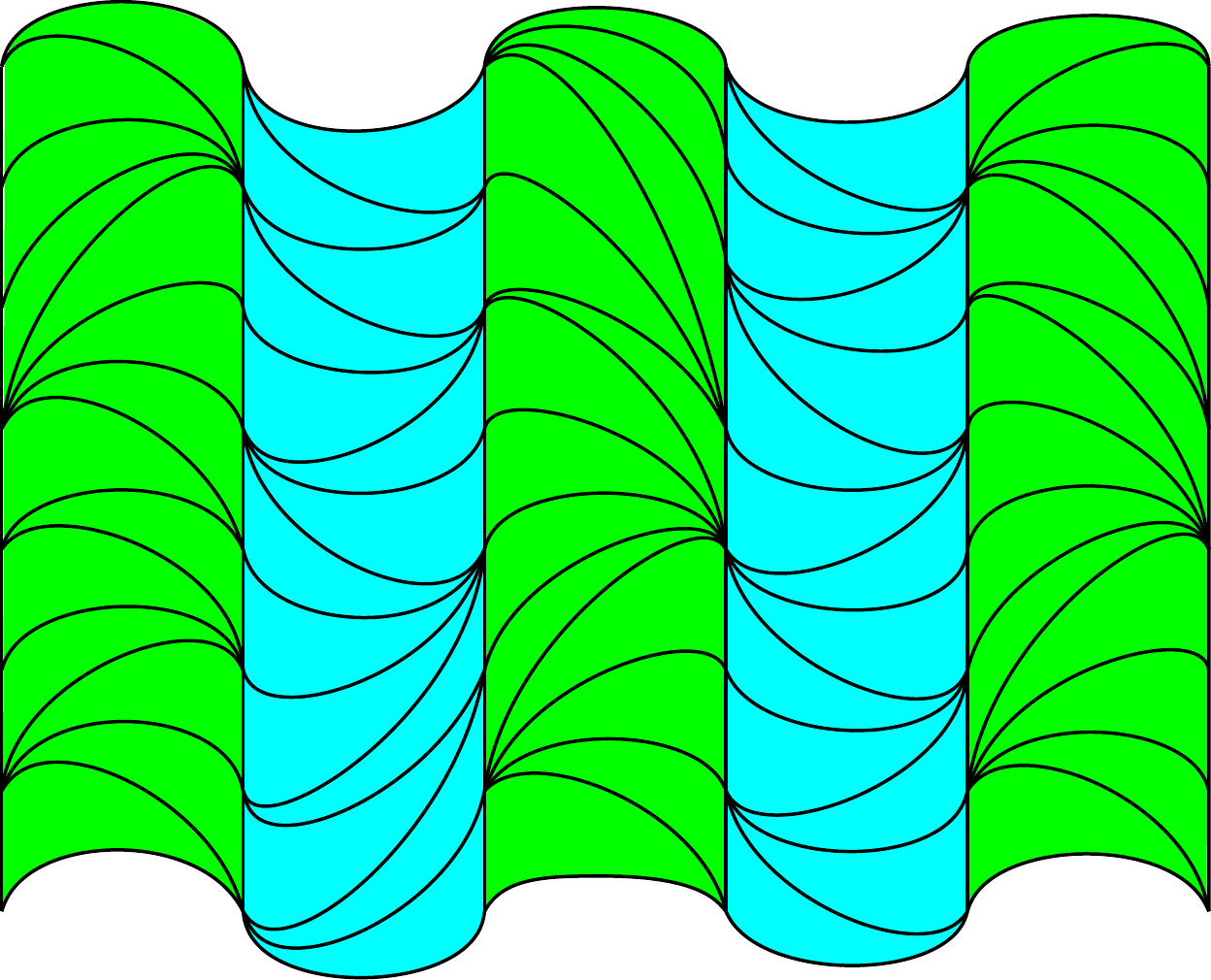}
\caption{The tessellation of $\partial M$ by flat triangles organized into upward and downward ladders, colored green and cyan respectively for compatibility with \Cref{fig:taut}.} 
\label{fig:ladders}
\end{figure}

The flat triangles of $\partial M$ are organized into \emph{upward and downward ladders} as follows (see \cite[Observation 2.8]{futer2013explicit}):
For each component of $\partial M$, the union of all upward flat triangles is a disjoint collection of annuli such that each annulus $A$ in the collection, called an \define{upward ladder}, has $\partial A$ carried by $\partial \tau$ and all other branches of $\partial \tau$ meeting $A$ (called \define{rungs}) join its two boundary components. We define a \define{downward ladder} similarly, and observe that each component of $\partial M$ is an alternating union of upward and downward ladders meeting along their common boundary (called \define{ladder poles}). See \Cref{fig:ladders}.

\subsection{Thurston norm}\label{sec:normbkgd}
Suppose that $M$ is a finite volume hyperbolic manifold. 
Then the \define{Thurston norm} $x$ is a norm on the vector space $H_2(M,\partial M; \RR)$ which extends the formula
\[
x(\alpha) = \min\{ -\chi(S)\}
\]
where $\alpha \in  H_2(M,\partial M; \ZZ)$
and the minimum is over all properly embedded surfaces representing $\alpha$ without sphere or disk components
\cite{thurston1986norm}. At times it will be convenient to identify $H_2(M,\partial M;\R)$ with $H^1(M;\R)$ via Lefschetz duality.

The unit ball $B_x = B_x(M)$ in $H_2(M,\partial M; \RR)$ with respect to $x$ is called the \define{Thurston norm ball} and it is a rational finite-sided polyhedron. There is a (possibly empty) collection of top dimension faces of $B_x$, called \define{fibered faces}, such that 
\begin{itemize}
\item every class $\alpha$ representing a fiber $S$ in a fibration of $M$ over $S^1$ is contained in the interior of the cone $\RR_+\bf F$ over a fibered face $\bf F$ of $B_x$,
\item every primitive integral class $\alpha$ in the interior of the cone over a fibered face represents a fiber in a fibration of $M$ over $S^1$ \cite{thurston1986norm}.
\end{itemize}

According to \cite{agol-overflow} and \cite[Proposition 2.7]{minsky2017fibered}, the Agol-Gu\'eritaud construction applied to any fiber in a fixed fibered face $\bf F$ produces the same veering triangulation $\tau$. In particular, a \emph{layered} veering triangulation $\tau$ of $M$ is canonically associated to a fibered face $\bf F_\tau$ of $B_x(M)$. 

In general, we say that a properly embedded surface $S$ is \define{taut} if no component of $S$ is nullhomologous and $S$ is Thurston norm minimizing, i.e. $-\chi(S) = x([S])$. For example, fibers are necessarily taut. 

Returning to the situation of interest, Lackenby \cite[Theorem 3]{lackenby2000taut} proved that for any taut ideal triangulation, the underlying branched surface is \define{taut} in the sense that every surface it carries is taut.

\subsection{Polynomials and specialization} \label{sec:spec}
Fix a finitely generated free abelian group $G$ and denote its group ring with integer coefficients by $\Z[G]$.
Recall that $\Z[G]$ is a UFD.

Let $P \in \Z[G]$ and write $P= \sum_{g\in G} a_g  \cdot g$. The \define{support} of $P$ is
\[
\mathrm{supp}(P) = \{g \in G : a_g \neq 0 \}.
\]

For $P \in \Z[G]$ with $P = \sum_{g\in G} a_g \cdot g$  and $\alpha \in \hom(G,\R)$
the \define{specialization of $P$ at $\alpha$} is the single variable expression
\[
P(u^\alpha) = \sum_{g\in G} a_g \cdot u^{\alpha(g)} \in \Z[u^r : r \in \RR].
\]
If $G$ has rank $n$ and is written multiplicatively, then we can identify $\Z[G]$ with $\Z[t_1^{\pm1}, \ldots, t_n^{\pm 1}]$. Making this identification, specialization is the image of $P$ under the homomorphism 
\linebreak
$\alpha_* \colon \Z[t_1^{\pm1}, \ldots, t_n^{\pm 1}] \to \Z[u^r : r \in \RR]$ obtained by replacing $t_i$ with $u^{\alpha(t_i)}$.


\section{The veering and taut polynomials} \label{sec:veering}
Let $M$ be a $3$-manifold with a
 veering triangulation $\tau$. Denote the sets of edges, faces, and tetrahedra of $\tau$ by $E$, $F$, and $T$, respectively. Since $\chi(M) = 0$, we see that $|E| = |T| = \frac{1}{2}|F|$. Each tetrahedron has a unique \define{bottom} edge and this induces a bijection
from tetrahedra to edges
 that we will use throughout. If ${\bf e}$ is the bottom of the tetrahedron $t$, then we also say that $t$ lies \define{above} ${\bf e}$.

As $\tau$ is veering, each of its edges is labeled to be either right or left veering. See \Cref{sec:veering_basics} and \Cref{fig:veer}. 
Each abstract veering tetrahedron $t$ gives a relation among its edges. Let $\mathbf{b}$ and $\mathbf{t}$ denote the the bottom and top edges of $t$, respectively. Among the 4 side edges of $t$, let $\mathbf{s}_1$ and $\mathbf{s}_2$ be those which have the \emph{opposite} veer from $\mathbf{t}$.
Then we have the \define{tetrahedron relation}
\begin{align} \label{eq:tet_rel}
\mathbf{b} = \mathbf{t} +\mathbf{s}_1 + \mathbf{s}_2,
\end{align}
associated to $t$. We also say that this tetrahedron relation is 
associated to the bottom edge $\bf b$ of $t$.
Informally, this relation mimics the map on measured train tracks induced by a fold. See \Cref{fig:stablebs}.

Let $G = \hom(H^1(M), \Z)$ be the first homology of $M$ modulo torsion and let $\aM$ be the associated covering space. This is the universal free abelian cover of $M$, and its deck group is $G$. Let $\widetilde \tau$ be the preimage of $\tau$. Note that edges, faces, and tetrahedra of $\widetilde{\tau}$ are in bijective correspondence with $E \times G$, $F \times G$, and $T \times G$. The correspondence is determined by any choice of lifts of simplices from $\tau$.

We define the \define{edge module} 
$\e(\widetilde{\tau})$ to be the free $\Z$-module on the edges of $\wt \tau$ modulo the relations from \Cref{eq:tet_rel} for each tetrahedron of $\widetilde{\tau}$. The action of $G$ on $\aM$ by deck transformations makes $\e(\widetilde{\tau})$ into a module over the group ring $\Z[G]$ and we will henceforth consider $\e(\widetilde{\tau})$ as a $\Z[G]$-module. 
Now \emph{choose} a lift of each edge of $\tau$ to $\wt \tau$.
Since the free $\Z$-module on the edges of $\wt \tau$ is isomorphic to $\Z[G]^E$ as a $\Z[G]$-module, $\e(\widetilde{\tau})$ has the presentation
\begin{align} \label{eq:edge_mod}
 \Z[G]^E \overset{L}{\longrightarrow} \Z[G]^E \longrightarrow \e(\widetilde{\tau}) \to 0,
\end{align}
where $L$ maps $\bf b$ to $\mathbf{b} - (\mathbf{t} +\mathbf{s}_1 + \mathbf{s}_2)$. 
That is, the image of each edge is determined by the tetrahedron relation for the 
tetrahedron lying above that edge.

We define the \define{veering polynomial} of $\tau$ to be the element
\[
V = V_\tau = \det(L) \in \Z[G].
\]

We remark that the map $L$ can we written in the form $L = I - A$, where $I$ is the identity matrix and $A$ is a matrix with coefficients in $G$.
In \Cref{sec:digraph}, we will see that $A$ can be interpreted as the adjacency matrix for a directed graph (the \emph{flow graph} of \Cref{sec:flowgraph}) associated to $\tau$.

The following lemma shows that $V_\tau$ is well-defined:

\begin{lemma} \label{lem:V_tau_defined}
The veering polynomial $V_\tau = \det(L) \in \Z[G]$ depends only on $\tau$ and not the choice of lifts of edges to $\wt \tau$.
\end{lemma}

\begin{proof}
If $e$ is an edge of $\tau$ and $\widetilde e$ is the chosen lift in $\widetilde \tau$, then the effect of replacing $\widetilde e$ with $g \cdot  \widetilde e$ amounts to conjugating the matrix $L$  by the $|E| \times |E|$ matrix which is obtained from the identity matrix by replacing the $1$ in the diagonal entry corresponding to $e$ with $g$. This does not affect the determinant.
\end{proof}

\subsection{The face module and $\Theta_\tau$}
\label{sec:taut}
Our construction of the veering polynomial $V_\tau$ make explicit use of the veering structure of $\tau$. Here, we define a closely related polynomial, which we call the taut polynomial of $\tau$, whose construction uses only the taut structure of $\tau$ (see \Cref{sec:veering_basics}). 
Although we begin to describe the connection between the two polynomials in this section, 
their precise relationship will be fully
explained in \Cref{sec:polys}. In \Cref{sec:layered} we will show that when $\tau$ is layered, its taut polynomial is \emph{equal} to the Teichm\"uller polynomial of the associated fibered face of the Thurston norm ball. Taken together, these results will give the explicit connection between $V_\tau$ and the Teichm\"uller polynomial.

Any face $f$ of the veering triangulation lies at the bottom of a unique tetrahedron $t$ and has a distinguished edge $\bf b$ which is the bottom edge of $t$. We often call $\bf b$ the \define{bottom} edge of $f$.
Let $\bf x$ and $\bf y$ be the other edges of $f$. Then the \define{face relation} associated to $f$ is
\begin{align} \label{eq:face_rel}
\mathbf{b} = \mathbf{x} + \mathbf{y}.
\end{align}
Since there are two faces lying at the bottom of tetrahedron, $t$ has two associated face relations.

Again, identify $\Z[G]^E$ with the free module of edges of $\widetilde \tau$, we quotient by the relations given in \Cref{eq:face_rel} for each face of $\widetilde \tau$ to obtain the \define{face module} $\e^\triangle(\widetilde \tau)$. 
We have the presentation
\begin{align}\label{eq:face_mod}
\Z[G]^{F}  \overset{L^\triangle}{\longrightarrow} \Z[G]^E \longrightarrow \e^\triangle(\widetilde \tau) \to 0,
\end{align}
where $L^\triangle$ is determined by mapping a face $f$ to its associated relation $\bf b - (x_1 +x_2)$.

The module $\e^\triangle(\widetilde \tau)$ has a well defined Fitting ideal $I^\triangle \subset \Z[G]$ generated by determinants of $|E| \times |E|$ submatrices of $L^\triangle$. (See, for example, \cite[Theorem 1]{northcott2004finite}.)
The \define{taut polynomial}
$\Theta_\tau \in \Z[G]$ is defined to be the greatest common divisor of the elements of $I^\triangle$. Note that this only determines $\Theta_\tau$ up to multiplication by a unit $\pm g \in \Z[G]$.

We next observe the following immediate consequence of our definitions and use it to relate the two polynomials. 

\begin{lemma} \label{lem:face_implies_tet}
For each bottom face $f$ of a tetrahedron $t$, there is a unique top face $f'$ of $t$ such that the sum of the face relations for $f$ and $f'$ is equal to the tetrahedron relation for $t$.

Further, $f'$ is characterized by the property that it meets $f$ within $t$ along the edge that has the same veer as the top edge of $t$.
\end{lemma}

\begin{proof}
Consider the face relation \Cref{eq:face_rel} for $f$. 
Recall that the tetrahedron relation for $t$ is $\mathbf{b} = \mathbf{t} + \mathbf{s_1} + \mathbf{s_2}$, where $\mathbf{t}$ is the top edge of $t$ and $\mathbf{s_1}, \mathbf{s_2}$ are the side edges which have opposite veer from that of $\mathbf{t}$. Let $\mathbf{r_1}, \mathbf{r_2}$ be the other two edges of $t$ labeled so that $\mathbf{b}, \mathbf{s_i}, \mathbf{r_i}$ give the bottom faces of $t$ for $i =1,2$. 

After possibly swapping indices, we may suppose that 
$\mathbf{b}, \mathbf{s_1}, \mathbf{r_1}$ span the face $f$ and so
$\mathbf{b} -( \mathbf{s_1} + \mathbf{r_1}) =0$ is the face relation for $f$. 
By construction, $\mathbf{r_1}$ has the {same} veer as $\mathbf{t}$
and so the edges $\mathbf{t}, \mathbf{r_1}, \mathbf{s_2}$ span a top face of $t$. Indeed since the third edge in the face with $\mathbf{t}$ and $\mathbf{r_1}$ must have opposite veer this leaves only 
$\mathbf{s_1}$ and $\mathbf{s_2}$ as possibilities and $\mathbf{b}, \mathbf{s_1} , \mathbf{r_1}$ is already a face on the bottom of $t$.
Hence, we see that $\mathbf{t}, \mathbf{r_1}, \mathbf{s_2}$ span a face $f'$ at the top of $t$.

From this we can see that the tetrahedron glued to $t$ along $f'$ has $\mathbf{r_1}$ as its bottom edge.
Otherwise, the bottom edge would be either $\mathbf{t}$ or $\mathbf{s_2}$. However, in the first case $\mathbf{t}$ would have an empty side of its fan and in the second case the
new tetrahedron would have $4$ side edges all the same veer (that of $\mathbf{t}$), 
either giving a contradiction. 
We conclude that the face relation for $f'$ is $\mathbf{r_1} -( \mathbf{t} + \mathbf{s_2})=0$. 

Finally, summing the face relations for the bottom face $f$ and top face $f'$ gives $\mathbf{b} -( \mathbf{t} + \mathbf{s_1} + \mathbf{s_2}) = 0$, which is the tetrahedron relation for $t$.  This completes the proof.
\end{proof}

We record the following observation made in the proof of \Cref{lem:face_implies_tet}.
\begin{fact}\label{rmk:top_of_fan}
Let $\bf{e}$ be an edge of $\tau$ and let $t$ be a tetrahedron in the fan of $\bf{e}$. Then the top edge of $t$ has the same veer as $\bf{e}$ if and only if $t$ is topmost in its side of $\bf{e}$.
\end{fact}

Since the faces relations determine the tetrahedra relations, we have

\begin{corollary} \label{cor:module_surject}
There is a surjective $\Z[G]$-module homomorphism:
\[
\e(\widetilde \tau) \longrightarrow \e^\triangle(\widetilde \tau) \longrightarrow 0,
\]
and so $\Theta_\tau$ divides $V_\tau$.
\end{corollary}

\begin{proof}
The surjection is immediately given by \Cref{lem:face_implies_tet} along with the definitions of the modules. This implies that $I^\triangle$ contains the Fitting ideal for $\e(\widetilde \tau)$ (see e.g. \cite[Appendix]{mazur1984class} or \cite[Chapter 3]{northcott2004finite}), which is principally generated by $V_\tau$. This completes the proof.
\end{proof}

In fact, the quotient polynomial $V_\tau / \Theta_\tau$ can be easily described as a products of polynomials of the form $(1\pm g)$ for $g \in G$ (\Cref{th:factorization!}). The proof, however, is quite involved and postponed until \Cref{sec:polys}.

\section{The flow graph and its Perron polynomial} \label{sec:digraph}
In this section, we reinterpret our construction in terms of the \emph{flow graph} (\Cref{sec:flowgraph}) associated to the veering triangulation. This will have implications for identifying the cone of classes carried by the veering triangulation in \Cref{sec:cones}.

For a directed graph $D$ with $n$ vertices, let $A$ denote the matrix with entries
\begin{align} \label{eq:adj}
A_{ab} = \sum_{\partial e=b-a} e,
\end{align}
where the sum is over all edges $e$ from the vertex $a$ to the vertex $b$. This is similar to a standard adjacency matrix for a directed graph, but note that $A$ lives in the matrix ring $M_{n\times n}(\Z[C_1(D)])$ where $C_1(D)$ is the group of simplicial 1-chains in $D$. We call $A$ the \define{adjacency matrix} for $D$.
The \define{Perron polynomial} of $D$ is defined to be $P_D = \det(I - A)$.
 
Following McMullen \cite{mcmullen2015entropy}, we define the \define{cycle complex} $\C(D)$ of $D$ to be the graph whose vertices are directed simple cycles of $D$ and whose edges correspond to disjoint cycles. We recall that $P_D$ equals the \define{clique polynomial} of $\C(D)$, which in particular shows that $P_D \in \Z[H_1(D)]$ (see \cite[Theorem 1.4 and Section 3]{mcmullen2015entropy}).  
Here, the clique polynomial associated to $\C(D)$ is
\begin{equation}\label{eq:cliquepoly}
P_D = 1 + \sum_{C} (-1)^{|C|} C \in \Z[H_1(D)],
\end{equation}
where the sum is over nonempty cliques $C$ of the graph $\C(D)$ and  $|C|$ is the number of vertices of $C$. 
One proves the above formula by writing the determinant as a sum over elements of the symmetric group $S_n$
and relating the sign of a permutation to the number of cycles in its cycle decomposition.

We begin by introducing an alternative construction of (a variant of) the Perron polynomial which  mimics the construction of the veering polynomial.

\subsection{Vertex modules of labeled graphs}

Let $D$ be an arbitrary directed graph with vertex set $V$, and let $\alpha \colon H_1(D) \to G$ be a surjective group homomorphism. As before, we also use $\alpha$ to denote the induced ring homomorphism $\alpha \colon \Z[H_1(D)] \to \Z[G]$. We think of $\alpha$ as labeling the cycles of $D$, and hence labeling each clique in $\C(D)$ by the product in $G$ of its constituent cycles.

Let $\widetilde D_\alpha$ be the cover of $D$ corresponding to $\mathrm{Ker}(\alpha \circ \mathrm{ab}\colon \pi_1(D) \to G)$ with deck group $G$. For example, if $\alpha$ is the identity homomorphism then $\widetilde D _\alpha$ is the universal free abelian cover of $D$. Note that by making a choice of lift for each vertex of $D$ and letting $G$ act by deck transformations, we obtain a bijection between the vertex set of $\widetilde D_\alpha$ and $V \times G$, and similarly for the edges. 

Define the \define{$\boldsymbol{\alpha}$-labeled vertex module} $V(\widetilde{D}_\alpha)$ of $D$ to be the $\Z[G]$-module obtained by taking the free module on the vertices of $\widetilde{D}_\alpha$ modulo the following relations: for each vertex $v$ of $\widetilde D_\alpha$ if $e_1, \ldots, e_k$ are the edges of $\widetilde D_\alpha$ with initial vertex $v$ and terminal vertices $w_1, \ldots, w_n$, respectively, then
\[
v = w_1 + \cdots + w_n.
\] 

This gives a presentation
\begin{align} \label{eq:vertex_mod}
\Z[G]^V\xrightarrow{L_{D, \alpha}} \Z[G]^V \longrightarrow V(\widetilde{D}_\alpha) \to 0,
\end{align}
where $L_{D,\alpha}$ is a square matrix determined by $L_{D,\alpha}(v) = v - (w_1 + \ldots + w_n)$. We set $P_{D,\alpha} = \det (L_{D,\alpha}) \in \Z[G]$. 
The following lemma is proven in the same manner as \Cref{lem:V_tau_defined}.

\begin{lemma} \label{lem:basis_ind_2}
The polynomial $P_{D,\alpha}$ depends only on $D$ and $\alpha$, and not on the choice of lifts of vertices to $\widetilde D_\alpha$.
\end{lemma}

The key technical result of this section is the following:

\begin{proposition} \label{prop:image_poly}
With notation as above, $P_{D,\alpha}  = \alpha(P_D) \in \Z[G]$.
\end{proposition}

In words, the polynomial $P_{D,\alpha}$ is obtained by replacing the terms of the Perron polynomial with their images under $\alpha$.

\begin{proof}
Choose a maximal tree $T\subset D$, and label each edge of $T$ with the trivial element of $G$. For each directed edge $e$ outside of $T$, there is a unique (not necessarily directed) cycle in $T\cup\{e\}$ traversing $e$ in the positive direction. Label $e$ with the image of this loop under $\alpha$. The corresponding labeling of the vertices of $\C(D)$ (obtained via concatenation of edges) is given by the homomorphism $\alpha$. This labeling gives $D$ the structure of a \emph{$G$-labeled directed graph} $D^G$. 
It has an adjacency matrix $A^G$ with entries in the group ring $\Z[G]$, given by replacing the edges in the entries of $A$ in \cref{eq:adj} with their $G$-labels. A proof entirely similar to that of \cref{eq:cliquepoly} gives that $\det(I-A^G) = \alpha(P_D)$ (c.f. \cite[Theorem 2.14]{algom2015digraphs}). 
Hence it suffices to show that $\det(I-A^G) = P_{D,\alpha}$.

 Lift $T$ to a tree $\widetilde T$ in $\widetilde D_\alpha$ which contains an orbit of each vertex. By \Cref{lem:basis_ind_2}, we are free to choose our lifts of vertices to be the vertices of $\wt T$. We claim that for this choice, $I-A^G = L_{D,\alpha}$.
For simplicity we reuse the symbol $v$ to denote the vertex of $\widetilde T$ lying above $v$.  We can write $L_{D,\alpha}v_i = v_i - (g_{i,1}v_1 + \cdots +g_{i,m}v_m)$ where $g_{i,j}\in G$ and each $g_{i,j}v_j$ corresponds to a unique edge $e_{i,j}$ from $v_i$ to $v_j$.
Note that $g_{i,j}=1=\id_G$ if and only if $e_{i,j}$ lies in $\wt T$. If $e_{i,j}$ lies outside of $\wt T$, then $g_{i,j}$ is the $\alpha$-image of the loop in $T$ corresponding to $e_{i,j}$.
This is exactly the description of the matrix $I-A^G$, so $I-A^G = L_{D,\alpha}$.
In conclusion,
\[
P_{D,\alpha} = \det (L_{D,\alpha}) = \det(I -A^G) = \alpha(P_D). \qedhere
\]
\end{proof}

We next turn to applying these results to a particular directed graph associated to a veering triangulation.

\subsection{The stable branched surface $B^s$ and the dual graph $\Gamma_\tau$}
\label{sec:stable_brached}
Here we define the \define{stable branched surface} $B^s$ in $M$ associated to the veering triangulation $\tau$. It will play an important role in this section and the next. The branched surface $B^s$ also appears in \cite{schleimer2019veering}, where is is called the \emph{upper branched surface in dual position}. 

Topologically, $B^s$ is the $2$-skeleton of the dual complex of $\tau$ in $M$. For each tetrahedron $t$ we define a smooth structure on $B^s_t = B^s \cap t$  as follows: if the top edge of $t$ is left veering, then we smooth according to the lefthand side of \Cref{fig:stablebs} and otherwise we smooth according the the righthand side. It remains to show that this defines a smooth structure globally, i.e. that it agrees across the faces of $\tau$. Note that for a face $f$ of $t$, $f \cap B^s_t$ is a train track with $3$ branches and a single interior switch. So $f \cap B^s_t$ has a single large branch.

\begin{figure}[h]
\centering
\includegraphics[height=1.5in]{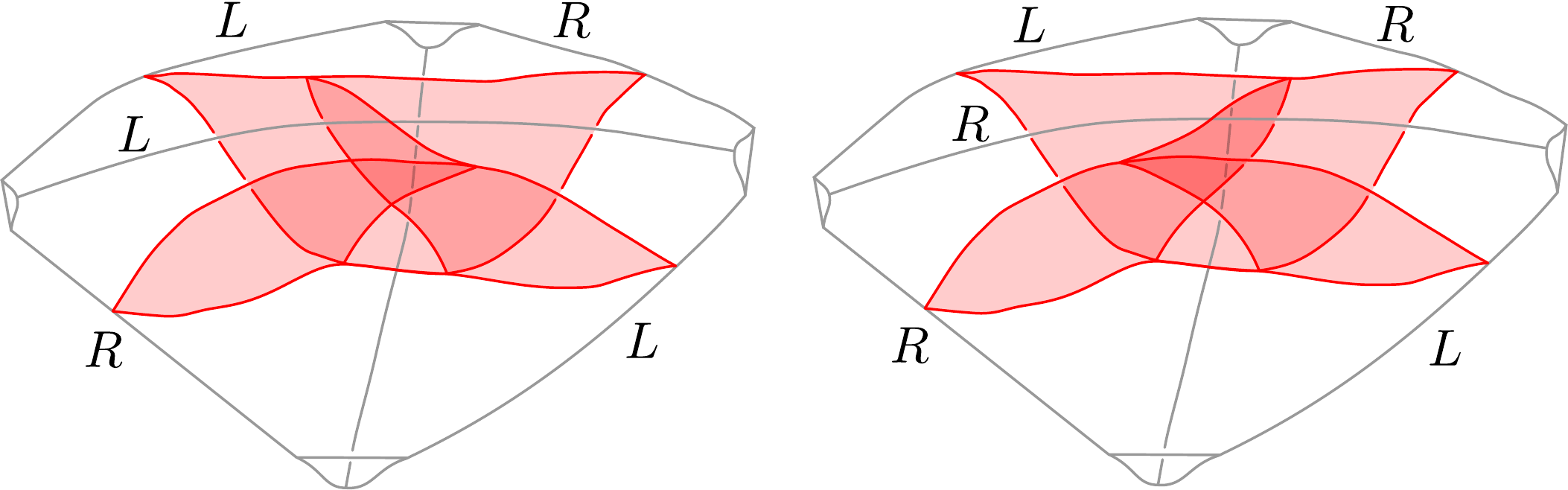}
\caption{The stable branched surface $B^s$ in a tetrahedron $t$ depending on whether the top edge of $t$ is left or right veering. The tetrahedron is a rotated version of the usual veering tetrahedron of \Cref{fig:veer} in order to make the picture easier to draw.}
\label{fig:stablebs}
\end{figure}

\begin{lemma}\label{lem:comp_smooth}
Let $f$ be a face of $\tau$ and let $e$ be its  bottom edge. For either of the two tetrahedra $t$ containing $f$, $f \cap B^s_t$ has its large branch meeting $e$. 
Hence $B^s$ has a well-defined smooth structure making it a branched surface.
\end{lemma}

We note that the smooth structure on $B^s$ is characterized by its intersection with each face, as given in the first sentence of the lemma. 

\begin{proof}
We argue within a single tetrahedron $t$.

If $f$ is a bottom face of $t$, then this is immediate from \Cref{fig:stablebs}. If $f$ is a top face of $t$, then we also see from the figure that the large brach of $f \cap B^s_t$ meets the edge of $f$ with the same veer as the top edge of $t$. Hence, it remains to observe that the bottom edge of tetrahedron $t'$ glued to $t$ along $f$ has the same veer as the top of $t$. This was established in the proof of \Cref{lem:face_implies_tet} (c.f. \Cref{rmk:top_of_fan}).
\end{proof}

We note that since it is topologically a spine for the ideal triangulation $\tau$, $B^s$ is a deformation retract of $M$. This can be seen directly in \Cref{fig:stablebs}.

\smallskip

The $1$-skeleton of $B^s$ is the graph $\Gamma_\tau$ dual to $\tau$ whose edges are directed by the coorientation on the faces of $\tau$. We call $\Gamma_\tau$ (or simply $\Gamma$) the \define{dual graph}. Alternatively, $\Gamma$ can be described as the graph with a vertex interior to each tetrahedron and a directed edge crossing each cooriented face from the vertex in the tetrahedron below the face to the vertex in the tetrahedron above the face. We will always view $\Gamma$ as embedded in $M$.

Note that since $\Gamma$ is the $1$-skeleton of $B^s$ onto which $M$ deformation retracts, $\pi_1(\Gamma) \to \pi_1(M)$ is surjective.

\subsection{The flow graph $\Phi_\tau$}\label{sec:flowgraph}
For a veering triangulation $\tau$, its \define{flow graph} 
is a directed graph denoted by $\Phi_\tau$ (or simply $\Phi$ if there is no chance of confusion) and defined as follows: the vertices of $\Phi$  are in correspondence with $\tau$-edges, and for each tetrahedron $t$ with relation \Cref{eq:tet_rel}, there are directed edges from $\mathbf{b}$ to each of $\mathbf{t}$,  $\mathbf{s_1}$, and $\mathbf{s_2}$. 
That is, there are $\Phi$-edges from the bottom $\tau$-edge of each tetrahedron to its top $\tau$-edge and the two side $\tau$-edges whose veer is opposite that of the top $\tau$-edge. 
Let $\bf e$ be a $\tau$-edge with corresponding $\Phi$-vertex $v$. Then $v$ has outgoing valence 3 and it is a consequence of the lemmas to follow that the incoming valence of $v$ is $n-3$, where $n$ is the degree of $\bf e$.
The name ``flow graph" is motivated by results in 
\cite{veeringpoly2}
showing that in the presence of a certain flow related to $\tau
$, directed cycles in $\Phi_\tau$ correspond in a uniform way to closed orbits of the flow.

\begin{figure}[h]
\centering
\includegraphics[height=1.5in]{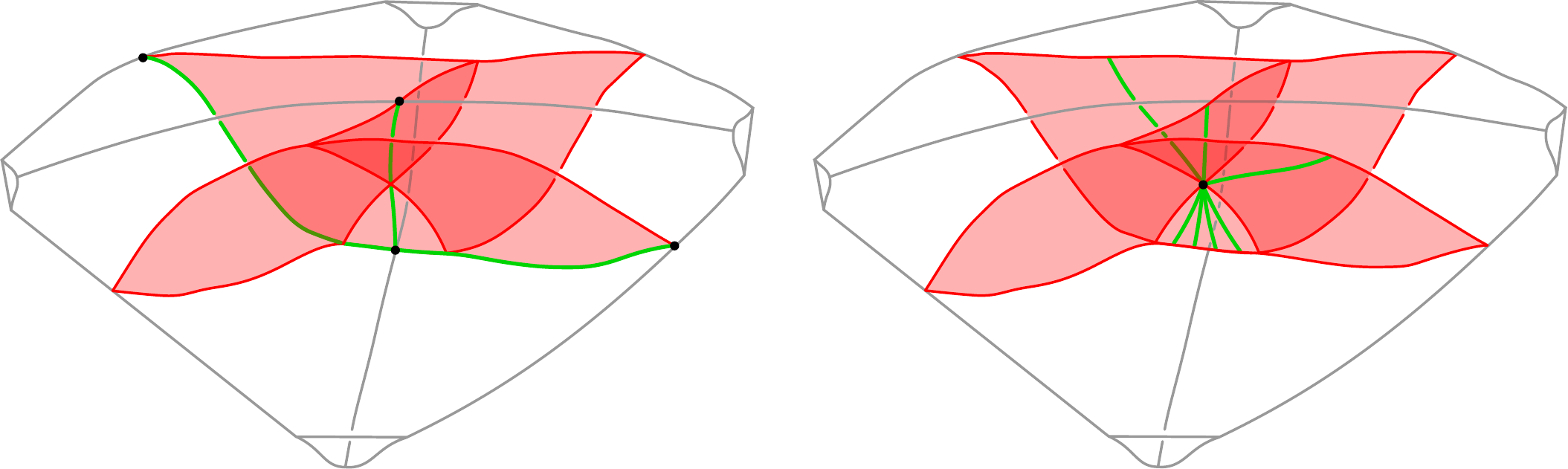} 
\caption{On the left is a local picture of $\Phi$ (green) in \emph{standard position}; $\Phi$-edges are oriented out of the vertex lying in the tetrahedron's bottom $\tau$-edge. On the right we have isotoped $\Phi$ upwards to lie in \emph{dual position}; the vertex which used to lie on the bottom $\tau$-edge is now equal to the $\Gamma$-vertex interior to the tetrahedron.}
\label{fig:flow_posns}
\end{figure}

There is an embedding $i\colon\Phi \hookrightarrow M$  which maps each vertex of $\Phi$ into its corresponding $\tau$-edge, and maps each $\Phi$-edge into either a $\tau$-face or a segment connecting the top and bottom $\tau$-edges of a tetrahedron. This can be done in such a way that the interior of each edge is smoothly embedded in $B^s$, as depicted on the lefthand side of \Cref{fig:flow_posns}. We call this the \textbf{standard position} for $\Phi$. This embedding induces a homomorphism $i_* \colon H_1(\Phi) \to G=H_1(M)/\text{torsion}$ and hence a ring homomorphism $i_* \colon \Z[H_1(\Phi)] \to \Z[G]$. The following lemma shows that these homomorphisms are surjective.

\begin{lemma} \label{lem:surj_hom}
The inclusion $i \colon \Phi \to M$ induces a surjection
\[
i_* \colon \pi_1(\Phi) \twoheadrightarrow \pi_1(M).
\] 
\end{lemma}

Before proving \Cref{lem:surj_hom} we develop the combinatorics of $B^s$ and its branch locus $\Gamma$, and discuss their interactions with $\Phi$. We recall that $\Gamma$ is the directed graph dual to $\tau$.

A line segment, closed curve, or ray which is smoothly immersed in the branch locus of $B^s$ is called a \textbf{branch segment, loop}, or \textbf{ray} respectively. A branch loop will also be called a \textbf{branch cycle} when we wish to think of it as a directed cycle of $\Gamma$. Let $p=(e_1,e_2,e_3,\dots)$ be a directed path in $\Gamma$, and let $v_i$ be the terminal vertex of $e_i$. We say $v_i$ is a \textbf{branching turn} of $p$ if the concatenation $e_i*e_{i+1}$ is a branch segment, and an \textbf{anti-branching turn} otherwise. 

Anti-branching turns will play an important role in \Cref{sec:polys}, and they can be combinatorially distinguished from branching turns in the following way. 

\begin{lemma}[Veering characterization of branching and anti-branching turns]\label{lem:branchingturn}
Let $t$ be a tetrahedron containing a $\Gamma$-vertex $v$. Let $f_1$ be  a bottom face and let $f_2$ be a top face of $t$, and let $e_1, e_2$ be the corresponding $\Gamma$-edges. Let $\bf{s}$ be the unique $\tau$-edge with the property that $f_1$ and $f_2$ are incident to $\bf{s}$ on the same side, with $f_2$ lying immediately above $f_1$. Then $v$ is a branching turn of $(e_1,e_2)$ if and only if $\bf{s}$ and the top edge of $t$ have opposite veer. In this case, when $\Phi$ is in standard position, there is a $\Phi$-edge lying in $f_1$ from the bottom $\tau$-edge of $t$ to $\bf{s}$.
\end{lemma}

The proof of \Cref{lem:branchingturn} follows from inspecting the two cases pictured in \Cref{fig:stablebs}.

Each sector $A$ of $B^s$ is a topological disk pierced by a single $\tau$-edge, as in \Cref{fig:sectors}. The $\Gamma$-edges bounding $A$ are oriented so that exactly one vertex is a source, which we call the \textbf{bottom} of $A$, and one is a sink, which we call the \textbf{top} of $A$. The top and bottom divide the boundary of $A$ into two oriented $\Gamma$-paths called \textbf{sides}. Each side has at least two $\Gamma$-edges because the $\tau$-edge piercing $A$ has a nonempty fan on each side.
The following lemma says that if you remove the last edge in any side of any sector of $B^s$, the resulting path is a branch segment, and that the entire side is never a branch segment.

\begin{lemma}[Sectors and turns]\label{lem:sectors_turns}
Let $A$ be a sector of $B^s$ and let $p$ be a side of $A$ considered as a directed path in $\Gamma$ from the bottom to the top of $A$.  The last turn of $p$ is anti-branching, and all other  turns are branching.
\end{lemma}

\begin{proof}
Let $\bf e$ be the $\tau$-edge piercing $A$. The side $p$ corresponds to a side of a fan of $\bf e$, and each turn of $p$ lies interior to one of the corresponding tetrahedra. As we observed in \Cref{rmk:top_of_fan}, the top $\tau$-edge of one of these tetrahedra will have the same veer as $\bf{e}$ if and only if it is topmost in the side of the fan.
By \Cref{lem:branchingturn}, the lemma is proved.
\end{proof}

By isotoping $\Phi$ upward from standard position we can arrange for the vertex sets of $\Gamma$ and $\Phi$ to coincide, and for $\Phi$ to still lie in $B^s$ with the interior of each $\Phi$-edge smoothly embedded. This isotopy pushes a $\Phi$-vertex in a $\tau$-edge $e$ onto the $\Gamma$-vertex interior to the tetrahedron above $e$. See the righthand side of \Cref{fig:flow_posns}. 
We call this the \textbf{dual position} for $\Phi$. Note that in dual position, the interior of each $\Phi$-edge is disjoint from $\Gamma$ and positively transverse to $\tau^{(2)}$.  
\emph{From this point forward we assume $\Phi$ is in dual position unless otherwise stated}.

The following lemma describes the intersection of $\Phi$ with a sector of $B^s$.

\begin{lemma}[Sectors and $\Phi$]\label{fact:flow_edge}
Let $A$ and $p$ be as in \Cref{lem:sectors_turns}. Then if $e$ is a $\Gamma$-edge in $p$, then the bottom vertex of $e$ is connected by a $\Phi$-edge in $A$ to the top of $A$ if and only if $e$ is not topmost in $A$.
\end{lemma}

\begin{proof}
It is clear that there is a $\Phi$-edge running from the bottom to the top of $A$. Further, \Cref{lem:branchingturn} tells us that there will be a $\Phi$-edge in $A$ from a turn of $p$ to the top of $A$ if and only if the turn is branching. Applying \Cref{lem:sectors_turns} finishes the proof.
\end{proof}

\begin{figure}[h]
\centering
\includegraphics[height=1.75in]{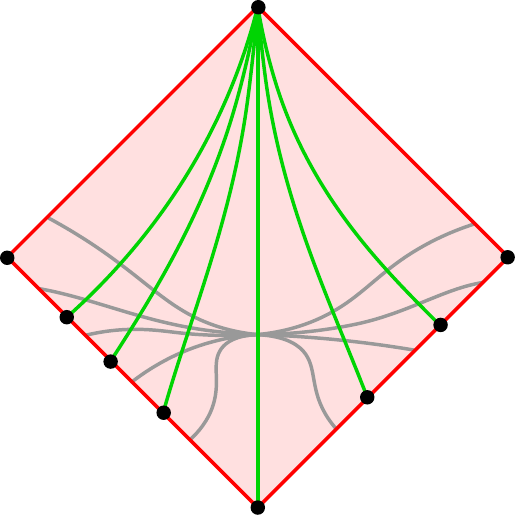}
\caption{A sector of $B^s$ with $\Phi$ in dual position. The intersection with $\tau^{(2)}$ is shown in gray and $\Phi$-edges are green.}
\label{fig:sectors}
\end{figure}

The picture described by \Cref{fact:flow_edge} is shown in \Cref{fig:sectors}. Note that some sectors may have sides with 2 edges. However, we see in the following argument that for every $\Gamma$-edge $e$, there is a sector $A$ of $B^s$ such that $e$ lies in a side of $A$ with at least 3 edges.

\begin{proof}[Proof of \Cref{lem:surj_hom}]
Let $e$ be a $\Gamma$-edge with terminal point $v$. We claim there is a sector $S_e$ and a side $p$ of $S_e$ such that $e$ lies in $p$ and
neither endpoint of $e$ is the last turn of $p$.
 This can be seen by considering the branching around $v$: of the (locally) 6 sectors meeting at $v$, the sector $S_e$ is determined uniquely by the property that its boundary contains both $e$ and the next edge in the branch cycle containing $e$ (the reader may find it helpful to consult \Cref{fig:stablebs}). Since $v$ is thus a branching turn for a side of $S_e$, we have proved our claim by \Cref{lem:sectors_turns}. 
It follows from  \Cref{fact:flow_edge} that each $\Gamma$-edge cobounds a disk in $B^s$ with two $\Phi$-edges. Since $\pi_1(\Gamma)$ surjects $\pi_1(M)$,
this proves the lemma.
\end{proof}

\begin{theorem} \label{th:veering_digraph}
Let $(M,\tau)$ be a veering triangulation and $\Phi_\tau$ its flow graph. Then 
\[
V_\tau =  i_*(P_{\Phi}) = 1 + \sum_{\sigma} (-1)^{|\sigma|} i_*(\sigma),
\]
where $P_{\Phi}$ is the Perron polynomial of $\Phi_\tau$.
\end{theorem}

In other words, the veering polynomial is obtained from the Perron polynomial of its flow graph by replacing its directed cycles with the corresponding homology classes in $M$. We warn the reader that some of these classes may be trivial, as is the case when $\tau$ represents a top dimensional face that is not fibered. This is discussed in \Cref{sec:conedimension}. 

\begin{proof}
Let $\aM\to M$ be the universal free abelian cover of $M$, which has deck group $G$, and let $\wt\Phi$ be the preimage of $\Phi$. Since $\Phi\to M$ is $\pi_1$-surjective by \Cref{lem:surj_hom}, the restriction $\wt \Phi\to \Phi$ is the covering map associated to 
$\ker(i_*\circ \text{ab}\colon \pi_1(\Phi)\to G)$.

Choose a lift of each $\Phi$-vertex to $\aM$, determining a presentation for $V(\wt\Phi_{i_*})$ and thus a matrix $L_{\Phi,i_*}$ as in \cref{eq:vertex_mod}. 
By \Cref{prop:image_poly}, $i_*(P_{\Phi})$ is equal to $P_{\Phi,i_*} = \det (L_{\Phi,i_*})$. 
By the correspondence between $\Phi$-vertices and $\tau$-edges, the chosen lifts of $\Phi$-vertices give a family of lifted $\tau$-edges and thus a presentation for $\mathcal E(\wt \tau)$ with an associated matrix $L$ as in \cref{eq:edge_mod}. For our choices of generators, the definition of $\Phi$ gives that these two presentation matrices are equal. Thus
\[
V_\tau =\det(L)=\det(L _{\Phi,i_*})=  i_*(P_{\Phi}),
\]
completing the proof.
\end{proof}

\section{Carried classes, homology directions, and the Thurston norm}
\label{sec:cones}

Fix a veering triangulation $\tau$ of $M$, let $\Gamma = \Gamma_\tau$ be its dual graph,
and let $\Phi = \Phi_\tau$ denote its flow graph. In this section, we first show that various cones in $H_1(M; \R)$ naturally associated to these directed graphs are equal (\Cref{th:cones_equal}). 
Then we show that the dual of these cones in $H^1(M; \R) = H_2(M, \partial M; \R)$ is precisely the cone over a face of the Thurston norm unit ball (\Cref{th:whole_face}) and describe a connection to the veering polynomial (\Cref{th:omni}).

\subsection{Homology directions and carried classes}\label{sec:conedefs}

Let $V$ be a finite dimensional real vector space and let $C\subset V$ be a \define{convex polyhedral cone}, which by definition is the nonnegative span of finitely many vectors in $V$. The \textbf{dimension} of $C$ is the dimension of the subspace generated by $C$. 
The \textbf{dual cone} to $C$ is 
\[
C^\vee=\{u\in V^*\mid u(v)\ge0 \text{ for all } v\in C\}\subset V^*.
\]
One sees that $(C^{\vee})^\vee=C$. Let $d(C)$ be the dimension of the largest linear subspace contained in $C$. Then the dimension of $C^\vee$ is $\dim(V)-d(C)$. If $d(C)=0$, we say $C$ is  \textbf{strongly convex}. Thus if $C$ is a top-dimensional strongly convex cone (e.g. a cone over a fibered face of $B_x(M)$), then $C^\vee$ is also a top-dimensional strongly convex cone. For a reference on convex polyhedral cones see \cite[Section 1.2]{Ful93}.

Let $D$ be a directed graph embedded in $M$. Define  
$\cone_1(D)\subset H_1(M ; \R)$ to be the nonnegative span of the images of directed cycles in $D$ under inclusion. We will write $\cone_1^\vee(D):=(\cone_1(D))^\vee$. For the dual graph $\Gamma$, we call $\cone_1(\Gamma)$ the \textbf{cone of homology directions} of $\tau$. Note that any directed cycle in $\Gamma$ gives a closed curve in $M$ which is positively transverse to $\tau^{(2)}$, and conversely any closed curve transverse to $\tau^{(2)}$ is homotopic to a directed cycle in $\Gamma$. Hence the cone of homology directions of $\tau$ is equal to the cone generated by all closed curves which are positively transverse to $\tau^{(2)}$ at each point of intersection.

We define the \textbf{cone of carried classes} of $\tau$, which we denote $\cone_2(\tau)\subset H_2(M,\partial M; \R)$, to be the cone of classes carried by the branched surface $\tau^{(2)}$.

The first goal of this section is to prove the following theorem relating the above cones.

\begin{theorem}[Cones]\label{th:cones_equal}
For any veering triangulation $\tau$ of $M$ with dual graph $\Gamma$ and flow graph $\Phi$, we have $\cone_1(\Gamma) = \cone_1(\Phi)$. Moreover,
\[
\cone_2(\tau) = \cone_1^\vee(\Gamma) = \cone_1^\vee(\Phi),
\]
where we have identified $H^1(M;\R) = H_2(M, \partial M; \R)$ via Lefschetz duality.
\end{theorem}

We now turn to the combinatorial definitions and observations we need to prove \Cref{th:cones_equal}.

\subsubsection{Relating $\Gamma$, $\Phi$, and $B^s$}
\label{sec:relating}

An oriented path, ray, or cycle in the dual graph $\Gamma$ is called a \define{dual path}, \define{ray}, or \define{cycle}, respectively.  Similarly, we call an oriented path, ray, or cycle in the flow graph $\Phi$ a \define{flow path}, \define{ray}, or \define{cycle}.

As a first step of the proof of \Cref{th:cones_equal}, we will prove \Cref{prop_samecone} which asserts that the dual cycles and flow cycles generate the same cone in $H_1(M;\R)$, i.e. $\cone_1(\Gamma) = \cone_1(\Phi)$. While it is clear from the transversality of $\Phi$ and $\tau^{(2)}$ that each flow cycle is homotopic to a dual cycle, the converse is not necessarily true; it may be necessary to square a dual cycle before it is homotopic to a flow cycle.
 This is not apparent from the picture we have developed so far of the relationship between $\Gamma$ and $\Phi$ (\Cref{lem:sectors_turns} and \Cref{fact:flow_edge}) which has been purely local in nature.
Hence we continue our discussion from \Cref{sec:flowgraph} of the interplay between $B^s$, $\Gamma$, and $\Phi$, broadening our scope beyond a single sector. It will be convenient to work in the universal cover $\wt M$ of $M$. Let $\wt \Phi$, $\wt \Gamma$, and $\wt B^s$ denote the lifts to $\uM$ of $\Phi$, $\Gamma$, and $B^s$ respectively.

Since $B^s$ is topologically the dual 2-complex to $\tau$, each component $U$ of $M\ssm B^s$ is homeomorphic to $T_U\times [0,1)$, where $T_U$ is the component of $\partial M$ contained in $U$. We call $U$ a \define{tube} of $B^s$. Let $\wt U$ denote a single lift of $U$ to $\uM$, which we also call a \define{tube} of $\wt B^s$. There is a collection of branch lines which are entirely contained in $\partial \wt U$ which we call \textbf{the branch lines of $\wt U$}. 
Each connected component of $(\partial\wt U\cap \wt B^s)\ssm\{\text{branch lines of } \wt U\}$ is called a \define{band} of $\wt U$. 
The image of a band of $\wt U$ under the covering projection is an immersed annulus in $M$ which we call a \textbf{band} of $U$.

\begin{figure}[h]
\centering
\includegraphics[height=1.5in]{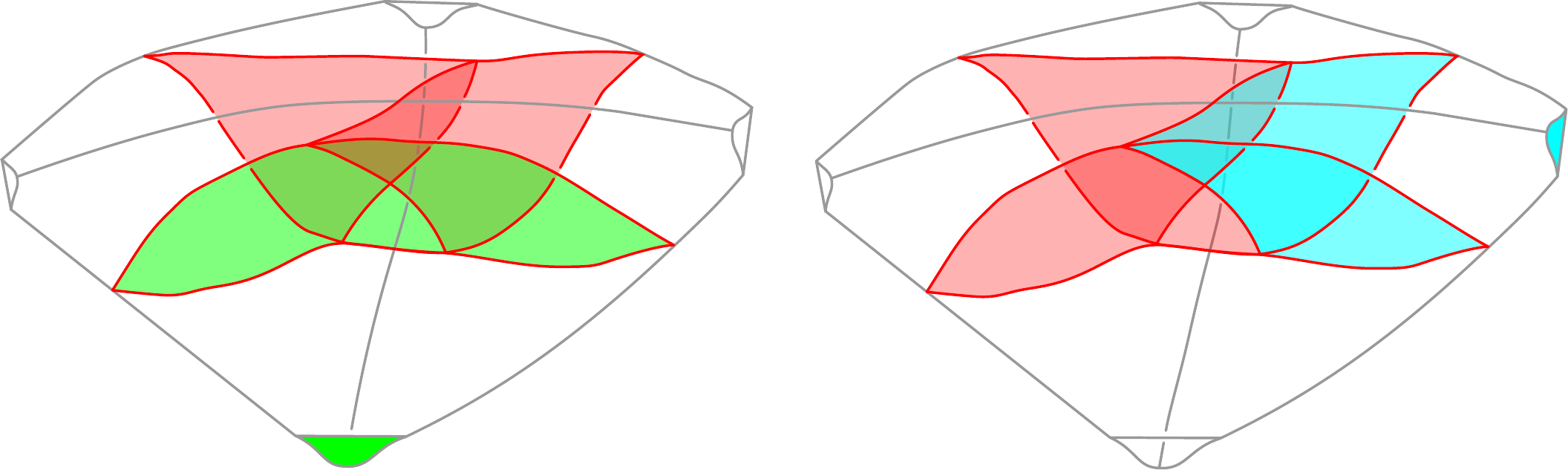}
\caption{Each tip of a tetrahedron is a flat triangle that determines a cell of the tesselation of the corresponding tube's boundary. We show the picture for a downward (left) and upward (right) flat triangles.}
\label{fig_tessident}
\end{figure}

\begin{lemma}[$\tau^{(2)}$ inside a tube]\label{lem:2skeltube}
 Let $\wt U$ be a tube of $\wt B^s$ covering a tube $U$ of $B^s$. Let $T$ be the boundary component of $M$ contained in $U$, and let $\wt T$ be the lift of $T$ contained in $\wt U$.
The intersection $\wt\tau^{(2)}\cap \closure (\wt U)$ can be identified with $(\wt\tau ^{(2)}\cap \wt T)\times[0,1]$ under the homeomorphism $\closure (\wt U )\cong \wt T\times[0,1]$. This gives an identification of the tesselation $\tau^{(2)}\cap\partial \wt U$ with the boundary tesselation of $\wt T$.
\end{lemma}

\begin{proof}
This is a consequence of the duality of $\tau^{(2)}$ and $B^s$. The identification of the two tesselations can be visualized as in \Cref{fig_tessident}.
\end{proof}

We will now use the structure of this tesselation on $\partial \wt U$ to reconstruct the intersections of $\wt \Gamma$ and $\wt\Phi$ with $\partial \wt U$. In summary, this can be done as follows: $\wt\Gamma\cap \partial \wt U$ is the dual graph to the tesselation $\wt\tau^{(2)}\cap \partial\wt U$, and each complementary component of $\wt\Gamma\cap \partial \wt U$ is a sector of $\wt B^s$. Thus the data of which $\Gamma$ turns are branching/anti-branching is determined by \Cref{lem:sectors_turns},
and the combinatorics of $\wt \Phi\cap \partial \wt U$ are determined by \Cref{fact:flow_edge}. 
The complete picture is shown in \Cref{tessandphi}; we emphasize that each individual sector in the picture is exactly as in \Cref{fig:sectors}. We justify this picture and add more detail in the lemmas to follow.

 \begin{figure}[h]
\includegraphics[height = 2.95in]{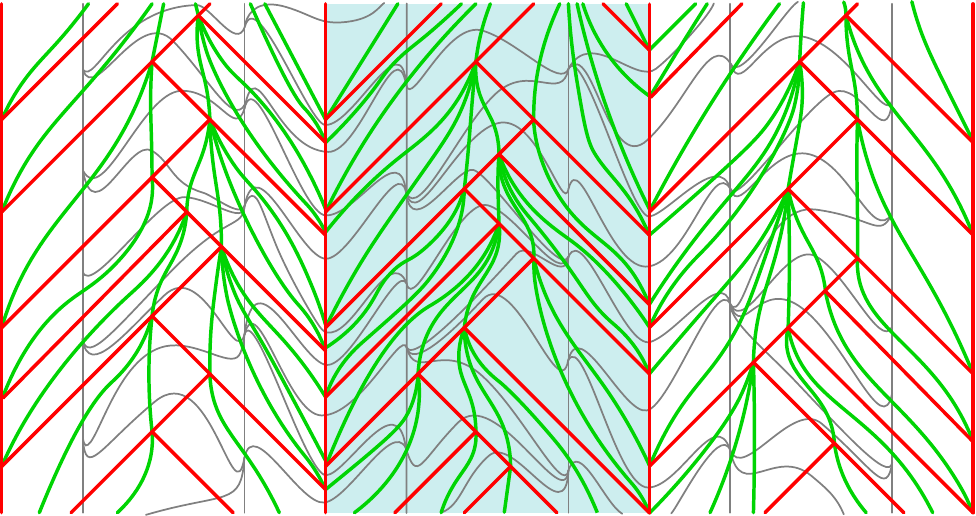}
\caption{Part of the boundary of a tube in $\uM$ and its intersection with $\partial \wt\tau^{(2)}$ (gray), $\wt \Gamma$ (red), and $\wt \Phi$ (green). One band has been shaded light blue.}
\label{tessandphi}
\end{figure}

Taking advantage of the identification between the two tesselations as in \Cref{lem:2skeltube}, we will speak of \textbf{ladders}, \textbf{ladderpoles}, and \textbf{flat triangles} on the boundaries of tubes in $\wt M$.

\begin{lemma}[$\Gamma$ on boundary of tube]\label{lem:gammabandcombinatorics}
 Let $\wt U, U, \wt T, T$ be as above. Then
\begin{enumerate}[label=(\roman*)]  
\item $\wt \Gamma\cap \partial\wt U$ is topologically the dual graph to $\wt\tau^{(2)}\cap \partial \wt U$, and the edge orientations of the former agree with the coorientations of the latter,
\item each component of $\partial\wt U- \wt \Gamma$ is a sector of $\wt B^s$,
\item the top of each of these sectors lies in a downward triangle, the bottom of each sector lies in an upward triangle, and 
\item each branch line of $\wt U$ bisects an upward ladder, and each upward ladder is bisected by a single branch line.  
\end{enumerate}
\end{lemma}

\begin{proof}
Statement $(i)$ is a direct consequence of $B^s$ being the 2-complex dual to $\tau^{(2)}$ and $\Gamma$ being its 1-skeleton. Statement $(ii)$ follows immediately from the fact that $\wt\Gamma$ is the branching locus of $\wt B^s$. Each sector contains exactly one vertex of the tesselation and the top of each sector lies in the flat triangle lying atop that vertex. This flat triangle must be downward, giving $(iii)$.

Let $\lambda$ be an upward ladder in $\partial \wt U$. By $(i)$, there is a unique $\wt\Gamma$-path $\gamma$ lying in $\lambda$. Because any two consecutive flat triangles visited by $\gamma$ share a vertex in the tesselation of $\partial \wt U$, any two consecutive edges $e,f$ of $\gamma$ lie in the boundary of some $\wt B^s$-sector $s$. By $(iii)$, neither $e$ nor $f$ is topmost in a side of $s$, so they define a branching turn by \Cref{lem:sectors_turns}. Therefore each upward ladder contains a branch line of $\wt U$. 

It remains to show that no other biinfinite $\wt\Gamma$-path is a branch line. Since any two branch lines must be disjoint, it suffices to show that a downward ladder $\lambda'$ in $\partial\wt U$ does not define a branch line. Let $t$ be a flat triangle in $\lambda'$ and let $v$ be the cusp of $t$ such that the two sides of $t$ meeting $v$ are rungs of $\lambda'$.  Further suppose that $t$ is the last triangle in
in the fan of $v$ that is contained in $\lambda'$. Then the triangle in $\lambda'$ directly above $t$ is a downward triangle having $v$ at its bottom. Hence, $t$ is the topmost flat triangle in its side of the fan for $v$. If $\gamma'$ is the $\wt \Gamma$-path determined by $\lambda'$, then we note that $\gamma'$ crosses every rung of $\lambda'$. By \Cref{lem:sectors_turns}, the turn of $\gamma'$ at the vertex in $t$ is anti-branching. This completes the proof.
\end{proof}

Thus each band $b$ of $\wt U$ contains two halves of upward ladders meeting $\partial b$, and one downward ladder interior to $b$.

Since $\Phi$ is in dual position, the vertex sets of $\wt \Phi$ and $\wt \Gamma$ are the same. The intersection of $\wt\Phi$  with $\partial \wt U$ is therefore a directed graph whose vertices correspond to the flat triangles of the tessellation. We describe its edges in the following lemma.

\begin{lemma}[$\Phi$ on boundary of tube]\label{lem:bandflowedges}
Fix $U, \wt U, T, \wt T$ as above. Identifying the vertices of $\wt \Phi\cap \partial \wt U$ with the flat triangles in the tessellation of $\partial \wt U$, the edges of $\wt \Phi\cap \partial \wt U$ have the following properties:
\begin{enumerate}[label=(\roman*)]
\item Each downward triangle $t$ has a single outgoing edge to a triangle $t'$ which is downward, lies in the same ladder as $t$, and is the endpoint of a $\wt \Gamma$-path starting at $t$.
\item Each upward triangle has two outgoing edges, which have endpoints in different downward ladders. 
\end{enumerate}
\end{lemma}

\begin{proof}
Let $t$ be a downward flat triangle in $\partial \wt U$. Then $t$ is topmost in the fan corresponding to exactly one of its $0$-vertices (see \Cref{rmk:top_of_fan} or consider e.g. \Cref{fig_tessident}).
 Let $v$ be the 0-vertex of $t$ such that $t$ is \emph{not} topmost in the fan corresponding to $v$, and let $t'$ be the downward flat triangle having $v$ as a 0-vertex. \Cref{fact:flow_edge} and the picture we have developed in \Cref{lem:2skeltube} and \Cref{lem:gammabandcombinatorics} gives that $t$ has a single outgoing edge, and its endpoint is in $t'$. 

Now let $t$ be an upward flat triangle in $\partial \wt U$. It again follows from \Cref{fact:flow_edge}, \Cref{lem:2skeltube}, and \Cref{lem:gammabandcombinatorics} that $t$ has 2 outgoing edges and that they connect to the downward triangle sharing $\pi$-vertex with with $t$, and to the downward triangle whose $\pi$-vertex is the 0-vertex of $t$ for which $t$ is not topmost in the corresponding fan. One sees that these two downward flat triangles lie in separate upward ladders.
\end{proof}

For a visual summary of \Cref{lem:gammabandcombinatorics} and \Cref{lem:bandflowedges}, see \Cref{tessandphi}. As previously remarked, one can recover the picture from only the tesselation on the boundary of a tube by drawing the directed graph dual to the tesselation and applying \Cref{lem:sectors_turns} and \Cref{fact:flow_edge}.

If $\ell$ is a branch line containing points $p$ and $q$, then we say $q$ \define{lies below} $p$ if there is an oriented branch segment in $\ell$ from $q$ to $p$, where we include the empty segment (so $p$ lies below itself).

\begin{lemma}[$\Phi$ in a band] \label{lem:bandbehavior}
The restriction of $\wt \Phi$ to a band $b$ in $\wt B^s$ has the following properties.
\begin{enumerate}[label=(\roman*)]
\item For every $\wt\Phi$-vertex $q$ lying in $b$, there is a unique $\wt\Phi$-ray $\rho_b(q)$ in $\wt\Phi|_b$ with $q$ as its initial vertex. 

\item For each $\wt\Gamma$-vertex $p$ interior to $b$, the union of the portions of the 2 branch lines through $p$ lying below $p$ and in $b$ divide $b$ into 2 components, $b_p^+$ and $b_p^-$.

\item If $q$ lies in $b_p^-$ on a branch line of $\partial b$, the ray $\rho_b(q)$ intersects a branch segment in $b$ containing $p$ at a point below or equal to $p$.
\end{enumerate}
\end{lemma}

A diagram of the situation in \Cref{lem:bandbehavior} is shown in \Cref{fig_bandray}.

\begin{figure}[h]
\includegraphics[height = 2in]{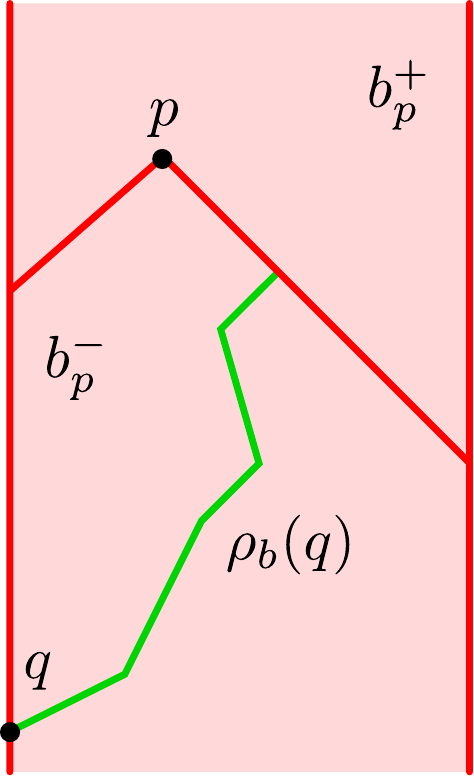}
\caption{An illustration of \Cref{lem:bandbehavior}.}
\label{fig_bandray}
\end{figure}

\begin{proof}
Statement $(i)$ follows from \Cref{lem:bandflowedges}.
Statement $(ii)$ follows from \Cref{lem:gammabandcombinatorics} and our understanding of which turns in the boundary of a sector are branching from \Cref{lem:sectors_turns}.
Statement $(iii)$ follows from $(i)$ and $(ii)$: the ray $\rho_b(q)$ eventually must cross from  $b_p^-$ to $b_p^+$, and it must do so at a $\wt\Gamma$-vertex lying below $p$.
\end{proof}

\subsubsection{Dual cycles, flow cycles, and carried classes}

Any dual cycle $c$ in $\Gamma$ has a decomposition into branch segments meeting at anti-branching ($AB$-) turns. The following lemma shows that the parity of the number of anti-branching segments composing the cycle is an invariant of its homotopy class.

\begin{lemma}\label{cor:push_parity}
If dual cycles $c_1$ and $c_2$ in $\Gamma$ are homotopic in $M$, then the number of $AB$-turns in $c_1$ has the same parity as the number of $AB$-turns in $c_2$. 
\end{lemma}

\begin{proof}
Since $M$ deformation retracts to $B^s$ and $\Gamma$ is the $1$-skeleton of $B^s$, it suffices to build a homomorphism $\mc Z \colon \pi_1(B^s) \to \ZZ/2$ such that if $c$ is dual cycle, then $\mc Z(c)$ is the number of $AB$ turns mod $2$. 

The branched surface $B^s$ is the base space of a vector bundle $E\to B^s$ defined by taking the tangent plane at each point of $B^s$. For any loop $\gamma$ in $B^s$, this pulls back to a plane bundle over a circle. If the pullback bundle is orientable set $\mc Z(c) =0$, and set $\mc Z(c)=1$ otherwise. It is a standard fact that $\mc Z$ defines a homomorphism. In fact $\mc Z$ is equal to the map $\pi_1 (B^s)\to \ZZ/2$ induced by the the first Stiefel-Whitney class of $E\to B^s$, see e.g. \cite[Chapter 3]{HatchVB}.

Now let $c$ be a dual cycle. We will show that 
\[
\mc Z(c) \equiv \# \{AB \text{-turns of }c \}\pmod 2.
\]

\begin{figure}[htbp]
\begin{center}
\includegraphics[height=2in]{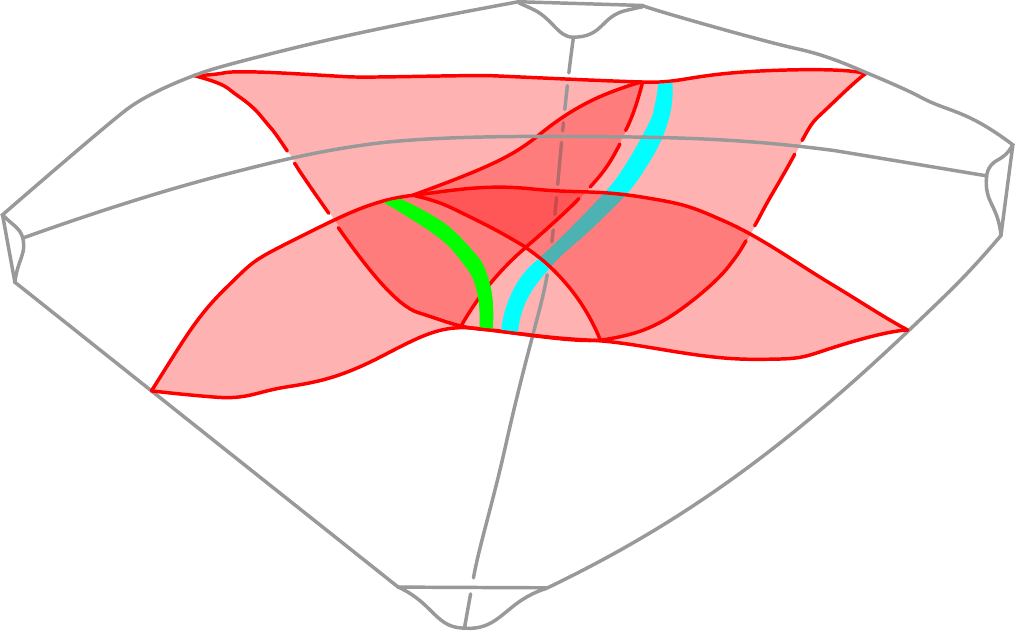}
\caption{When a dual cycle makes an anti-branching turn, we can homotope it to lie in a strip like the green one shown above. The blue strip corresponds to a branching turn. }
\label{fig:parity}
\end{center}
\end{figure}

For this, observe that each time $c$ passes through a tetrahedron it makes either a branching or anti-branching turn, and this portion of $c$ can be homotoped to lie in one of the two strips shown in \Cref{fig:parity}. From the picture we see that these strips glue together to give an annulus if and only if the number of anti-branching turns is even. 
Indeed, at any point along the path in the branch locus, a local orientation of the plane field is obtained from the direction of the path and the rule that the branching is ``on the left.'' Now the $AB$ path in \Cref{fig:parity} enters the tetrahedron where the branching is on one side and exits the tetrahedron where the branching is on the other side. Hence the orientation of the plane flips, which is to say it disagrees with the orientation carried continuously from the entry point. Since the opposite is true for the branching turn, this gives the desired parity property for the loop. 
\end{proof}

The next proposition establishes half of \Cref{th:cones_equal}. Although our emphasis here is on the equality of the cones, the stronger fact that for any positive transversal of $\tau^{(2)}$ either it or its square is homotopic to a flow cycle will be central in \cite{veeringpoly2}.

\begin{proposition}\label{prop_samecone}
For each dual cycle $c$ of $\Gamma$, there is a flow cycle $f$ of $\Phi$ such that $f$ is homotopic to $c^i$ for $i=1,2$. Moreover, if $c$ is composed of an even number of branch segments, then $f$ is homotopic to $c$. 

Hence the flow graph and the dual graph determine the same cones in $H_1(M,\R)$; that is
\[
\cone_1(\Gamma) = \cone_1(\Phi).
\]
\end{proposition}

\begin{proof}
We first explain that every branch curve is homotopic to a flow cycle. Let $\gamma$ be a branch curve in $B^s$. Then $\gamma$ determines a tube $T_\gamma$ of $M-B^s$. Let $b$ be a band of $T_\gamma$. The core curve of $b$ is homotopic to $\gamma$. 
By \Cref{lem:bandbehavior}, there is a $\Phi$-ray lying exclusively in $b$, and it must be eventually periodic. Hence there is a $\Phi$-cycle $c$ in $b$, which must then be homotopic to $\gamma^i$ for $i\ge1$. The fact that $i=1$ is a consequence of the uniqueness part of \Cref{lem:bandbehavior}, item (i).

Next, let $c$ be a dual cycle which is not a branch curve, and let $\wt c$ be a lift to $\uM$. 
Let $g \in \pi_1 (M)$ be the deck transformation that generates the cyclic subgroup stabilizing $\wt c$. Further suppose that $g$ translates $\wt c$ in its positive direction.
Express $\wt c$ as a concatenation of $\wt \Gamma$-edges $(\dots,e_{-2}, e_{-1}, e_0, e_1,e_2,\dots)$ and let $p_i$ be the terminal point of $e_i$.

We build a $\wt\Phi$-ray inductively. 
There are exactly two branch lines through $p_0$. Let $q_0$ be any point on either of these branch lines below $p_0$. Suppose $q_i$ is on a branch line below $p_i$. If $p_{i+1}$ lies on the same branch line containing $q_i$ and $p_i$ (including the case where $q_i = p_i$), let $q_{i+1}=q_i$. If $p_{i+1}$ does not lie in the same branch line as $p_i$ and $q_{i}$, let $b$ be the band 
containing the branch line through $p_i$ and $q_{i}$ in its boundary and with $p_{i+1}$ in its interior.
Then  $q_{i}\in b_{p_{i+1}}^-$, so as in \Cref{lem:bandbehavior} the (unique) ray $\rho_b(q_i)$ intersects a branch line containing $p_{i+1}$ at a point $q_{i+1}$ below or equal to $p_{i+1}$. Let $f_i$ be the $\wt \Phi$-segment from $q_i$ to $q_{i+1}$ traversed by $\rho_b(q_i)$. 
(These two cases are depicted in the first two images in \Cref{fig:to_the_left} where the points $q_i', q_{i+1}'$ can be ignored for now.)
We will show that the concatenation $\wt f = (f_1,f_2,f_3, \ldots)$ of these segments form a preperiodic ray in $\wt \Phi$.

Let $\ell_i$ be the branch line containing $e_i$, and let $b_i$ be the band of $\wt B^s$ determined by the property that it contains $\ell_i$ in its boundary and contains the first edge of $\wt c$ after $e_i$ not lying in $\ell_i$. As a comprehension check, we note that if $e_i$ and $e_j$ are contained in the same branch line, then $b_i=b_j$.

Fix some $j \ge 0$ such that $\ell_j\ne \ell_{j+1}$. Consider $p_j$, which as a reminder is the terminal point of $e_j$ and the initial point of $e_{j+1}$. By construction, $q_j$ lies below $p_j$ on $\ell_j$ or $\ell_{j+1}$. 

\begin{itemize}
\item First, suppose that $q_j$ lies below $p_j$ on $\ell_{j+1}$ (this includes the case where $q_j=p_j$). Let $e_k$ be the last edge of $\wt c$ contained in $\ell_j$. Then $q_j=q_{j+1}=\cdots=q_k$, and $q_{k+1}$ is equal to either $p_{k+1}$ or one of finitely many points below $p_{k+1}$ on a branch segment in $b_k$.
\item Next, suppose first that $q_j$ lies below $p_j$ on $\ell_j$. Then $q_{j+1}$ lies in $b_j$ and is equal either to $p_{j+1}$ or one of finitely many points below $p_{j+1}$ on a branch segment in $b_j$. 
\end{itemize}

Recall that the deck transformation $g$ stabilizes $\wt c$. The discussion above shows that there are finitely many vertices $r_1,\dots, r_m$ such that, for each $i>0$, $\wt f$ must pass through one of $g^i(r_1),\dots, g^i(r_m)$. Moreover, since the construction shows that each subray of $\wt f$ is determined by any initial vertex, it follows that the projection $f$ of $\wt f$ to $M$ is eventually periodic, giving a $\Phi$-cycle that is homotopic to a positive integer multiple of $c$.

By truncating an initial segment of $\wt f$ we can assume $\wt f$ is periodic under some power of $g$. More precisely, there exist $k,\ell$ such that $g^k(q_0)=q_\ell$. By construction, $\wt f$ is determined by any of its vertices and so for any $i \in \ZZ$, either $\wt f$ has no common vertices with  $g^i(\wt f)$ or the two are equal. To bound the period of $\wt f$ we first need to discuss some extra structure associated to the branch lines through $\wt c$.

First, label each $\ell_i$ with either an $L$ or $R$ such that if $\ell_j$ and $\ell_{j+1}$ are adjacent (i.e. intersecting and not equal) branch lines along $c$, as in the above construction, then they receive opposite labels. Such a labeling is uniquely determined by giving $\ell_1$, the branch line containing the edge $e_1$, the label $L$. 
Note that this labeling is well-defined: if $\wt c$ leaves a branch line and then returns to it, it does so only after an even number of anti-branching turns by \Cref{cor:push_parity}.

If $p$ is a vertex of $\wt c$ and $\ell^L$ and $\ell^R$ are the branch lines through $p$ with indicated labels, we define an order on the vertices lying below $p$ on $\ell^R$ and $\ell^L$. If $q_1,q_2$ are vertices on these lines which lie below $p$, then we say that $q_1$ is \define{to the left of} $q_2$ (at $p$) if any of the following hold:
\begin{itemize}
\item $q_1$ lies in $\ell^L$ and $q_2$ lies in $\ell^R$,
\item $q_1$ and $q_2$ lie in $\ell^L$ where $q_1$ lies below $q_2$, or
\item $q_1$ and $q_2$ lie in $\ell^R$ where $q_2$ lies below $q_1$.
\end{itemize}

Now if $c$ is composed an an even number of branch segments (i.e. $c$ has an even number of anti-branching turns), then $g$ preserves the $L/R$ labeling on the branch lines through $\wt c$, and therefore preserve the order relation defined above.  For the remainder of the argument, we assume this to be the case. Otherwise, after replacing $c$ (and $g$) by its square, the same argument applies.

Let $p$ be any vertex of $\wt c$ and let $q$ be a vertex of $\wt f$ such that $q$ lies below $p$ in a branch line $\ell$, as in the construction of $\wt f$. (If $q = p$, then take $\ell$ to be the branch line containing the edge of $\wt c$ with initial vertex $p$. Suppose that $\ell$ is labeled by $L$. Set $\wt f' = g^{-1}(\wt f)$. If $\wt f = \wt f'$, then the proof is complete. So assume that $\wt f \ne \wt f'$ and so these flow lines have no vertices in common.

By construction, there is also a vertex $q'$ of $\wt f'$ such that $q'$ lies below $p$ on some branch line. Since $q$ and $q'$ are not equal, one lies strictly to the left of the other. For the sake of argument, suppose $q$ lies strictly to the left of $q'$. We claim that this implies that for any other vertex $x$ of $\wt c$, the vertex of $\wt f$ below $x$ on some branch lines lies strictly to the left of the vertex of $\wt f$ below $x$ on some branch line. (As shorthand, we say that $\wt f$ lies to the left of $\wt f'$ at $x$.)
It suffices to prove the claim for vertices after $p$ along $\wt c$ and we do so by induction.

\begin{figure} 
\centering
\includegraphics[width=4.5in]{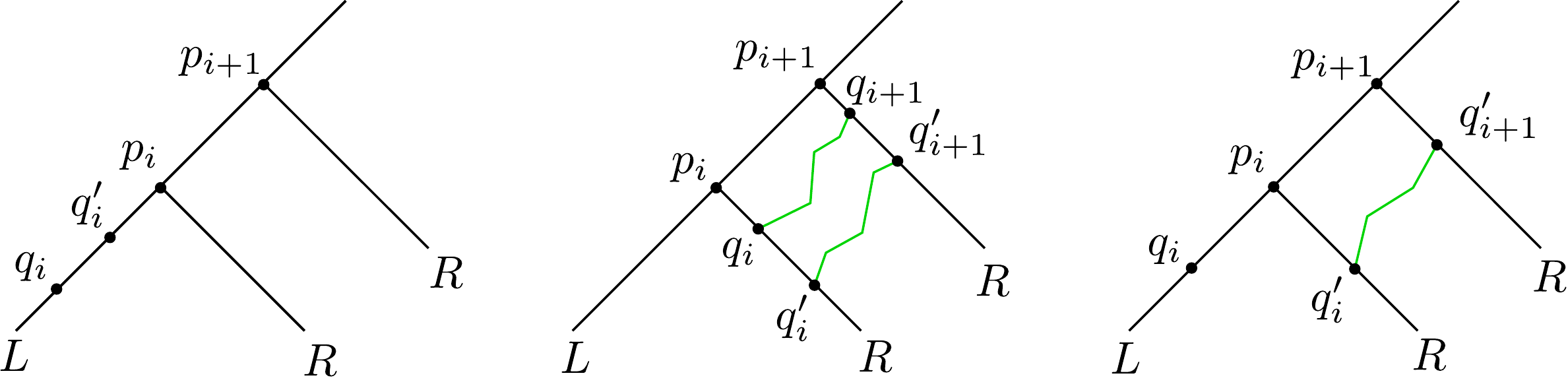}
\caption{Cases (i)-(iii) (ordered left to right) proving that $\wt f$ is to the left of $\wt f'$. In the first case, it is possible that $q_i' = p_i$. In the second and third cases, it is possible that $q_{i+1} = p_{i+1}$.  In the third case, it is possible that $q_i = p_i$. The $\Phi$-segments $f_i$ and $f_i'$ are drawn in green.}
\label{fig:to_the_left}
\end{figure}

Suppose that $\wt f$ is to the left of $\wt f$ at the vertex $p_i$ along $\wt c$. Let $\ell_i^L$ and $\ell_i^R$ be the left and right branch lines through $p_i$, respectively, and let $q_i, q_i'$ be the vertices of $\wt f, \wt f'$, respectively, that lie below $p_i$ on either $\ell_i^L$ or $\ell_i^R$. Let us suppose that $p_{i+1}$ (the next vertex along $\wt c$) also lies on $\ell_i^L$; the proof in the opposite case is symmetric. We have the following cases (see \Cref{fig:to_the_left}):
\begin{enumerate}[label=(\roman*)]
\item If $q_i$ and $q_i'$ are on $\ell_i^L$, then $q_{i+1} = q_i$ and $q_{i+1}' = q_i'$. Hence, $q_{i+1}$ is to the left of $q_{i+1}'$ at $p_{i+1}$. 
\item If $q_i$ and $q_i'$ are on $\ell_i^R$ (and $q_i \neq p_i$), then both $f_i$ and $f_i'$ (the segments of $\wt f, \wt f'$ joining $q_i,q_i'$ to $q_{i+1},q_{i+1}'$, respectively) are contained in the same band $b$. Since $\ell_i^L = \ell_{i+1}^L$, we have that $q_{i+1},q_{i+1}'$ are both contained in $\ell_{i+1}^R$.
As $\wt f$ and $\wt f'$ are assumed to be distinct, the segments $f_i,f_i'$ do not intersect in $b$ and so their terminal vertices are orders along $\ell_{i+1}^R$ in the same manner as they were ordered along $\ell_i^R$. Hence, $q_{i+1}$ is to the left of $q_{i+1}'$ at $p_{i+1}$. 
\item If $q_i$ and $q_i'$ are on different branch lines, then it must be that $q_i$ is on $\ell^L_i$ and $q_i'$ is on $\ell^R_i$. In this case, $q_{i+1} = q_i$ and $q_{i+1}'$ lies in $\ell_{i+1}^R$ (where $q_{i+1}' = p_{i+1}$ is possible).  Hence, $q_{i+1}$ is to the left of $q_{i+1}'$ at $p_{i+1}$. 
\end{enumerate}

We have shown that $\wt f$ lies to the left of $\wt f' = g^{-1}(\wt f)$ at $g^i(p)$ for each $i \in \ZZ$. By $\langle g \rangle$-equivariance, this implies that $g^{i+1}(\wt f)$ lies strictly to the left of $g^{i}(\wt f)$ at $p$ for every $i>0$. This, however, contradicts that $\wt f$ is periodic under some power of $g$. Hence, we must have that $\wt f = g(\wt f)$ and so $f$ is homotopic in $M$ to $c$.
\end{proof}

For the other half of \Cref{th:cones_equal}, we need two more lemmas.

Recall that a directed graph is \define{strongly connected} if for any two of its vertices $u$ and $v$, there is a directed path from $u$ to $v$.

\begin{lemma}\label{lem:sc_dual}
The dual graph $\Gamma$ is strongly connected. 
Moreover, any cycle in $\Gamma$, not necessarily directed, is homologous to a linear combination of dual cycles. 

Hence, the dual cycles generate $H_1(M;\Z)$.
\end{lemma}

\begin{proof}
We first show that $\Gamma$ is strongly connected.
Consider a tetrahedron $T$, and let $S(T)$ be the union of all tetrahedra accessible from $T$ via a directed path in $\Gamma$. It is immediate that $S(T)$ has no top faces on its boundary. Since the tetrahedra composing $S(T)$ have an equal number of top and bottom faces, and each top face is glued to a bottom face, $S(T)$ has no bottom faces in its boundary. Hence $S(T)=M$, proving strong connectivity.

Now, let $x$ be a cycle in $\Gamma$.  We can break $x$ into a concatenation of segments which are maximal with respect to the property that they either completely agree or disagree with the edge orientations of $\Gamma$. We write $x=(f_1, b_1,\dots, f_n, b_n)$ where $f$ and $b$ are chosen to evoke ``forward" and ``backward." Denote the initial vertex of $f_i$ by $p_i$, the terminal vertex by $q_i$, and observe that this dictates that $b_i$ travels from $q_i$ to $p_{i+1}$ where indices are taken modulo $n$.

Let $\alpha_i$ be a directed path from $q_i$ to $p_i$, which exists by the strong connectivity of $\Gamma$. Let $A_i$ be the directed cycle $(f_i,\alpha_i)$, and let $B$ be the directed cycle $(b_n^{-1}, \alpha_n,\dots, b_1^{-1},\alpha_1)$. Then $x$ is homologous to $A_1+\cdots+A_n-B$. Since $\pi_1(\Gamma)\to\pi_1(M)$ is surjective, this completes the proof.
\end{proof}

Recall that if $A$ is a convex set in $\R^d$, a \emph{supporting hyperplane} for $A$ is a hyperplane $H$ such that $H\cap \partial A\ne \varnothing$ and $A$ is contained in one of the two closed half spaces determined by $H$. We will use the following finite-dimensional version of the Hahn-Banach Theorem:

\begin{lemma}[Hahn-Banach]\label{GHB}
Let $C$ be a strongly convex polyhedral cone in a finite dimensional vector space $V$. If $K$ is a subspace of $V$ with $K\cap \intr(C)=\varnothing$, then $K$ can be extended to a supporting hyperplane $H$ for $C$. If $K\cap C=\{0\}$, we can choose $H$ such that $H\cap C=\{0\}$.
\end{lemma}

\begin{proof}
Apply the separating hyperplane theorem (see, for example, \cite[Section 2.5]{boyd2004convex}),  
to $K$ and $\intr(C)$ for the first statement. For the second, slightly enlarge $C$ to $C'$ so that $C'$ is still strongly convex with $K\cap \intr(C')=\varnothing$ and apply the same theorem to $K$ and $\intr(C')$.
\end{proof}

The following is a strengthening of \cite[Proposition 2.12]{mosher1991surfaces}. The appeal to Hahn-Banach comes from \cite[Theorem I.7]{sullivan1976cycles} and is made explicit in \cite[Theorem 5.1]{mcmullen2015entropy}, where the second statement of the lemma is proven.

\begin{lemma} \label{lem:ext}
Suppose that $D$ is a finite directed graph and $\eta \in H^1(D; \R)$ is nonnegative on all directed cycles. Then there is a nonnegative cocycle $m \colon E(D) \to \R_{\ge 0}$ 
representing $\eta$. 

If $\eta$ is positive on all directed cycles, then $m$ can be taken to be positive on all directed edges of $D$. 
\end{lemma}

\begin{proof}
We assume $\eta$ is not identically 0, since in that case the statement is clearly true.
First suppose that $D$ is strongly connected. Let $E=E(D)$, let $Z_1 \subset \R^E$ be the subspace of cycles in $D$, and let $\R_+^E$ and $\R_{\ge 0}^E$ be the open positive orthant and its closure, respectively. 

Note that $Z_1\cap \R_+^E$ is nonempty by strong connectivity, and that $\eta$ is strictly positive on this set since $\eta\ne 0$. By \Cref{GHB}, $\ker(\eta)$ can be extended to a supporting hyperplane $H$ for $\R_{\ge0}^E$. Let $V$ be the span of a vector in $Z_1\cap \R_+^E$, and choose a decomposition $H=\ker(\eta)\oplus W$ for some subspace $W$. Then $\R^E=V\oplus \ker(\eta)\oplus W$, because neither $\ker(\eta)$ nor $W$ intersect $\R_+^E$. Thus we can define a linear functional $m\colon \R^E\to \R$ which extends $\eta$ by requiring $m|_W=0$. Note that $m$ is a nonnegative cocycle representing $\eta$.
If $\eta$ is strictly positive on cycles in $D$, then we can choose $H$ so that $H\cap \R_{\ge0}=\{0\}$, guaranteeing that $m$ is a positive cocycle. The same argument shows that the result holds when $D$ is a union of strongly connected components.

Now suppose that $D$ is arbitrary. Let $R$ be the union of recurrent components of $D$, i.e. the union of all maximal strongly connected subgraphs. Let $A$ be the complementary subgraph of $R$ in $D$. Let $c$ be a cocycle representing $\eta$, and write $c=c_R+c_A$, where $c_R$ and $c_A$ are supported on $R$ and $A$. By the case above, $c_R$ can be written as $c_++\delta h$ for some $0$-cochain $h$, where $c_+$ is either positive or nonnegative depending on whether $\eta$ is positive or nonnegative on directed cycles.

Consider the quotient $D_R$ obtained by collapsing each component of $R$ to a vertex. Note that $D_R$ contains no directed cycles, so the edge orientations induce a partial order on its vertices and there is a function $f \colon V(D_R)\to \R$ compatible with this partial order. We may lift $f$ to a function $F$ on $D$ which is constant on each component of $R$. Then $\delta F$ is a coboundary which is $0$ on $E(R)$ and positive on $E(A)$. By replacing $F$ with a sufficiently large multiple of itself, we can guarantee that for each edge $a$ of $A$, $F(a)>|(c_A+\delta h)(a)|$.
Then $c_R + c_A + \delta F$ is the desired cocycle.
\end{proof}

\begin{proposition}\label{prop:dual_carried}
There is an equality of cones
\[
\cone_1^\vee(\Gamma) = \cone_2(\tau).
\]
That is, the classes carried by $\tau^{(2)}$ are precisely those whose algebraic intersection with any closed positive transversal to $\tau^{(2)}$ is nonnegative.

Moreover, any integral class $\alpha \in \cone_1^\vee(\Gamma) \cap H^1(M; \ZZ)$ is represented by a surface $S$ carried by $\tau^{(2)}$.
\end{proposition}

\begin{proof}
Since elements of $\cone_1(\Gamma)$ are nonnegative combinations of directed cycles which intersect the faces of $\tau$ positively, it's clear that $\cone_2(\tau) \subset \cone_1^\vee(\Gamma)$.

Now let $0\neq \eta \in \cone_1^\vee(\Gamma)$ and let $i^*\eta \in H^1(\Gamma)$ be its pullback to $\Gamma$. By \Cref{lem:ext} we can represent $i^*\eta$ by a nonnegative cocycle $m \colon E(\Gamma) \to \R_{\ge 0}$ on the edges of $\Gamma$. 
Using the identification between directed edges of $\Gamma$ and cooriented faces of $\tau$, the cocycle $m$ gives a 2-chain $m_2\colon \{\text{$\tau$-faces}\}\to \R$. 

From the duality between $\Gamma$ and $\tau$, it follows that $m_2$ is a cycle rel boundary and thus defines a class in $H_2(M,\partial M)$. More explicitly, for a $\tau$-edge $e$, let $\sigma_e$ denote the $B^s$-sector pierced by $e$. Let $\ell$ and $r$ be the two sides of $\sigma_e$. To check that $m_2$ is indeed a relative cycle, we need only check that $m(\ell)=m(r)$. Since the loop $i(l \cdot \bar r)$ is trivial in $M$, 
\[
m(l \cdot \bar r) = i^*\eta(\ell \cdot \bar r) = \eta(i(\ell \cdot \bar r)) = 0,
\]
and $m(\ell) =m(r)$ as required. Hence $m_2$ determines a carried class $h\in \cone_2(\tau)$ whose algebraic intersection with any dual cycle $c$ is equal to $\eta(c)$. Since the dual cycles generate $H_1(M)$, $h$ is the Lefschetz dual of $\eta$.

The moreover statement now follows from \Cref{lem:integral_classes}.
\end{proof}

\subsection{Faces of the Thurston norm ball} \label{sec:faces_thurston_norm}
Here we show that the veering triangulation always determines an \emph{entire} face of the Thurston norm ball, generalizing what was known for the layered case. We use this to give a criterion (\Cref{th:omni}) to detected fiberedness, which can be thought of as a combinatorial version of Fried's condition for a flow to be circular  \cite[Theorem D]{fried1982geometry}.

From its definition, we see that the dual graph $\Gamma = \Gamma_\tau$ is a cycle in $M$ and so represents a homology class $[\Gamma] \in H_1(M)$. 
Further, for any surface $\Sigma$ carried by $\tau^{(2)}$, we see that $\chi(\Sigma) = - \frac{1}{2}\langle [\Gamma], [\Sigma] \rangle$, where $\langle \cdot, \cdot \rangle$ is algebraic intersection. This is because the intersection of the carried surface with $[\Gamma]$ equals the number of ideal triangles in an ideal triangulation of the surface.

With this in mind, we define the combinatorial \define{Euler class} of $\tau$ to be
\[
e_\tau=-\frac{1}{2}\langle [\Gamma],\cdot\rangle.
\]
We have $e_\tau\in H^2(M,\partial M)$. Recall from \Cref{sec:normbkgd} that $\hbs$ is a taut branched surface. Hence, if $\Sigma$ is carried by $\tau^{(2)}$ then $-e_\tau([\Sigma]) = x([\Sigma])$, where $x(\cdot)$ is the Thurston norm.

\begin{theorem}[The whole face and nothing but the face] \label{th:whole_face}
The subset of $H_2(M,\partial M)$ on which the Thurston norm $x$ is equal to $-e_\tau$ is exactly $\cone_2(\tau)$. Hence, $\cone_2(\tau)$ is equal to the cone over a face $\bf F_\tau$ of $B_x(M)$.
\end{theorem}

We say that $\tau$ \emph{determines} the face $\bf F_\tau$. Here we note that, 
as in work of Mosher \cite{mosher1992dynamical},
the empty face is allowed. 
The proof of \Cref{th:whole_face} will require additional terminology that we now turn to explain.

\subsubsection{The unstable branched surface $B^u$}
Recall that in \Cref{sec:stable_brached} we introduced the stable branched surface $B^s$ as a particular smoothing of the $2$-complex of $M$ dual to $\tau$. According to \Cref{lem:comp_smooth}, this smoothing is characterized by the condition that for each face $f$ of $\tau$, $f \cap B^s$ is the train track whose large branch meets the bottom edge of $f$.

In an analogous way, we define the \define{unstable branched surface} $B^u$ to be the branched surface with the same underlying $2$-complex such that for each face $f$ of $\tau$, $f \cap B^u$ is the train track whose large branch meets the top edge of $f$. Here, the top edge of a face $f$ is the unique top edge of the tetrahedron $t$ that contains $f$ as a top face. The intersection of $B^u$ with a tetrahedron is show in \Cref{fig:unstable_branched}. Just as with the stable branched surface, the branch locus of $B^u$ can be naturally identified with the dual graph $\Gamma$. We fix the position of both $B^s$ and $B^u$ in $M$, but make no assumption about how they sit with respect to each other.

\begin{figure}[htbp]
\begin{center}
\includegraphics[height=1.5in]{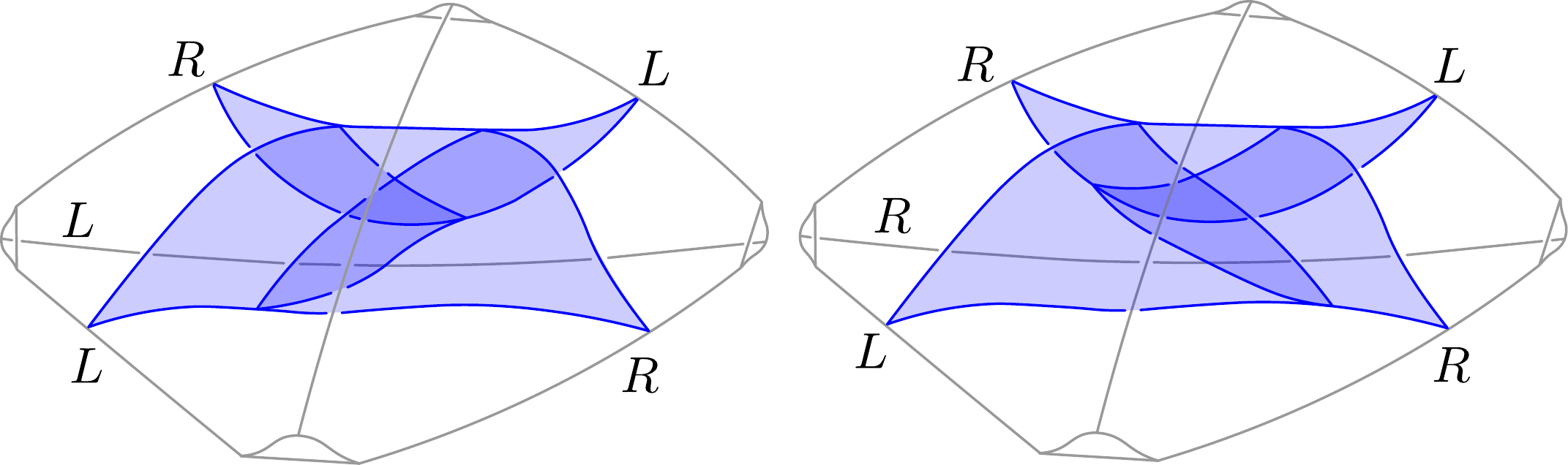}
\caption{The unstable branched surface $B^u$ in a tetrahedron $t$ depending on whether the bottom edge of $t$ is left or right veering.}
\label{fig:unstable_branched}
\end{center}
\end{figure}

Exactly as in \Cref{sec:flowgraph}, we can define \define{unstable branch loops} as smoothly immersed closed curves in the branch locus of $B^u$.

We note that if $S$ is an embedded surface in $M$ transverse to $B^u$ and $B^s$, then the intersections $S \cap B^u$ and $S \cap B^s$ are train tracks on the surface $S$ and we refer to the complementary regions as \define{patches}. (Here, we make no assumption on the Euler characteristic of a patch.) If $S$ is oriented, then each intersection with $\Gamma$ has a sign, and since these intersections are in bijection with the cusps of $S \cap B^s$ (and $S \cap B^u$), each cusp of these tracks is either \textbf{positive} or \textbf{negative} accordingly.

\subsubsection{Partial branched surfaces and a result of Landry}
To apply results from \cite{Landry_norm} we replace $M$ with the corresponding compact manifold as follows: First, let $\mr M$ be the manifold obtained by truncating the cusps of the tetrahedra of $\tau$, as in \Cref{sec:veering_basics}. We continue to use $\tau^{(2)}$ to denote the corresponding branched surface with boundary in $\mr M$. Now let $M$ be the homeomorphic manifold obtained by attaching a thickened torus $T^2 \times [0,1]$ to each boundary component of $\mr M$. 

In the terminology of \cite{Landry_norm}, $\tau^{(2)}$ is a \define{partial branched surface} of $M$ with respect to $U = M \ssm \mathrm{int} (\mr M)$. A properly embedded surface $S\subset M$ is \textbf{carried} by the partial branched surface $\tau^{(2)} \subset \mr M$ if $S$ has no components completely contained in $U$, $S \ssm \intr(U)\subset \mr M$ is carried by $\tau^{(2)}$, and each component of $S\cap U$ is a properly embedded $\pi_1$-injective annulus in $U$ with either one or both boundary components on $\partial \mr M$.

The following theorem \cite[Theorem 8.1]{Landry_norm} is a key ingredient for \Cref{th:whole_face}:

\begin{theorem}[Landry] \label{th:landry}
Let $S$ be an incompressible, boundary incompressible surface in $M$. Further suppose that $S$ has the property that for
any surface $S'$ isotopic to $S$ that is transverse to $B^u$ and $B^s$, either 
\begin{itemize}
\item one of $S' \cap B^u$ or $S' \cap B^s$ has a nullgon or monogon patch, or
\item for both tracks $S' \cap B^u$ and $S' \cap B^s$, each negative cusp belongs to a bigon patch.
\end{itemize}
Then $S$ is isotopic to a surface carried by the partial branched surface $\tau^{(2)} \cap \mr M$.
\end{theorem}

We remark that it is observed in \cite{Landry_norm} that the tracks $S \cap B^{u/s}$ never have monogon patches.

\subsubsection{The proof of \Cref{th:whole_face}}
Before beginning the proof, we recall a formula for computing Euler characteristic of surfaces using train tracks (see e.g.  \cite{casson-bleiler}). For a surface $C$ with finitely many punctures and cusped boundary, we can define 
\[
\ind(C)=2\chi(C)-\text{(no. of boundary cusps)}.
\]
If $t$ is a train track on a surface $S$ of finite genus with finitely many punctures, then we have 
\[
2\chi(S)=\sum_C \ind(C)
\]
where the sum is taken over all patches of $S$ with respect to $t$. Note that if $\chi(S) \le 0$, then the only patches with positive index are nullgons and monogons.

\begin{proof}[Proof of \Cref{th:whole_face}]
Let $W\subset H_2(M,\partial M)$ be the cone on which $-e_\tau=x$. Any class contained in $\cone_2(\tau)$ is represented by a surface carried by $\tau^{(2)}$ by \Cref{lem:integral_classes}. As discussed at the beginning of this section, such a surface $\Sigma$ is taut and has $-e_\tau([\Sigma]) = x([\Sigma])$. This yields the containment $\cone_2(\tau)\subset W$.

To show the reverse containment it suffices to produce for any integral class $\alpha\in W$, a surface representing $\alpha$ and carried by $\tau^{(2)}$.

Let $S_\alpha$ be any taut surface representing $\alpha$. We will use \Cref{th:landry} to prove that $S_\alpha$ is isotopic to a surface carried by the partial branched surface $\tau^{(2)} \cap \mr M$. For this, suppose that $S$ is a surface transverse to the branched surface $B^u$ and $B^s$ and isotopic to $S_\alpha$. Let $t^s = S \cap B^s$ and $t^u = S \cap B^u$ be the corresponding train tracks on $S$. By \Cref{th:landry} it suffices to assume that neither of these tracks has a nullgon or monogon patch and to show that every negative cusp belongs to a bigon patch. The arguments for $t^s$ and $t^u$ are identical, so we work with $t = t^u$.

Recall that there is a bijection between intersections $S \cap \Gamma$ and cusps of $t$, and that each cusp of $t$ is the cusp of exactly one patch $D \subset S$. (Since $S$ misses the vertices of $\Gamma$, the track $t$ is generic.) Also recall that the sign of a cusp to be the sign of the corresponding intersection point of $\Gamma$ and $S$.

It is clear that the smooth annular patches (i.e. topological annuli without cusps) do not contribute to the signed number of cusps, and we claim the same is true for the patches that are topological disks (i.e. unpunctured $n$-gons). Indeed, suppose that $D$ is a patch which is a topological disk contained in a tube $V$ of $B^u$, defined as in \Cref{sec:relating}. 
Since $D$ is a disk in $V$ which is unpunctured, $\partial D$ is a trivial curve in $\partial V$. Hence, $\partial D$ meets each unstable branch curve of $\partial V$ (i.e. the intersection of the cusps of $B^u$ with $V$) with algebraic intersection number $0$. Since this intersection number is also the signed number of cusps in $D$, we see that the signed number of cusps for any disk is also $0$.

Partition the set of patches on $S$ into the set of topological disks and smooth annuli $\mc DA$ and the rest $n\mc DA$. Note that each patch in $n \mc DA$ has both negative index and nonpositive Euler characteristic. 

The quantity $\langle [S],[\Gamma]\rangle$ can be computed as the signed number of intersections between $S$ and $\Gamma$, or alternatively as the signed count of cusps over all patches of $S$ with respect to $t = S\cap B^u$.
Since the signed number of cusps of annuli and disks is $0$, we have
\[
2\chi(S)=-\langle [S],[\Gamma]\rangle=\sum_{n\mc{DA}} -\algcusps(D) 
\]
where $\algcusps(D)$ denotes the signed number of cusps in $\partial D$ and the sum is over patches that are not topological disks or smooth annuli. Subtracting this expression for $2\chi(S)$ from the expression in terms of the indices of patches provided above gives
\begin{align*}
0&=\sum_{\mc {DA}}\ind(C)+\sum_{n\mc DA}
\left( \ind(D)+\algcusps (D) \right ).
\end{align*}

We claim that each term of the above sum is nonpositive. First, each term of the first sum is nonpositive, since each patch has nonpositive index following our assumption that there are no patches which are nullgons or monogons.
For the second sum, the term for $D$ in $n \mc DA$ is equal to $2\chi(D) + (\algcusps(D) - \mathrm{cusps}(D))$, which is nonpositive since  $\chi(D)\le 0$ and $\algcusps(D) \le \mathrm{cusps}(D)$. Hence, each term of the sums is nonpositive and so it must be that each term is actually $0$. We conclude that each patch in $\mc DA$ is an annulus or a bigon with cusps whose signs cancel, and that each patch in $n \mc DA$ is a once-punctured disk with only positive cusps. 

The same argument applies to $t^s = S \cap B^s$, and so by \Cref{th:landry} we conclude that, after an isotopy, $S$ is carried by the partial branched surface $\tau^{(2)} \cap \mr M$. In particular,  $S \cap \mr M$ is carried by $\tau^{(2)}$ and under the canonical isomorphism $H_2(\mr M, \partial \mr M) = H_2(M,\partial M)$, it represents the class $\alpha$.  Hence, $\alpha \in \cone_2(\tau)$ and the proof is complete.
\end{proof}

\subsection{Dimension and fiberedness of cones via the veering polynomial}
\label{sec:conedimension}

We begin with the following lemma showing that $P_\Phi$ determines the cone of homology directions of $\tau$. 

\begin{lemma} \label{lem:support_gen}
For a veering triangulation $\tau$, the cone of homology directions $\cone_1(\Gamma)$ is generated by  $i_*(\supp(P_{\Phi}))$, where $i_*\colon H_1(\Phi; \R) \to H_1(M; \R)$ is induced by the inclusion of the flow graph into $M$.
\end{lemma}

\begin{proof}
By \Cref{th:cones_equal}, $\cone_1(\Gamma) = \cone_1(\Phi)$. Since an arbitrary directed cycles in $\Phi$ is a sum in $H_1$ of simple directed cycles, it suffices to know that every simple directed cycles of $\Phi$ appears in the support of the Perron polynomial $P_\Phi$. But this follows immediately from \Cref{eq:cliquepoly}.
\end{proof}

The following theorem summarizes the main results of this section and gives a characterization of fibered faces in the spirit of Fried's criterion.

\begin{theorem}
\label{th:omni}
Let $\tau$ be a veering triangulation. Then $\cone_2(\tau)$ is equal to the subset of $H_2(M,\partial M)$ on which $x=-e_\tau$, which is the cone over a face $\bf F_\tau$ of $B_x(M)$. The codimension of $\RR_+\bf{F}_\tau$ is equal to the dimension of the largest linear subspace contained in $\cone_1(\Gamma)$. 

Furthermore, the following are equivalent:
\begin{enumerate}[label=(\roman*)]
\item $i_*(\supp(P_{\Phi_\tau}))$ lies in an open half-space of $H_1(M;\R)$,
\item there exists $\eta\in H^1(M)$ with $\eta([\gamma])>0$ for each closed $\tau$-transversal $\gamma$,
\item $\tau$ is layered, and 
\item $\bf F_\tau$ is a fibered face.
\end{enumerate}
\end{theorem}
\begin{proof}
\Cref{th:whole_face} gives that $\cone_2(\tau)$ equals the cone over a face $\bf F_\tau$ of $B_x(M)$. The statement about its codimension follows from the discussion at the beginning of \Cref{sec:conedefs} (or see \cite[Section 1.2, Fact 10]{Ful93}) after we recall that the dual cone of $\cone_2(\tau)$ equals $\cone_1(\Gamma)$ by \Cref{prop:dual_carried}.

It remains to show the equivalence of the conditions. 

For $(i) \iff (ii)$, since $i_*(\supp(P_{\Phi_\tau}))$ lies in an open half-space of $H_1(M;\R)$, there is an $\eta \in H^1(M)$ that is positive on all directed cycles in $i_*(\supp(P_{\Phi_\tau})$. In particular, no such directed cycle is $0$ in $H_1(M; \R)$. Hence, if $\gamma$ is a closed $\tau$-transversal, then $[\gamma] \in H_1(M; \R)$ is in the positive span of directed cycles in $i_*(\supp(P_{\Phi_\tau})$ by \Cref{lem:support_gen} and so $\eta([\gamma])>0$. For the converse, if $\eta$ is positive on closed transversals, then it is positive on $i_*(\supp(P_{\Phi_\tau})$. This implies that $i_*(\supp(P_{\Phi_\tau})$ lies in an open half-space.

The implication $(iii) \implies (ii)$ is immediate because a fiber of a layered veering triangulation positively intersects every closed transversal. 

Next we turn to $(ii) \implies (iii)$. Let $\eta \in H^1(M)$ be positive on closed transversals of $\tau$. Then its pullback to the dual graph $\Gamma$ is positive on directed cycles and so by 
\Cref{lem:ext} it is represented by a positive cocycle $m$ on the edges of $\Gamma$. After perturbing slightly, we can assume that $m \colon E(\Gamma) \to \mathbb Q_+$ and so for some $n \ge 0$, $n\cdot m$ assigns a positive integer to each edge of $\Gamma$. As in the proof of \Cref{prop:dual_carried}, this implies that $n \cdot \eta$ is presented by a surface $S$ \emph{fully carried} by $\tau^{(2)}$, meaning that $S$ traverses each face of $\tau$. From this, it follows that the components of $M \ssm S$ are $I$-bundles and that $\tau$ is a layered triangulation on any component of $S$.

Finally, $(iii) \implies (iv)$ is clear because a `layer' of a layered triangulation is a fiber of $M$. The reverse implication is more difficult and reserved for \Cref{prop:layered}.
\end{proof}

The following proposition completes the proof of \Cref{th:omni}. In short, it states that only layered veering triangulations can carry fibers.

\begin{proposition} \label{prop:layered}
Let $\tau$ be a veering triangulation of $M$ whose associated face ${\bf F_\tau} \subset B_x(M)$ is fibered. Then $\tau$ is layered.
\end{proposition}

The proof is inspired by Agol's proof of virtual fibering \cite{agol2008criteria}. 
Recall that $G = \pi_1(M)$ is virtually special by Wise \cite[Theorem 17.14]{wisestructure} 
(see also Cooper--Futer \cite[Theorem 1.4]{cooper2019ubiquitous} and Groves--Manning \cite[Theorem A]{QuasiGM})
and so it is virtually RFRS by Agol \cite[Corollary 2.3]{agol2008criteria}. 
This means that there exists a finite index subgroup $G_0 \le G$ and subgroups $G_0 \supset G_1 \supset G_2 \cdots$ such that
\begin{enumerate}
\item $\bigcap_i G_i = 1$,
\item $G_i$ is a normal, finite index subgroup of $G_0$, and
\item for each $i$, the map $G_i \to G_i / G_{i+1}$ factors through $G_i \to  H_1(G_i; \Z)/\text{torsion}$.
\end{enumerate}
Such a chain of subgroups is called a \emph{RFRS tower}.

\begin{proof}[Proof of \Cref{prop:layered}]
Let $G = \pi_1(M)$ and let $G \ge G_0 \supset G_1 \supset G_2 \cdots$
 be a RFRS tower for the finite index subgroup $G_0$.
For each $i$, let $\wt M_i$ be the corresponding cover of $M$. 

Since fibers are the unique taut surfaces in their homology class, $\tau$ carries all the fibers in the cone $\R_+{\bf F_\tau}$ by \Cref{th:whole_face}. Let $\wt \tau_i$ be the lifted veering triangulation of $\wt M_i$. Since $\wt \tau_i$ carries the lifted fibers from $\tau$, we have that each cone $\R_+ {\bf F}_i : =\R_+{\bf{F}}_{\wt \tau_i}$ is fibered and hence top dimensional. 

Suppose that $\tau$ is not layered; hence, no surface is fully carried by $\tau^{(2)}$. 
We choose a (multiple of a) fiber $S$ that is maximal with resect to the weights placed on faces of $\tau$. That is, $S$ is carried by $\tau$ and traverses every faces that is traversed by any carried surface. Such a surface can be constructed by summing weights of finitely many fibers that traverse the faces traversed by some carried surface. Of course, the complete preimage $\wt S_i$ of $S$ in $\wt M_i$ is a (multiple of a) fiber carried by $\wt \tau_i$ whose weights on faces are lifted from those of $S$. We claim that $\wt S_i$ also has the weight maximality property for $\wt \tau_i$:

\begin{claim} \label{cl:maximal}
The carried surface $\wt S_i$ traverses every face of $\wt \tau_i$ that is traversed by any carried surface. In particular, each $\wt \tau_i$ is also not layered for $i \ge 0$. 
\end{claim}

\begin{proof}[Proof of Claim 1]
Suppose that $Z$ is a carried surface that traverses a face not traversed by $\wt S_i$. Then the system of weights on faces of $\wt \tau_i$ corresponding to $Z$
pushes down to $\tau$ to gives a system of weights (satisfying the matching conditions) and determines a surface carried by $\tau$ that traverses faces not traversed by $S$, a contradiction.
\end{proof}

Since $\wt \tau_i$ is not layered, it has dual cycles that do not intersect $\wt S_i$; see $(ii) \implies (iii)$ in the proof of \Cref{th:omni}. (Here, we set $\wt M_{-1} = M$, $\wt \tau_{-1} = \tau$, and $\wt S_{-1} = S$.)
These dual cycles lie in the `guts' of $\wt S_i$ defined as follows: Let $\mc G(\wt S_i)$ be the open region obtained by taking the open tetrahedra of $\wt \tau_i $ together with the open faces of $\wt \tau^{(2)}_i$ that are \emph{not} traversed by $\wt S_i$. 
Any dual cycle not crossing $\wt S_i$ lives in $\mc G(\wt S_i)$ by construction. 

\begin{claim} \label{cl:triv_im}
The image of $H^1(\wt M_i) \to H^1(\mc G(\wt S_i))$ is trivial for $i \ge -1$. 
\end{claim}

\begin{proof}[Proof of Claim 2]
Suppose that some nontrivial $\alpha$ is in the image of the map. Since $\R_+ {\bf F}_i$ is top dimensional in $H^1(\wt M_i) = H_2(\wt M_i, \partial \wt M_i)$, i.e. it has nonempty interior, there is a $\beta \in \R_+ {\bf F}_i$  such that $\alpha + \beta \in \R_+ {\bf F}_i$. But then $\alpha + \beta$ is represented by a surface $Z$ carried by $\wt \tau_i$ that has nonzero intersection number with an oriented loop in $\mc G(\wt S_i)$. This contradicts \Cref{cl:maximal} which implies that $Z$ is disjoint from $\mc G(\wt S_i)$.
\end{proof}

Now by \Cref{cl:triv_im}, each component of $\mc G(\wt S_0)$ lifts homeomorphically to $\wt M_1$. Indeed, this amounts to the claim that any loop in $\mc G(\wt S_0)$ is trivial under the homomorphism $G_0 \to H_1(\wt M_0)/\text{torsion}$.
By definition of $\wt S_i$,
all the lifts of $\mc G(\wt S_0)$ to $\wt M_1$ live in $G(\wt S_1)$. 
Then again applying \Cref{cl:triv_im}, we see that $\mc G(\wt S_0)$ lifts to $\wt M_2$. Continuing in this way, we should have that $\mc G(\wt S_0)$ lifts to each $\wt M_i$ and hence for each component $C$ of $\mc G(\wt S_0)$, $\pi_1(C) \subset \bigcap_i G_i$. Hence, each component of $\mc G(\wt S_0)$ has trivial fundamental group. This however contradicts the fact that if $\tau$ (and hence $\wt \tau_0$) is not layered, then $\mc G(\wt S_0)$ contains dual cycles, which are homotopically essential by a combinatorial version of a theorem of Novikov; see \cite[Theorem 3.2]{schleimer2020essential} or the references found in therein.
The proof of \Cref{prop:layered} is complete.
\end{proof}

\section{Relating the polynomials $V_\tau$ and $\Theta_\tau$} \label{sec:polys}
In this section, we establish a precise version of the identity 
\[
V_\tau =  \Theta_\tau \cdot \prod(1 \pm g_i),
\]
where the $g_i \in G$ are represented by certain directed cycles in the dual graph $\Gamma = \Gamma_\tau$. We will see in \Cref{sec:layered} that this is a generalization of McMullen's Determinant Formula from the fibered setting.

After some definitions we will state \Cref{th:factorization!} which is the main
result of this section.

\subsection{The AB polynomial}
As before, we let $\aM$ denote the universal free abelian cover of $M$.

Recall the definitions of $L \colon \Z[G]^E\to\Z[G]^E$ and $L^\triangle \colon \Z[G]^F\to\Z[G]^E$ and
their respective cokernels, $\e(\wt\tau)$ and $\e^\triangle(\wt\tau)$ (\Cref{eq:edge_mod}
and \Cref{eq:face_mod}). 
To relate the two we recall from \Cref{lem:face_implies_tet}
the fact that, if $f$ is a bottom face
of a tetrahedron $t$ and $e$ its bottom edge, then there is a unique top face $f'$ of $t$
such that $L^\triangle(f)+L^\triangle(f') = L(e)$. This face $f'$ is characterized by the fact that the turn in the dual graph $\Gamma$ at $t$ given by the pair $(f,f')$ is \emph{anti-branching} 
(see \Cref{lem:branchingturn}).

The correspondence $f\mapsto f'$ defines a map $A \colon F\to F$, which after lifting to $\aM$
induces a $\Z[G]$-module homomorphism $A \colon \Z[G]^F\to\Z[G]^F$.   The correspondence $f
\mapsto e$ sending a face to its bottom edge gives a $2-1$ map $\varepsilon\colon  F\to E$, which also extends to
$\varepsilon\colon \Z[G]^F \to \Z[G]^E$. 
Defining $L^\ab = I+A$, we can express the above identity as 
\begin{equation}\label{eq:Lab L}
  L^\triangle \circ L^\ab = L\circ \varepsilon.
\end{equation}
Now define the $AB$-module $\ab(\wt\tau)$ as the cokernel of $L^\ab$, that is via
\begin{align}
 \Z[G]^F \overset{L^\ab}{\longrightarrow} \Z[G]^F \longrightarrow \ab(\widetilde{\tau})
 \to 0. 
\end{align}
As $L^\ab$ is a square matrix, its fitting ideal is principally generated by $\det(L^\ab)$ and we define the \define{$\bs{AB}$ polynomial} of $\tau$ by
\[
V^\ab = V^\ab_\tau = \det(L^\ab) \in \Z[G].
\]
We will see in \Cref{lem:computing_correction} that $V^\ab$ has the form $\prod(1\pm g_i)$ over certain $g_i\in
G$. Our main result will then be: 

\begin{theorem} \label{th:factorization!}
Suppose that $\mathrm{rank}(H_1(M)) >1$. Then
\[
V_\tau = V^\ab \cdot \Theta_\tau
\]
up to multiplication by a unit in $\Z[G]$.
\end{theorem}

When the rank of homology is $1$, the conclusion holds up to multiplying by $(1\pm t)$, where $t$ generates $G$. See \Cref{rmk:rank_1}.

Let us give a short sketch of the proof.

We first introduce $AB$-cycles, which are just the cycles of
the permutation $A$, and give in \Cref{lem:computing_correction}
a factorization formula for $V^\ab$. 

\Cref{prop:V_divides} gives the ``easy'' direction of the theorem, namely $V \: | \: \left
(V^\ab \cdot \Theta_\tau \right)$, which follows from standard facts about Fitting ideals
(but we supply an explicit proof).

In \Cref{lem:alt_pres_face} we give an alternate presentation of $\e^\triangle$ that allows
us to express $\Theta_\tau$ as the gcd of minors of a matrix involving $L$ and selected
columns of $L^\triangle$. 

\Cref{lem:cycle eq}, the $AB$-cycle equation, is an identity combining $L$, $L^\triangle$ and the
$AB$-cycles.  In \Cref{prop:facting_out_families} we use this equation to enable column
operations that express $V_\tau$ as one of the minors of the matrix in
\Cref{lem:alt_pres_face}.  This gives us, in \Cref{cor:factoring_families}, a result of
the form
$$
\left(\Theta_\tau\cdot \prod(1\pm g_i)\right) \: \big|\: V_\tau
$$
where the product is over a restricted collection of cycles called a reducing family.

We then introduce {\em $AB$-chains} which are cycles formed of segments of $AB$-cycles in
a restricted manner. We obtain in \Cref{lem:using_AB_chains} and \Cref{cor:factor_AB_chain} that
$$
\left(V^\ab\cdot \Theta_\tau\right) \: \big| \: (1\pm z) V_\tau
$$
where $z$ is the image in $G$ of such a cycle.

In \Cref{lem:AB_gen} we show that the homology classes of $AB$-chains and $AB$-cycles suffice to
generate $G$ and use this to show that, by applying \Cref{cor:factoring_families}
 and \Cref{cor:factor_AB_chain} over all $AB$-cycles
and $AB$-chains, we can establish
$$
\left(V^\ab\cdot \Theta_\tau\right) \: \big|\: V_\tau
$$
which completes the proof.

\subsection{Cycles and the factorization of $V^\ab$}
The map $A\colon F\to F$ is a permutation and we call its cycles the \define{$\bs{AB}$-cycles} of $\tau$. Recall that \emph{dual cycles} and \emph{dual paths} are directed cycles and paths respectively in the dual graph $\Gamma$. Each
$AB$-cycle $c$ determines a unique dual path whose $\Gamma$-edges correspond to the $\tau$-faces in $c$. We will often speak of these two objects interchangeably. In this language, an $AB$-cycle is a directed cycle in $\Gamma$ that makes only \emph{anti-branching} turns in the sense of \Cref{sec:flowgraph}.

A $\tau$-face $f\in F$ is dual to a $\Gamma$-edge pointing into the
tetrahedron for which $f$ is a bottom face, and we can label this tetrahedron (which is
dual to a $\Gamma$-vertex) by its bottom $\tau$-edge which is exactly $\varepsilon(f)$.
We say that an $AB$-cycle {\em passes through the faces}
$f,Af,A^2f,\ldots$ and, with a slight abuse of notation, {\em encounters the $\tau$-edges}
$\varepsilon(A^if)$.

Note that a $\tau$-face $f$ belongs to a unique $AB$-cycle, while
each $\tau$-edge is encountered twice by $AB$-cycles (possibly the same one). This
corresponds to the fact that the dual graph $\Gamma$ is $4$-valent.

Each $AB$-cycle $c$ determines a homology class in $M$ and hence an image $g=[c]\in G$. 
Let $k=k(c)$ denote the length of the cycle, or the number of $f\in F$ which it passes
through. This structure allows us to compute $V^\ab$:

\begin{lemma} \label{lem:computing_correction}
Let $c_1,\ldots,c_n$ denote the $AB$-cycles of $\tau$, $k_i$ their lengths, and
$g_i=[c_i]$ their images in $G$. Then  we may identify
\[
\ab (\widetilde \tau) \cong \bigoplus_{i=1}^n \frac{\Z[G]}{ \left(1+(-1)^{k_i+1}g_i \right)},
\]
where the map from the $i$th copy of $\Z[G]$ into $\mathcal{AB}(\widetilde \tau)$ maps $1$
to any (fixed) face in $c_i$. In particular, 
\[
V^\ab = \prod_{i=1}^n \left(1 + (-1)^{k_i+1}g_i \right)
\]
up to a unit factor in $\Z[G]$. 
\end{lemma}

\begin{remark}
We will often write the expression $(1 + (-1)^{k_i+1}g_i)$ as
    $(1\pm g_i)$, for brevity.
\end{remark}

\begin{proof}
 In this proof we abbreviate $\ab:=\ab(\wt\tau)$. Fix a face $f_i$ for each cycle $c_i$ and define a map
  $$
  \gamma\colon  \Z[G]^n \to \Z[G]^F
  $$
  as in the statement, taking the generator of the $i$-th copy of $\Z[G] $ to
  $f_i$. Let $\hat\gamma$ denote $\gamma$ composed with the surjection $\Z[G]^F \to \ab$.
  The map $\hat\gamma$ is surjective because, in $\ab$, the identity $f = -Af$ holds (since
  $L^\ab(f) = f+Af $ is a relation), so one face from each cycle suffices to generate the rest.

  Lifting the permutation $A$ to the faces in the universal cover, it acts as a
  translation on the lift of each cycle. Thus if we denote  
  the distinguished lift of $f_i$ by $f_i$ again,  we have the $A$-orbit $ f_i \to A f_i \to \cdots A^{k_i}
   f_i$ where $A^{k_i} f_i$ is just $g_i \cdot  f_i$, the translate of the lift $ f_i$
  by the element of the deck group associated to the cycle.

Since in $\ab$, we have $A^j f_i = -A^{j+1} f_i$, 
we obtain the relation $g_i\cdot f_i = (-1)^{k_i} f_i $. Hence $\hat\gamma$
factors through the quotient $\oplus_{i=1}^n \Z[G]/(1\pm g_i)$. 
An inverse map is easily constructed by appealing to the same relations, and so
we obtain the desired description of $\ab$.

This gives us a new  presentation of $\ab$, namely $\Z[G]^n \to \Z[G]^n \to \ab$
where the matrix is diagonal with $(1\pm g_i)$ along the diagonals. Since the Fitting
ideal of $\ab$ is independent of presentation, 
the expression for $V^\ab$ follows.   
\end{proof}

\subsection{Diagrams and presentations}
It is helpful to organize our maps in the following commutative diagram: 

\begin{proposition}\label{prop:cd}
We have the following commutative diagram with exact rows and columns:
\[
\begin{CD}
 \Z[G]^F  @> \varepsilon >>  \Z[G]^E @>>> 0\\
 @V L^\ab VV     @V L VV          \\
 \Z[G]^F @> L^\triangle >> \Z[G]^E@>>>  \e^\triangle(\widetilde \tau)@>>> 0\\
 @VVV     @VVV        @| \\
 \ab(\widetilde \tau) @> \overline{L^\triangle}>> \e(\widetilde \tau) @>>>
 \e^\triangle(\widetilde \tau) @ >>> 0\\
 @VVV @VVV\\
0 @. 0
  \end{CD}
\]
\end{proposition}

\begin{proof}
  Surjectivity of $\varepsilon$ is immediate from the definition, and commutativity of the
  top square is \Cref{eq:Lab L}.

  Exactness of the middle row is the definition of $\e^\triangle$, and  exactness of the first
  two columns is the definition of $\ab$ and $\e$. 

  Commutativity of the bottom left square is the definition of $\overline{L^\triangle}$.

  The definition of the  arrow $\e(\wt\tau) \to \e^\triangle(\wt\tau)$, and exactness of
  the bottom row, follow from a short diagram chase. 
\end{proof}

\subsubsection{Presenting $\e$:}
We can now consider an alternative presentation of $\e=\e(\wt\tau)$:
\begin{proposition} \label{prop:alt_pres_edge}
Consider the module homomorphism 
\[
\mathcal{L} \colon \Z[G]^F \oplus \Z[G]^F \to \Z[G]^F \oplus \Z[G]^E
\]
given in block form by the $(F+E)\times(2F)$ matrix
\[
\mathcal{L} =
 \left[
\begin{array}{c|c}
L^\ab & I \\
\hline
0 & L^\triangle
\end{array}
\right].
\]
Then $\e(\wt \tau)$ is isomorphic to the cokernel of $\mathcal{L}$.
\end{proposition}

\begin{proof}
Denote the cokernel of $\mc L$ by $\e'$ and note $\e'$ is the free module $\Z[G]^F \oplus \Z[G]^E$ modulo two relations for each face $f$:
\begin{align*}
&L^\ab(f) = f + f' =0 \\
&f + L^\triangle(f) = 0.
\end{align*}

Hence, if 
$e = \varepsilon(f)$, then
\begin{align*}
L(e) &= L^\triangle(f) + L^\triangle(f') \\
& =  (f+ L^\triangle(f)) + (f'+L^\triangle(f')) -  L^\ab(f)  \\
& \in  \im(\mc L),
\end{align*}
and so the factor inclusion $\Z[G]^E \to \Z[G]^F \oplus \Z[G]^E$ descends to a module homomorphism $\e \to \e'$. 

An inverse homomorphism is induced by first defining a homomorphism $\Z[G]^F \oplus \Z[G]^E \to \Z[G]^E$ by 
\[
(f,e) \to e - L^\triangle(f).
\] 
Since this maps the first relation on $\e'$ to a tetrahedron relation of $\e$ and the second relation on $\e'$ to $0$, it induces a map $\e' \to \e$. This is easily seen to be the inverse of the one defined above.
\end{proof}

From this we obtain one direction of our main theorem:

\begin{proposition} \label{prop:V_divides}
\[
V_\tau \: | \: \left (V^\ab \cdot \Theta_\tau \right)
\]
\end{proposition}
\begin{proof}
The fitting ideal $(V_\tau)$ of $\e$ is independent of the presentation of $\e$. Hence, we get that $(V_\tau)$ is also the ideal generated by the minors of $\mc L$ of size $|F| + |E|$. In particular $V_\tau$ divides the product of $\det (L^\ab)$ with each minor of $L^\triangle$ of size $|E|$. Hence $V_\tau \: | \: \left (V^\ab \cdot \Theta_\tau \right)$.
\end{proof}

\begin{remark}
This is actually a general fact: If $P\to Q \to R \to 0$ is an exact sequence of
finitely generated modules then $\mathrm{Fitt}(P) \cdot \mathrm{Fitt}(R) \subset \mathrm{Fitt}(Q)$. 
For example, this can be derived from \cite[Chapter 3, Exercise 2]{northcott2004finite}. 
  We can apply this to the sequence   $\ab\to\e\to\e^\triangle\to 0$ from \Cref{prop:cd}. 
We chose to give an explicit proof for completeness.
\end{remark}

\subsubsection{Presenting $\e^\triangle$:}
The following presentation of $\e^\triangle(\wt\tau)$ gives us a new description of
$\Theta_\tau$. 

\begin{lemma} \label{lem:alt_pres_face}
Let $c_1, \ldots, c_n$ be the collection of all $AB$-cycles and let $f_i$ be a face of $c_i$. Then
$\e^\triangle(\wt \tau)$ is presented by the $|E| \times (n+|E|)$ matrix
\[
[ L^\triangle(f_1) \: |\:  \ldots  \: |\:  L^\triangle(f_n) \: | \: L].
\]
Hence, $\Theta_\tau$ is the gcd of the $|E| \times |E|$ minors of this matrix.
\end{lemma}

Here we have used the following notation: If $A$ and $B$ are matrices with the same number of rows, then $[A|B]$ denotes the matrix whose columns are the columns of $A$ follows by the columns of $B$. If $A$ or $B$ is a vector, then we interpret it as a column vector. 

\begin{proof}
This matrix defines a map  $M\colon \Z[G]^{|E|+n}\to \Z[G]^E$ and it suffices to show that its
image equals the image of $L^\triangle$.

Returning to the diagram of \Cref{prop:cd}, since $\varepsilon$ is surjective, the image of $L$ must be contained in the image of
$L^\triangle$. Thus $\im(M) \subset \im( L^\triangle)$.

Now for any $f\in F$ we have $L^\triangle(f+Af) = L(\varepsilon(f))$. Hence
$L^\triangle(Af)$ is contained in the span of $\im(L)$ and $L^\triangle(f)$. Repeating inductively,
the entire cycle of $f_1$ has $L^\triangle$-image in the span of $\im(L)$ and
$L^\triangle(f_1)$.
Applying to all the cycles we see that the image of $L^\triangle$ is contained in the
image of $M$. 

The statement about $\Theta_\tau$ follows again from the fact that the Fitting ideal is
independent of presentation.
\end{proof}

\subsection{Factoring $V_\tau$}
We next turn to describing our primary means of factoring the veering polynomial.

\subsubsection{The $AB$-cycle equation}
The cycle structure of the permutation $A$ gives rise to a useful identity relating $L$
and $L^\triangle$. This uses a combination of the discussion in the proof of \Cref{lem:computing_correction} and the identity
\Cref{eq:Lab L}. 

\begin{lemma}[$AB$-cycle equation]\label{lem:cycle eq}
  Let $c$ be an $AB$-cycle of $F$, $e(c)$ the set of $\tau$-edges encountered by $c$, 
  $f$ one of the $\tau$-faces of $c$, and $g$ the class $[c]$ in $G$. Then the
 following identity holds: 
  \begin{equation}\label{eq:ab_cycles_eq}
\sum_{e\in e(c)} \alpha_e\cdot L(e) = (1\pm g) \cdot L^\triangle(f)
  \end{equation}
  where $\alpha_e$ is a unit in $\Z[G]$ if $e$ appears once along $c$, and a sum of two
  units if $e$ appears twice. 
\end{lemma}

\begin{proof}
We identify  $f$ with its selected lift in $\aM$, and enumerate a lift of the
cycle as $f_i = A^i f$, $i=0,\ldots,k$ where $k$ is the period of $c$ - so that $f_k =
A^kf = gf$. Then each $\wt\tau$-edge $\varepsilon(f_i)$ can be written as
$a_i e_i$ where $e_i$ is one of the elements of $e(c)$ (again identified with its selected
lift in $\aM$) and $a_i\in G$.
The identity $L\circ \varepsilon = L^\triangle\circ (I+A) $  gives us
$$
a_i L(e_i) = L^\triangle(f_i) + L^\triangle(f_{i+1}). 
$$
If we combine these in an alternating sum we therefore obtain, after cancellations,
$$
\sum_{i=0}^{k-1} (-1)^i a_i L(e_i) = L^\triangle(f_0) + (-1)^{k-1} L^\triangle(f_k).
$$
Now since $f_k=g  f$, the right hand side is $(1\pm g)
L^\triangle(f)$.
On the left side, if $e$ appears once as $e_i$ then we have a term of the form $\pm a_i
L(e)$, and if it appears twice as $e_i$ and $e_j$ we combine two terms to get $(\pm a_i \pm
a_j) L(e)$. This gives the desired statement. 
\end{proof}

As an immediate corollary we have: 
\begin{corollary}\label{cor:once-edge}
  Let $c$ be an $AB$-cycle of $F$ and $e$ an edge appearing exactly once in $c$. Then
  \begin{equation}\label{eq:once-edge}
    L(e) = u(1\pm g) \cdot L^\triangle(f)+ \sum_{e'\in e(c)\ssm\{e\}} \beta_{e'} \cdot L(e').
  \end{equation}
where $g=[c]\in G$, $f$ is a face of $c$, $u$ is a unit in $\Z[G]$, and $\beta_{e'}$ is sums of at most two units in $\Z[G]$. 
\end{corollary}

\subsubsection{Factoring with reducing families of cycles}
A collection $\C$ of $AB$-cycles forms a \define{reducing family of cycles} if it can be
ordered $\C= \{c_1, \ldots, c_k\}$ so that for each $1\le i \le k$, there is a $\tau$-edge $e_i$ through which $c_i$ passes exactly once and which is not visited by any $c_j$ for $j>i$.
Such an ordering is not necessarily unique. However, given such an ordering, we call it the \emph{preferred order}. We call the $\tau$-edges $e_i$ the \emph{distinguished edges} of the cycles in the family (these too
may not be unique, but we can make an arbitrary choice). In the results and proofs to follow, given a reducing family $\{c_1,\dots, c_k\}$ we will often write $g_i=[c_i]\in G$.

We first show that reducing families are not hard to find: 
\begin{lemma}[Finding families] \label{lem:finding_family}
Any proper subset of $AB$-cycles is a reducing family.
\end{lemma}

\begin{proof}
For the argument, we think of $AB$-cycles as directed cycles in the dual graph $\Gamma$,
so that $\tau$-edges of a cycle correspond to $\Gamma$-vertices.

A collection of directed cycles in $\Gamma$ that crosses every vertex either zero or two times must
in fact cross every vertex, since $\Gamma$ is strongly connected by \Cref{lem:sc_dual}. Thus, a proper subset
$\C$ of $AB$-cycles must cross some vertex $v_1$ of $\Gamma$ exactly once. 
Let $c_1$ be the cycle that contains $v_1$. Now apply the same argument to $\C\ssm \{c_1\}$
and continue inductively.
\end{proof}

Reducing families provide factorizations of $V_\tau$ of the following type: 

\begin{proposition} \label{prop:facting_out_families}
For any reducing family of cycles $\C$, we have the factorization (up to a unit $\pm h \in \Z[G]$)
\[
V_\tau = \left(\prod_{c_i \in \C} (1\pm g_i)\right) \cdot  \det[ L^\triangle(f_1) \: |\:  \cdots  \: |\:  L^\triangle(f_r) \: | \: L'],
\]
where $f_i$ is the (unique) face in the cycle $c_i \in \C$ such that $\varepsilon(f_i)$ is
the distinguished edge $e_i$ and $L'$ is the matrix obtained from $L$ by removing the
columns corresponding to the distinguished edges of $\C$.
\end{proposition}

\begin{proof}
Let $\C = \{c_1, \ldots c_r\}$ be written in preferred order, and reorder the columns of $L$ so that the $i$th column
(for $i \le i \le r$) corresponds to the distinguished edge $e_i$ of $c_i$. Since $e_1$
appears exactly once in $c_1$, \Cref{cor:once-edge} of the 
$AB$-cycle equation implies that
we can perform a column replacement on the first column of $L$ to obtain  (up to units)
\begin{align*}
V_\tau &= \det( [L(e_1) \: | \: L_1' ]) \\
&= \det([ (1\pm g_1) L^\triangle(f_1) \: | \: L_1' ]) \\
&= (1\pm g_1) \cdot \det([ L^\triangle(f_1) \: | \: L_1' ]),
\end{align*}
where $L'_1$ denotes $L$ with its first column removed.

Now since $\C$ is a reducing family, the cycle $c_2$ does not encounter $e_1$. This is to
say that  \Cref{eq:ab_cycles_eq} for the cycle $c_2$ does not involve
$L(e_1)$. Hence, again using that $e_2$ appears exactly once in $c_2$, we can apply a
column replacement on the second column of $[L^\triangle(f_1) \: | \: L_1' ]$ to obtain 
\[
V_\tau = (1\pm g_1)(1\pm g_2) \cdot \det([ L^\triangle(f_1) \: | \: L^\triangle(f_1) \: | \: L_{1,2}' ], 
\]
where $ L_{1,2}'$ is obtained by removing the first two columns from $L$. Continuing in this manner proves the proposition.
\end{proof}

The proposition along with \Cref{lem:alt_pres_face} gives the following:
\begin{corollary} \label{cor:factoring_families}
For each reducing family of cycles $\C$,
\[
\left (  \Theta_\tau\cdot \prod_{c_i \in \C} (1\pm g_i)  \right ) \: \big| \: V_\tau.
\]
\end{corollary}

\begin{proof}
By \Cref{prop:facting_out_families}, $V' = V_\tau /  \prod_{c_i \in \C} (1\pm g_i)$ defines an element of $\Z[G]$ up to a unit, and it suffices to show $\Theta_\tau | V'$. For this, note that the expression for $V'$ given by \Cref{prop:facting_out_families} is a minor appearing in the definition of $\Theta_\tau$ from \Cref{lem:alt_pres_face}. \end{proof}

\subsubsection{Factoring with $AB$-chains}\label{sec:abchains}
We now consider 
a class of dual cycles which will be useful for factoring $V_\tau$. 

First, define the \define{$\bs{AB}$ length} of a dual cycle $c$ to be its number of branching turns.
For any dual cycle $c$ of positive $AB$ length there is a  decomposition of $c$ as a concatenation of dual paths $c=(p_1,\dots, p_k)$ such that each anti-branching turn of $c$ is interior to some $p_i$, and where $k$ is the $AB$-length of $c$. This decomposition is unique up to cyclic permutation and we call it the \textbf{$\bs{AB}$ decomposition} of $c$. 
We say such a $c$ is an \textbf{$\bs{AB}$-chain} if it has the following two properties with respect to its $AB$ decomposition $c=(p_1,\dots, p_k)$:
\begin{enumerate}
\item (${AB}$ simple) each $p_i$ is a proper subpath of some $AB$-cycle $c_i$ such that $c_i\ne c_j$ for $i \neq j$, and
\item (Endpoint simple) each endpoint of each $p_i$ is visited exactly once by $c$.
\end{enumerate}

Note that if $c$ is a simple (i.e. embedded) dual cycle with $AB$ length $k>0$ then $c$ is automatically endpoint simple, although we allow nonsimple $AB$-chains. The endpoints of the paths $p_1,\dots, p_k$ in the $AB$ decomposition correspond to the branching turns of $c$, and we call the corresponding edges of $\tau$ the \textbf{branching $\bs{\tau}$-edges} of $c$. If $c_i$ is the $AB$-cycle containing $p_i$, we say $c$ \textbf{uses} the $AB$-cycles $c_1,\dots, c_k$.

Note that an $AB$-cycle has $AB$ length $0$ and so the above definitions do not apply.

 \Cref{lem:AB_gen} below will furnish many useful $AB$-chains, but for now let
 us explain how these objects can be used to factor (a multiple of) $V_\tau$. 

\begin{lemma} \label{lem:using_AB_chains}
Suppose that $c$ is an $AB$-chain which has branching $\tau$-edges $e_1, \ldots, e_k$ and
which uses distinct $AB$-cycles $c_1, \ldots, c_k$. Let $f_i$ be the face in $c$ with
$\varepsilon(f_i) = e_i$. Then, up to a unit in $\Z[G]$,
\[
(1\pm z) \cdot V_\tau = \left(\prod_{1\le i \le k}(1\pm g_i) \right)\cdot  \det[ L^\triangle(f_1) \: |\:  \cdots  \: |\:  L^\triangle(f_k) \: | \: L_c'], 
\]
where $z = [c]$, $g_i = [c_i]$, and $L_c'$ is the matrix obtained from $L$ by removing the columns corresponding to $e_1, \ldots, e_k$. 
\end{lemma}

\begin{proof}
Let $(p_1, \dots ,p_k)$ be the $AB$ decomposition of $c$. We label $\tau$-edges so that $p_i$ is an oriented path from $e_i$ to $e_{i+1}$ with indices taken mod $k$. (Here, we are using the correspondence between vertices of $\Gamma$ and edges of $\tau$.)

Lift $c$ to $\aM$ so that it begins with a fixed lift of $e_1$ 
which we also denote $e_1$. Then each $p_i$ lifts to a path from $a_ie_i$ to $a_{i+1}e_{i+1}$
where $a_i$ denotes the group element translating a fixed lift of $e_i$ to
the one encountered by our lift of $c$. Thus $a_1=1$ and the last $\tau$-edge is $a_{k+1}e_1$, where
$a_{k+1}$ is exactly $z = [c]$. 

First suppose that $k\ge2$. (The case of $k=1$ is easier and handled later.)
Note that because $c$ is $AB$ simple and $k\ge 2$, each $e_i$ is encountered by two distinct $AB$-cycles $c_i,c_{i+1}$ and therefore no single $AB$-cycle encounters $e_i$ twice. Hence applying the $AB$-cycle equation (\Cref{lem:cycle eq}) to $c_i$ we obtain 
\begin{equation}\label{eq:pi term}
\pm a_iL(e_i) \pm a_{i+1}L(e_{i+1}) + \sum_{e'} \alpha_{e'}L(e')= h_i (1\pm g_i)L^\triangle(f_i),
\end{equation}
where $h_i\in G$ corrects for the difference between our lift of $p_i$ here and the 
lift of $c_i$ from \Cref{lem:cycle eq}. In \Cref{eq:pi term} above, the sum on the lefthand side is over all $e'$ in $e(c_i)\ssm\{e_i,e_{i+1}\}$, where $e(c_i)$ is the set of $\tau$-edges encountered by $c_i$ and each $\alpha_{e'}$ is a sum of at most two units in $\Z[G]$.
For the last segment $p_k$ the lefthand side becomes $\pm a_kL(e_k) \pm z L(e_1) + \sum_{e'} \alpha_{e'}L(e')$. 

We claim that no edge encountered by $c_i$ other than $\{e_i,e_{i+1}\}$ is a branching edge of $c$. 
In other words, for each $1\le i\le k$ and for each $e'\in e(c_i)\ssm\{e_i,e_{i+1}\}$, $L(e')$ is a column of $L_c'$. Otherwise we would have $e'=e_j$ for some $j\in\{1,\dots,k\}\ssm\{i,i+1\}$. This would force $e_j$ to be part of distinct $AB$-cycles $c_i, c_{j-1},c_j$, which is impossible as each $\tau$-edge is encountered by at most two $AB$-cycles.

This enables us to make the following computation, in which the equalities are taken up to
multiplication by units $\pm h$: 
\begin{align*}
&\left(\prod_{1\le i \le k}(1\pm g_i) \right)\cdot  \det[ L^\triangle(f_1) \: |\:  \cdots  \: |\:  L^\triangle(f_k) \: | \: L_c'] \\
=&  \det[ (1\pm g_1) L^\triangle(f_1) \: |\:  \cdots  \: |\: (1\pm g_k) L^\triangle(f_k) \: | \: L_c'] \\
=& \det[ \pm L(e_1) \pm a_2 L(e_2) \: |  \cdots  \: |\: \pm a_kL(e_k) \pm z L(e_1)  \: | \: L_c']. 
\end{align*}
In the last equality we use the substitution of \Cref{eq:pi term} in each column and then 
use the columns of $L_c'$ to remove all the terms that do involve the branching $\tau$-edges $e_i$.

Now adding (or subtracting) the $i$th column to (from) the $(i+1)$st column, starting at $i=1$, the
columns are replaced by  $\pm L(e_1) \pm a_{i+1} L(e_{i+1})$ for $i=1,\ldots,k-1$, and
$\pm L(e_1) \pm zL(e_1)$ for the $k$-th column. Thus we have: 
\begin{align*}
&\det  [ \pm L(e_1) \pm a_2 L(e_2) \: |\:  \cdots  \: |\: \pm L(e_1) \pm z L(e_1)  \: | \: L_c']\\
=& \det[ \pm L(e_1) \pm a_2 L(e_2) \: |\:  \cdots  \: |\: \pm (1\pm z) \cdot L(e_1)  \: | \: L_c']\\
=& (1 \pm z) \cdot \det[ \pm L(e_1) \pm a_2 L(e_2) \: |\:  \cdots  \: |\: L(e_1)  \: | \: L_c']\\
=& (1 \pm z) \cdot \det[ \pm a_2L(e_2)  \: |\:  \cdots  \: |\: L(e_1)  \: | \: L_c']\\
=& (1 \pm z) \cdot V_\tau.
\end{align*}
In the penultimate line we used the $k$-th column to remove the $L(e_1)$ terms from the
previous columns.
In the last line we note that the determinant is, up to unit multiples, equal to the
determinant of $L$, or $V_\tau$. 

Finally, when $k=1$ we have that $p_1$ is a proper subpath of $c_1$ which starts and ends at $e_1$ and forms the directed cycle $c$. In particular, $c_1$ encounters $e_1$ twice and so in place of \Cref{eq:pi term} we have
\begin{equation}\label{eq:pi term_2}
\pm L(e_1) \pm zL(e_{1}) + \sum_{e'} \alpha_{e'}L(e')= h_1 (1\pm g_1)L^\triangle(f_1),
\end{equation}
where notation is as before. Then proceeding exactly as above, we have

\begin{align*}
(1\pm g_1) \cdot  \det[ L^\triangle(f_1) \: |\:  L_c'] 
&=  \det[ (1\pm g_1) L^\triangle(f_1) \: |\:  L_c'] \\
&= \det[ \pm L(e_1) \pm z L(e_1) \: | L_c'] \\
& = (1\pm z) \cdot V_\tau,
\end{align*}
and the proof is complete.
\end{proof}

As a corollary we obtain the following factorization: 
\begin{corollary} \label{cor:factor_AB_chain}
For each $AB$-chain $c$ with $z = [c]$:
\[
\left (V^{AB }\cdot \Theta_\tau \right ) \: \big | \: (1\pm z) \cdot V_\tau.
\]
\end{corollary}

\begin{proof}
  Let $\C_c$ be the collection of $AB$-cycles not used by $c$. Since $AB$-chains are
  nonempty by definition, $\C_c$ is a proper collection and hence a reducing family by
  \Cref{lem:finding_family}. Note that $\C_c$ may be empty if $c$ uses all of the $AB$-cycles in $\Gamma$.

Each branching $\tau$-edge of $c$ is contained in two $AB$-cycles used
by $c$. Since a $\tau$-edge is encountered by at most two $AB$-cycles, 
the branching $\tau$-edges are not encountered by any of the $AB$-cycles of
$\C_c$. 

Enumerate the branching $\tau$-edges of $c$ by $e_1,\ldots,e_k$ as above, and let
$e_{k+1},\ldots,e_n$ denote the distinguished $\tau$-edges of the cycles of $\CC_c$. Let
$f_i$ be the associated $\tau$-faces as before.

Let $L''$ be the matrix obtained from $L$ by removing the columns corresponding to $e_1,\dots, e_n$. We see that 
\begin{align*}
(1\pm z) \cdot V_\tau &=\Big(\prod_{1\le i \le k}(1\pm g_i)\Big) \cdot  \det[ L^\triangle(f_1) \: |\:  \cdots  \: |\:  L^\triangle(f_k) \: | \: L_c'] \quad \text{ (by \Cref{lem:using_AB_chains})}\\
&= \Big(\prod_{1\le i \le k}(1\pm g_i) \Big)\cdot  \Big(\prod_{k+1\le j \le n}(1\pm g_j)\Big) \cdot  \det[ L^\triangle(f_1) \: |\:  \cdots  \: |\:  L^\triangle(f_n) \: | \: L''] \\
&= V^{AB} \cdot  \det[ L^\triangle(f_1) \: |\:  \cdots  \: |\:  L^\triangle(f_n) \: | \: L''],
\end{align*}
 The second equality above follows from techniques in the proof of \Cref{prop:facting_out_families}. More specifically, we repeatedly apply \Cref{cor:once-edge} using the facts that $\CC_c$ is a reducing family and that no $AB$-cycle in $\CC_c$ encounters any of $e_1,\dots,e_k$.
  
 As in the proof of \Cref{cor:factoring_families}, $[
  L^\triangle(f_1) \: |\:  \cdots  \: |\:  L^\triangle(f_n) \: | \: L'']$ is a minor of the presentation matrix for $\mc E^\triangle(\wt \tau)$ from \Cref{lem:alt_pres_face}. It follows that $\Theta_\tau$ divides its determinant.
  \end{proof}

\subsection{$\bs{AB}$-chains and homology}
Until now we have not proved the existence of any $AB$-chains, but in fact there are
enough of them to generate the homology of $\Gamma$:

\begin{lemma}\label{lem:AB_gen}
The $AB$-chains and $AB$-cycles together generate $H_1(\Gamma)$.
\end{lemma}

\begin{proof}
We first show that any dual cycle $c$ is homologous to an integer linear combination of $AB$ cycles and $AB$ simple dual cycles. 
For this, we induct on the {$AB$ length} of $c$. If $c$ has $AB$ length 0 then $c$ is an $AB$-cycle; also, if $c$ has $AB$ length 1 then $c$ is easily seen to be an $AB$-chain, so is $AB$-simple.

Now suppose $c$ has $AB$ length $k>1$ with $AB$ decomposition $c=(p_1,\dots,p_k)$, and suppose $p_1$ and $p_j$ are part of the same $AB$-cycle $g$.  Now:
\begin{itemize}
\item Let $g_1$ be the directed subpath of $g$ from the initial point of $p_1$ to the terminal point of $p_j$. Note that this includes the edges $p_1$ and $p_j$.
\item Let $g_2$ be the directed subpath of $g$ from the initial point of $p_j$ to the terminal point of $p_1$.
\item Let $G_1=(g_1,p_{j+1},\dots,p_k)$.
\item Let $G_2= (g_2, p_2, \dots, p_{j-1})$.
\end{itemize}
\begin{figure}[htbp]
\begin{center}
\includegraphics[height=1.5in]{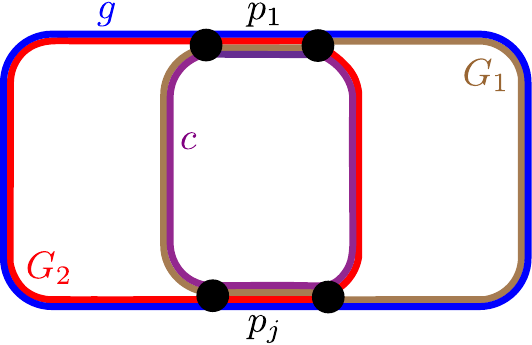}
\caption{Decomposing the dual cycle $c$ as $c=G_1+G_2-g$ as in the proof of \Cref{lem:AB_gen}, where we have drawn the cycles so that orientations are clockwise in the page. }
\label{fig:decomposing_c}
\end{center}
\end{figure}
Then $c=G_1+G_2-g$ in $H_1(\Gamma)$ (see \Cref{fig:decomposing_c}). Note that  $G_1, G_2$ both have smaller $AB$ length than $c$, as does $g$ whose $AB$ length is $0$. 
By induction, $G_1$ and $G_2$ are each homologous to a $\Z$-linear combination of $AB$ cycles and $AB$ simple dual cycles whereby $c$ is as well.

We next claim that any $AB$ simple dual cycle is also endpoint simple unless it is the concatenation of two $AB$-cycles. 
Combined with the discussion above, this will show that each directed cycle is an integer linear combination of $AB$-chains and $AB$-cycles, proving the lemma by \Cref{lem:sc_dual}.

For this, suppose that $c$ is an $AB$ simple dual cycle. We may assume that the $AB$ length of $c$ is greater than $1$ since otherwise $c$ is an $AB$-chain or $AB$-cycle. As before, let $c=(p_1,\dots, p_k)$ be the $AB$ decomposition of $c$.
Let $c_i$ be the $AB$-cycle containing $p_i$. By $AB$ simplicity, each $c_i$ is distinct.

By symmetry it suffices to show that the terminal vertex $v$ of $p_1$ appears exactly once in $c$ unless $c$ is the concatenation of two dual cycles. Suppose $v$ appears again in the directed cycle $c$. We split the possibilities into four cases.

\emph{Case 1: $v$ belongs to $p_j$ for $j\ne 1,2$}. In this case $v$ is contained in the $AB$-cycle $c_j$ and by $AB$ simplicity $c_1, c_2, c_j$ are all distinct. This is impossible since any $\Gamma$-vertex is contained in at most two distinct $AB$-cycles.

\emph{Case 2: $v$ is an interior vertex of either $p_1$ or $p_2$}. In this case either $c_1$ or $c_2$ must meet the vertex $v$ twice, which is impossible since $v$ meets $c_1$ and $c_2$ once each and so meets neither twice. 

\emph{Case 3: $v$ is the initial vertex of $p_1$}. This implies $v$ is the terminal vertex of $p_k$, so $v$ belongs to the $AB$ cycles $c_k, c_1, c_2$. We conclude that $k=2$ so $c=(p_1,p_2)$ and each of $p_1$ and $p_2$ is an $AB$-cycle as claimed.

\emph{Case 4: $v$ is the terminal vertex of $p_2$}. A symmetric argument to Case 3 shows that  in this case $c$ is also a concatenation of two $AB$-cycles.

This proves our claim that every $AB$ simple dual cycle is either endpoint simple or the concatenation of two $AB$-cycles, completing the proof of the lemma.
\end{proof}

In the proof of \Cref{th:factorization!} we will also need this short lemma about relatively prime
elements in $\Z[G]$:
\begin{lemma}\label{lem:rel_prime}
  In the group ring $\Z[G]$,
$(1\pm g)$ and $(1\pm h)$ are relatively prime if $\langle g,h \rangle \le G$ has rank $2$.
\end{lemma}

\begin{proof}
Since $(1-g^2) = (1+g)(1-g)$, it suffices to prove that $(1-g)$ and $(1-h)$ are relatively prime. Let $n$ be the rank of $G$, and write $G =   \langle u\rangle \oplus \langle t \rangle \oplus \ \Z^{n-2}$ where $g = u^a$ and $h = u^b t^c$ 
for $a,c \ge 1$ and $b\ge 0$. Begin by factoring 
\[
1-g = 1- u^a = \prod_{d | a} \Phi_d(u),
\]
where $\Phi_d(u)$ is the $d$th cyclotomic polynomial in the variable $u$. Note that each $\Phi_d$ is irreducible in $\Z[G]$ since $R_d =\Z[G] / (\Phi_d) \cong \Z[G/\langle u \rangle] [\omega_d]$ is a domain. Here $\omega_d$ is a primitive $d$th root of unity.

Since $\Z[G]$ is a UFD, it suffices to show that $\Phi_d(u)$ does not divide $(1-h)$ for any $d$. If it did, then $(1-h)$ would be zero in the ring $R_d$. But by further quotienting $R_d \to \Z[G / \langle u \rangle]$, we have
\[
1-h = 1-t^c = \prod_{d | c} \Phi_d(t) \quad \text{ in } \Z[G/\langle u \rangle],
\] 
which is clearly nonzero since $t$ is primitive in $G/\langle u \rangle$.
\end{proof}

\subsection{Completing the proof}
We can now assemble the proof of this section's main theorem on factoring $V_\tau$.

\begin{proof}[Proof of \Cref{th:factorization!}]
For any $AB$-cycle $c$ with homology class $g$, Let $\C_c$ be the set of cycles other
than $c$, which by \Cref{lem:finding_family} is a reducing family. 
\Cref{cor:factoring_families} gives us
\[
\left (  \Theta_\tau \cdot \prod_{c_i \in \C_c} (1\pm g_i)\right ) \: \big| \: V_\tau, 
\]
and multiplying by $(1\pm g)$ gives
\[
\left ( V^{AB} \cdot \Theta_\tau \right ) \: \big | \: (1\pm g) \cdot V_\tau.
\]
This, together with \Cref{cor:factor_AB_chain}, implies that 
\begin{align}\label{eq:up_to_gcd}
\left ( V^{AB} \cdot \Theta_\tau \right ) \: \big | \: p \cdot V_\tau,
\end{align}
where $p$ is the gcd of all polynomials of the form $(1\pm g)$ for $g$ the homology class of an $AB$-cycle and $(1\pm z)$ for $z$ the homology class of an $AB$-chain. 
Since $\mathrm{rank}(H_1(M)) >1$, \Cref{lem:AB_gen} implies there are such homology classes that are independent in $H_1(M)$ (i.e. do not generate a rank $1$ subgroup). By \Cref{lem:rel_prime}, the corresponding polynomials are relatively prime and so $p$ is a unit in $\Z[G]$.
Combining this with \Cref{prop:V_divides} completes the proof.
\end{proof}

\begin{remark} \label{rmk:rank_1}
If $\mathrm{rank}(H_1(M)) =1$, so that $G=\langle t\rangle$,  the conclusion of \Cref{th:factorization!} still holds after possibly multiplying $V_\tau$ by $(1\pm t)$. To see this, first note that  \Cref{prop:V_divides} is unaffected by the rank assumption. Second, as in the proof of \Cref{th:factorization!}, \Cref{eq:up_to_gcd} holds where $p$ is the gcd of polynomials of the form $(1\pm t^{k_i})$, where each $t^{k_i}$ corresponds to an $AB$-chain or $AB$-cycle. The homology classes of $AB$-chains and $AB$-cycles still generate $H_1(M)$ by \Cref{lem:AB_gen}, so the gcd of the $k_i$ is 1. Factoring $1-t^{k_i}$ and $1+t^{k_i}$ into cyclotomics, we see that the only possible divisors of all the $1\pm t^{k_i}$ are $\Phi_1=1-t$ and $\Phi_2=1+t$. However, at most one of $\Phi_1,\Phi_2$ is a common divisor of the $1\pm t^{k_i}$ since $(1-t)\nmid (1+t^{k_i})$ and $(1+t) \nmid (1-t^{k_i})$ for $k_i$ odd. 
\end{remark}

\section{The fibered case and the Teichm\"uller polynomial} \label{sec:layered}
In this section, we examine the veering polynomial $V_\tau$ in the case of a layered triangulation $\tau$, when $\tau$ is canonically associated to a \emph{fibered face} $\mathbf{F}_\tau$ of the Thurston norm ball (as in \Cref{th:omni}). 
Ultimately, we demonstrate the connection between the veering polynomial and McMullen's Teichm\"uller polynomial. We first do this for 
the manifolds admitting veering triangulations
(\Cref{th:modules_agree})
and then derive the general case (\Cref{prop:Teich_punctured}) using the fact that any hyperbolic fibered manifold admits a veering triangulation after puncturing along the singular
orbits of its associated suspension flow.
We conclude by showing how these results combine with those of the previous section to give a `determinant formula' for computing the Teichm\"uller polynomial using only the veering data (\Cref{cor:veering_teich}).

Throughout this section, all veering triangulations are layered.

\subsection{Teichm\"uller polynomials from veering triangulations}
Fix a layered veering triangulation $\tau$ of a manifold $M$ and denote its 
associated fibered face by $\bf F = \bf F_\tau$. Since our goal is to introduce a veering interpretation of the Teichm\"uller polynomial, we closely follow McMullen's paper \cite{mcmullen2000polynomial} 
and refer the reader there for the complete construction.

Let $\mathcal{L}$ be the $2$-dimensional expanding lamination of $M$ 
associated to $\bf F$, and $\widetilde {\mathcal{L}}$ its lift to 
$\aM$, the free abelian cover of $M$ with deck group $G$.
We recall that up to isotopy $\mathcal{L}$ is obtained by 
suspending the expanding lamination of the monodromy for any
fiber in $\R_+\bf F$  \cite[Corollary 3.2]{mcmullen2000polynomial}.

Associated to $\wt{\mc L }$ and its $G$ action is a \define{module of transversals} $T(\widetilde{\mathcal{L}})$, which like the modules of \Cref{sec:veering} is a $\Z[G]$-module. We will not require its precise 
definition here but instead rely on several of its properties that will be recalled below.
Let $\Theta_{\bf F}$ be the Teichm\"uller polynomial associated to $\bf F$. By definition, $\Theta_{\bf F}$ is the gcd of the elements of the Fitting ideal for $T(\widetilde{\mathcal{L}})$ and is hence well defined up to a unit $\pm g \in \Z[G]$.

We will show in \Cref{th:modules_agree} that $\Theta_{\bf F} = \Theta_\tau$ up to a unit, where $\Theta_\tau$ is the taut polynomial for $\tau$ defined in \Cref{sec:taut}.

Before stating \Cref{th:modules_agree}, we address two technical issues arising from the different conventions used to define Teichm\"uller and taut polynomials. First, our definition of the face module $\e^\triangle(\wt \tau)$ from \Cref{sec:taut} is most naturally related to the module of transversals for the \emph{contracting lamination} dual to $\mc L$. To remedy this, we define a $\Z[G]$-module $\e^\bigtriangledown(\wt \tau)$ that is a variation on the definition of the face module. For this, we note that each face $f$ lies at the top of a unique tetrahedron and we refer to the edge $\bf{t}$ at the top of this tetrahedron as the \define{top} edge of $f$. The $\Z[G]$-module $\e^\bigtriangledown(\wt \tau)$ is defined exactly as $\e^\triangle(\wt \tau)$ but with the following modification: for each face $f$ with top edge $\bf{t}$, instead of the relation from \Cref{eq:face_rel}
we use the relation
\begin{align}\label{eq:top}
\bf{t} = {\bf w} + {\bf z},
\end{align}
where $\bf{w},\bf{z}$ are the other two edges of $f$.  As shown by Parlak \cite{Parlak1}, the $\Z[G]$-modules $\e^\triangle(\wt \tau)$ and $\e^\bigtriangledown(\wt \tau)$ are isomorphic. Indeed, 
an isomorphism is induced by the map $\Z[G]^E\to \Z[G]^E$ which sends an edge $\bf e$ to itself if it is right veering and to $-\bf{e}$ if it is left veering. For each face $f$, this map exchanges the relation from \Cref{eq:face_rel} with the relation from \Cref{eq:top}.
This follows easily 
from two facts: $(1)$ every face $f$ of $\tau$ has edges of each veer, and $(2)$ the bottom and top edges of $f$ have the same veer. 
The first fact is immediate from the veering definition and the second follows from \Cref{rmk:top_of_fan}. Thus we get a $\Z[G]$-module homomorphism which is easily seen to be invertible.

Second, to define his module $T(\widetilde{\mathcal{L}})$, McMullen uses the action of $G$ on transversals by taking preimages under deck transformations, whereas our action of $G$ on the edges of $\wt \tau$ (which form a special class of transversals) is by taking images. 
To deal with this discrepancy, we introduce the group isomorphism $\inv \colon G \to G$ defined by $\inv(g) = g^{-1}$. We extend this to a ring isomorphism $\inv  \colon \Z[G] \to \Z[G]$ and use it to define a $\Z[G]$-module $\e^\triangle_*(\wt \tau)$ which as a $\Z$-module
 is equal to
 the face module $\e^\triangle(\tau)$ from \Cref{sec:taut}  and whose $\Z[G]$-module structure is determined by
\[
g \cdot x = \inv(g) x,
\]
where $g \in G$ and $x \in \e^\triangle_*(\tau)$. In words, we are simply replacing the action of an element by its inverse. We define $\e^\bigtriangledown_*(\wt \tau)$ similarly.

From the definitions, it is clear that the gcd of the elements of the Fitting ideal of $\e^\triangle_*(\wt \tau)$ (and of $\e^\bigtriangledown_*(\wt \tau)$ by the discussion above) is equal to $\inv(\Theta_\tau)$. In what follows, we will use the symmetry of the Teichm\"uller polynomial \cite[Corollary 4.3]{mcmullen2000polynomial} which states that $\Theta_{\bf F} = \inv(\Theta_{\bf F})$, up to a unit $\pm g \in \Z[G]$.

\begin{theorem} \label{th:modules_agree}
Let $\tau$ be a layered veering triangulation of $M$
representing a fibered face ${\bf F}$. 
The module of transversals $T(\widetilde{\mathcal{L}})$ is isomorphic as a $\Z[G]$-module to $\e^\bigtriangledown_*(\wt \tau)$ and hence to $\e^\triangle_*(\wt \tau)$. In particular, 
\[
\Theta_\tau = \Theta_{\bf F}
\]
up to a unit in $\Z[G]$.
\end{theorem}

Before the proof we need some additional preliminaries. 
Let $S$ be a fiber of $M$ representing a class in the 
cone over $\bf F$. Let $\psi \colon S \to S$ be the corresponding monodromy so that
$M$ can be recovered as the mapping torus of $\psi$:
\[
M = \frac{S \times [0,1]}{(s,1) \sim (\psi(s),0)}.
\]
We denote $\psi$'s expanding lamination by $\lambda$ so that the suspension of $\lambda$ is isotopic to $\mc{L}$. Either directly from Agol's construction \cite{agol2011ideal}, from \cite[Lemma 3.2]{minsky2017fibered}, or from \Cref{th:omni} above, the fiber $S$ is carried by $\tau$ up to isotopy. We fix any carrying map of $S$ into $\tau^{(2)}$. Then if we pull back edges and faces we obtain an ideal triangulation 
$\T$ of $S$ and a simplicial map $S \to M$. Such a map is called a \emph{section} in the terminology of \cite{minsky2017fibered}. The ideal triangulation $\T$ of $S$ is dual to a train track
$\mc V$ on $S$ which is obtained by pulling back the intersection of $S$ with the \emph{unstable} branched surface $B^u$. 
We observe that with our setup, the track $\mc V$ carries $\lambda$ and is an invariant track for 
$\psi$ in the sense that $\psi(\mc V)$ is carried by $\mc V$. 
This can be seen by noting that the sequence of upward diagonal exchanges from $\mc T$ to itself in $\tau$ corresponds to a sequence of diagonal exchanges from $\psi(T)$ to $T$ on $S$, and this sequence is dual to a \emph{folding} sequence from $\psi(\mc V)$ to $\mc V$ (see e.g. \Cref{fig:unstable_branched}.)

Next we follow a modified version of the discussion in \cite[Section 3]{mcmullen2000polynomial}.
An elevation $\wt S$ of $S$ to $\aM$ is the cover of $S$ corresponding to the kernel of the homomorphism $\pi_1(S) \to H_1(S) \to G$, whose image we denote by $H$. 
Fix a lift $\wt \psi \colon \wt S \to \wt S$ and, using that $M$ is the mapping torus of $\psi$, split $G = H \oplus \Z u$. Here, $u$ acts on $\aM= \wt S \times \RR$ as the deck transformation $\wt \Psi$ mapping $(s,t) \mapsto (\wt \psi (s), a(s,t))$, where $a(\cdot,\cdot)$ is real-valued function 
such that for each $s \in \wt S$ and $t \in \RR$, $a(\psi(s),t) \le a(s,t)$.
(Since the fixed section $S \to M$ may not be an embedding, we cannot necessarily make the conventional choice $a(s,t) = t-1$.)

Similarly, McMullen shows (\cite[Theorem 3.5]{mcmullen2000polynomial}) that if $\wt
\lambda$ denotes the lift of $\lambda$ to $\wt S$
and $T(\wt \lambda)$ is the $\Z[H]$-module of transversals for $\wt \lambda$, 
then $\wt{\mc L} = \wt \lambda \times \RR$ and $T(\wt{\mc L}) = T(\wt \lambda)$ as $\ZZ[H]$-modules. 
Then by considering the action of $u$ on $T(\wt \lambda)$, McMullen establishes that
\[
 T(\wt{\mc L}) \cong \frac{T(\wt \lambda) \otimes \Z[u]}{\mathrm{im}(uI - \wt \psi^*)},
\]
as $\Z[H][u] = \Z[G]$-modules. Here $\wt \psi^*$ is the action of $\wt \psi$ on $T(\wt \lambda)$ defined by taking preimages of transversals.

To connect the discussion back to the veering triangulation $\tau$, we recall a consequence of Gu\'eritaud's construction \cite{gueritaud} of $\tau$ in the layered setting.
We refer the reader to  \cite{minsky2017fibered} for additional details.
Fix a $\psi$-invariant quadratic differential $q$ on $S$ and let $\wt q$ be its lift to $\wt S$. Then there is a projection $\Pi \colon \aM \to \wt S$ that maps edges of $\wt \tau$ to saddle connections of $\wt q$. More precisely, $\Pi$ induces a bijection between edges of $\wt \tau$ and saddle connections of $\wt q$ that span singularity-free, immersed euclidean rectangles whose vertical/horizontal sides are segments of the vertical/horizontal foliations of $\wt q$.
The projection $\Pi$ is also equivariant in the sense that 
$\Pi(\wt \Psi (\wt e)) = \wt \psi (\Pi(\wt e))$. In particular, $\Pi(\wt e)$ can be naturally regarded as a transversal of $\wt \lambda$, which is the lamination associated to the horizontal foliation of $\wt q$. See \Cref{fig_expandingtt}.

\begin{figure}[h]
\centering
\includegraphics[]{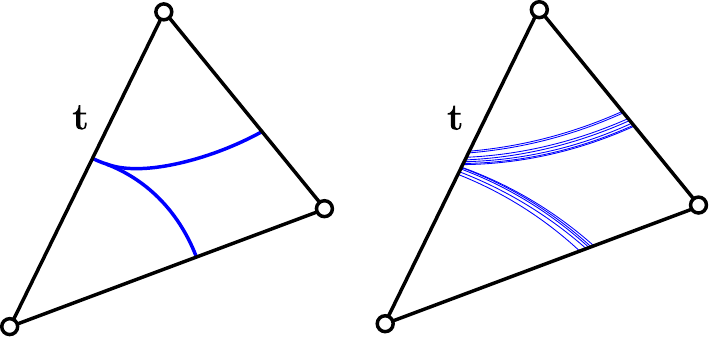}
\caption{A face $f$ with top edge $\bf{t}$. Each edge of $f$ is dual to a branch of the intersection $f \cap B^u$ (left) and is a transversal to the expanding lamination (right).}
\label{fig_expandingtt}
\end{figure}

We now turn to the proof of \Cref{th:modules_agree}.

\begin{proof}[Proof of \Cref{th:modules_agree}]
Let $\e_H^\bigtriangledown(\wt \tau)$ denote the $\Z[H]$-module obtained from the $\Z[G]$-module $\e^\bigtriangledown_*(\wt \tau)$ by restricting scalars. For each edge $\wt e$ of $\wt \tau$, $\Pi(\wt e)$ is a transversal of $\wt \lambda$ and this induces a homomorphism from the free module on edges of $\wt \tau$ to the free module on transversals. This is a $\Z[H]$-module homomorphism because the action on each is given by pullback under deck transformations. Moreover, this homomorphism descends 
to a $\Z[H]$-module homomorphism $\Pi_* \colon \e_H^\bigtriangledown(\wt \tau) \to T(\wt \lambda)$. Indeed, it is easily checked that relations among edges of $\wt \tau$ from \Cref{eq:top} are satisfied in $T(\wt \lambda)$, see \Cref{fig_expandingtt}.
In fact, we have:

\begin{claim}
$\Pi_* \colon \e_H^\bigtriangledown(\wt \tau) \to T(\wt \lambda)$ is an isomorphism of $\Z[H]$-modules.
\end{claim}
\begin{proof}[Proof of claim]

The preimage of $\mc V$ in $\wt S$ is a $\wt \psi$-invariant train track $\wt{\mc V}$ which is dual to the ideal triangulation $\wt{\mc T}$ of $\wt S$ induced by the simplicial map $\wt S \to \aM$. Mapping each branch of $\wt{\mc V}$ to its dual edge of $\wt{\mc T}$, regarded as an edge of $\wt \tau$, induces a $\Z[H]$-module homomorphism $\mc{D} \colon T(\wt{\mc V}) \to \e_H^\bigtriangledown(\wt \tau)$.
Here $T(\wt{\mc V})$ is the $\Z[H]$-module generated by the edges of $\wt{\mc V}$ modulo relations imposed by the switch conditions---these relations 
are mapped by $\mc D$ to the relations from \Cref{eq:top} used to define
 $\e_H^\bigtriangledown(\wt \tau)$. See the lefthand side of \Cref{fig_expandingtt}.

The composition $\Pi_* \circ \mc{D} \colon T(\wt{\mc V}) \to T(\wt \lambda)$ can alternatively be described as follows: If $\wt f \colon \wt \lambda \to \wt{\mc V}$ is the carrying map, then the branch $b$ is mapped to $f^{-1}(x)$, where $x$ is any point in the interior of $b$. Hence, \cite[Theorem 2.5]{mcmullen2000polynomial} 
proves that the composition $\Pi_* \circ \mc{D} \colon T(\wt{\mc V}) \to T(\wt \lambda)$ is an isomorphism of $\Z[H]$-modules. This immediately gives that $\Pi_*$ is surjective. Injectivity will follow from the fact that $\mc D$ is surjective.

To prove that $\mc D$ is surjective, we show the more general statement that for any ideal triangulation $\wt {\mc T}$ of $\wt S$ corresponding to a simplicial map $\wt S \to \aM$ lifting a section $S \to M$, the edges of $\wt{\mc T}$ generate $\e_H^\bigtriangledown(\wt \tau)$. This follows from two facts. First, if $\wt {\mc T} \to \wt{\mc T}'$ is a diagonal exchange corresponding to pushing through a single tetrahedron $t$ of $\wt \tau$, then the new edge of $\wt{\mc T'}$ is a sum or difference of two edges of $\wt{\mc T}$ that belong to $t$ (depending on whether the tetrahedron lies above or below $\wt{\mc T}$). Second, for any edge $\wt e$ of $\wt \tau$ there is triangulation $\wt{\mc T'}$ of $\wt S$ containing $\wt e$ and a finite sequence $\wt{\mc T} = \wt{\mc T}_1, \ldots ,\wt{\mc T}_n = \wt{\mc T'}$ so that each $\wt{\mc T_i} \to \wt{\mc T}_{i+1}$ corresponds to simultaneously pushing through an $H$-orbit of tetrahedra of $\wt \tau$ \cite[Lemma 3.2 and Proposition 3.3]{minsky2017fibered}. These two facts show that every edge of $\wt \tau$ is a linear combination of edges of $\wt {\mc T}$, so $\mc D$ is surjective as desired.
\end{proof}

Returning to the proof of the theorem, we note that since $\Pi(\wt \Psi^{-1} (\wt e)) = \wt \psi^* (\Pi(\wt e))$, the isomorphism $\Pi_*$ takes the action of $\wt \Psi^{-1}$ on $\e_H^\bigtriangledown(\wt \tau)$ to the action of $\wt \psi^*$ on $T(\lambda)$. Hence, it extends to an isomorphism of $\Z[H] \otimes \Z[u] = \Z[G]$-modules:
\[
\frac{\e_H^\bigtriangledown(\wt \tau) \otimes \Z[u]}{\mathrm{im}(uI - \wt \Psi^{-1})} \longrightarrow \frac{T(\lambda) \otimes \Z[u]}{\mathrm{im}(uI - \wt \psi^*)},
\]
where the second module is $T(\wt{\mc L})$, as explained above. Since the first module simply reintroduces the action of $u$, it is isomorphic to $\e^\bigtriangledown_*(\wt \tau)$. Hence, $\e^\bigtriangledown_*(\wt \tau)$ is isomorphic to $T(\wt{\mc L})$ as a $\Z[G]$-module. 

Using the remarks preceding the theorem, we also have that $\e^\triangle_*(\wt \tau)$ is isomorphic to $T(\wt{\mc L})$ as a $\Z[G]$-module and we
conclude that $\inv(\Theta_\tau) = \Theta_{\bf F} = \inv(\Theta_{\bf F})$ and so $\Theta_\tau = \Theta_{\bf F}$, up to a unit.
This completes the proof of \Cref{th:modules_agree}.
\end{proof}

\subsection{The general case via puncturing}

On its face \Cref{th:modules_agree} applies only when starting with a fibered face associated to a veering triangulation.
However, the  following proposition shows that the constructions are compatible with puncturing along orbits of the associated suspension flow.

Let $N$ be any hyperbolic $3$-manifold with fibered face ${\bf F}_N$.
There is a flow $\phi$ such that each fibration 
associated to $\R_+{\bf F}_N$ may be isotoped so 
that the first return map to a fiber is pseudo-Anosov and that, up to 
reparametrization, $\phi$ is the associated suspension flow
\cite[Theorem 14.11]{fried1979fibrations}. 
The orbits of the pseudo-Anosov's singularities are called the 
singular orbits of $\phi$.

Now let $M$ be the result of puncturing $N$ along the singular orbits of $\phi$ and let $\tau$ be the associated veering triangulation of $M$ (see \Cref{sec:normbkgd}). 
Let $i \colon M \hookrightarrow N$ be the inclusion map. Then $i^* \colon H^1(N) \to H^1(M)$ maps the cone over ${\bf F}_N$ into the cone over a fibered face ${\bf F}_M := {\bf F}_\tau$ of $M$.

The first statement of the next proposition is observed by McMullen to prove Equation 6.1 in \cite{mcmullen2000polynomial}. We provide some details using results from \cite[Section 4]{mcmullen2000polynomial}.

\begin{proposition} \label{prop:Teich_punctured}
With the setup as above, $\Theta_{{\bf F}_N} = i_*(\Theta_{{\bf F}_M})$ up to a unit. Hence, if $\tau$ is the layered veering triangulation on $M$ associated to the fibered face ${\bf F}_M$, then 
\[
\Theta_{{\bf F}_N} = i_*(\Theta_\tau),
\]
up to a unit $\pm g \in \Z[G]$.
\end{proposition}

\begin{proof}
The second claim follows immediately from the first and \Cref{th:modules_agree}.

We recall that if $\Theta $ is the Teichm\"uller polynomial associated to a fibered face $\bf F$, then there is a unique $g \in \supp(\Theta)$ such that  for all $\alpha$ in the interior of $\R_+\bf{F}$
\[
\alpha(g) > \alpha(h) \text{ for all } h \in \supp(\Theta) \ssm \{g\}.
\]
This follows from the proof of \cite[Theorem 6.1]{mcmullen2000polynomial}; see also item 5 in the subsection `Information packaged in $\Theta_{\bf F}$' of \cite[Section 1]{mcmullen2000polynomial}. Hence, we can normalize $\Theta$ by multiplying by the unit $\pm g$ so that for any $\alpha$ as above, $\Theta(u^\alpha)$ has a positive constant term
and all other terms have negative exponent (here we are thinking of $\alpha$ as a first cohomology class). We perform this normalization on both
$\Theta_{{\bf F}_N}$ and $\Theta_{{\bf F}_M}$.

As in the proof of \cite[Corollary 4.3]{mcmullen2000polynomial}, to show that $\Theta_{{\bf F}_N} = i_*(\Theta_{{\bf F}_M})$ it suffices to show that for each integral $\alpha = [S] \in \RR_+{\bf F}_N$, there is the equality of specializations $\Theta_{{\bf F}_N}(u^\alpha) = i_*(\Theta_{{\bf F}_M})(u^\alpha)$.
This is because $\RR_+{\bf F}_N$ is open and so one can find such an $\alpha$ so that the values $\alpha(g)$ are 
distinct for all $g \in \supp (\Theta_{{\bf F}_N}) \bigcup \supp( i_*(\Theta_{{\bf F}_M}))$.

Next we note that  $i_*(\Theta_{{\bf F}_M})(u^\alpha) = \Theta_{{\bf F}_M}(u^{i^*\alpha})$. The pullback $i^*\alpha$ is dual to the class $[\mr S]$, where $\mr S$ is obtained from $S$ by puncturing at the singularities of its monodromy $\psi$. Hence, by \cite[Theorem 4.2]{mcmullen2000polynomial}, up to units $\pm u^k$,
$\Theta_{{\bf F}_N}(u^\alpha)$ is equal to the characteristic polynomial of $\psi$ and 
$\Theta_{{\bf F}_M}(u^{i^*\alpha})$ is equal to the characteristic polynomial of $\mr \psi$, the induced monodromy of $\mr S$ (unless the $2$-dimensional lamination $\mc L$ is orientable, in which case we multiply first by $(u-1)$ in both cases). However, the characteristic polynomial from \cite{mcmullen2000polynomial} is defined solely in terms of $\psi$'s action on its expanding lamination $\lambda$, and since we have punctured $S$ at an $\psi$-invariant set in the complement of $\lambda$, $\lambda$ is also the expanding lamination of $\mr \psi$. Hence, the maps $\psi \colon \lambda \to \lambda$ and $\mr \psi \colon \lambda \to \lambda$ are equal. (This can also be seen by considering the intersection of $\mc L$ with $S$ and $\mr S$ in $N$.) We conclude that their characteristic polynomials are equal and therefore that 
$\Theta_{{\bf F}_N}(u^\alpha) = \Theta_{{\bf F}_M}(u^{i^*\alpha})$
 up to a unit $\pm u^k$. However, it is clear from our normalizations that this unit must be the identity,
and the proof is complete.
\end{proof}

\subsection{Computing $\Theta_{\bf F}$ via $V_\tau$}
Here we state an analog of McMullen's determinant formula which follows immediately from \Cref{th:modules_agree} and \Cref{th:factorization!}. By further applying \Cref{prop:Teich_punctured}, we obtain a general method to compute the Teichm\"uller polynomial from the veering triangulation on the associated fully punctured manifold.

\begin{corollary} \label{cor:veering_teich}
Let $(M,\tau)$ be a layered veering triangulation associated to a fibered face $\bf F$ and suppose that $\mathrm{rank}(H_1(M)) >1$. Then 
\[
\Theta_{\bf F}  = \frac{V_\tau}{V^{AB}_\tau} = \frac{\det L}{\det L^{AB}}
\]
where $L$ is as in \Cref{sec:veering} and $L^{AB}$ is as in \Cref{sec:polys}.
\end{corollary}

Note that in this case (i.e. when $\tau$ is layered), $V^{AB}_\tau \neq 0$ since no $AB$-cycle (nor any dual cycle) can be trivial in $H_1$ by \Cref{th:omni}.

\begin{remark}
\Cref{cor:veering_teich} indicates that a slight modification is needed for McMullen's determinant formula \cite[Theorem 3.6]{mcmullen2000polynomial}. For this we first define, given a veering triangulation $\tau$, its negative $-\tau$ to be the veering triangulation obtained by reversing the coorientation on faces. With this definition, we have $\e^\bigtriangledown(\wt \tau) = \e^\triangle(\wt {-\tau}) $.

Then one can directly show that if the train track $\mc V$ from above
is used in McMullen's construction,
then $V_{-\tau} = \det(uI -P_E)$, up to a unit in $\Z[G]$. 
(Here, $P_E$ and $P_V$ are the matrices with entries in $\Z[G]$ that represent the action of the lifted monodromy on the branches and switches of the lifted track $\wt{\mc V}$, respectively.)
Since the determinant formula states that 
 \[
 \Theta_{\bf F} = \frac{\det(uI -P_E)}{ \det(uI -P_V)},
 \]
 one would then expect that $V^{AB}_{-\tau} =   \det(uI -P_V)$, up to a unit, but this need not be the case. 
The issue is that the definition of $P_V$ implied in McMullen's paper needs to be correctly interpreted to account for switches being mapped to themselves with their sides reversed 
(by making certain entries negative).  
In short, for a switch $v$ of $\mc V$, McMullen's Equation $3.4$ only commutes up to sign.
From our point of view this occurs in the presence of $AB$-cycles of odd length (see 
\Cref{lem:computing_correction}). 
However, the addition of appropriate minus signs in $P_V$ would make the above equality true and correct the general determinant formula.
\end{remark}

\bibliography{veering_poly1.bbl}

\bibliographystyle{amsalpha}

\end{document}